\documentclass[12pt,reqno,a4paper]{amsart}
\usepackage{hyperref} 
\usepackage{fullpage}%[cm]
\usepackage[utf8]{inputenc}
\usepackage[T1]{fontenc}
\usepackage{enumitem}
\usepackage{amsmath,amssymb,amsthm,color}
\usepackage{todonotes,framed} %added
\usepackage{color}
\usepackage{changes}
\usepackage{siunitx}
\usepackage{moreverb}
\usepackage{subcaption}
\usepackage{graphicx}
\usepackage{boxedminipage}
\usepackage{nameref}
\usepackage{mdframed}

\usepackage{pgfplots}
\pgfplotsset{compat=1.9}
\usepackage{pgfplotstable}
\def\norm#1#2{\|#1\|_{#2}}
\def\matrixnorm#1{\|#1\|}
\def\set#1#2{\big\{#1\,:\,#2\big\}}
\def\dual#1#2#3{{\langle#1\,,\,#2\rangle}_{#3}}
\def\energydual#1#2#3{{\langle\!\langle\!\langle#1\,,\,#2\rangle\!\rangle\!\rangle}_{#3}}

\def\N{\mathbb{N}}
\def\R{\mathbb{R}}

\def\SS{\mathcal S}
\def\TT{\mathcal T}
\def\NN{\mathcal N}
\def\MM{\mathbf M}

\def\ff{\boldsymbol{f}}
\def\mm{\boldsymbol{m}}
\def\nn{\boldsymbol{n}}
\def\xx{\boldsymbol{x}}
\def\vv{\boldsymbol{v}}
\def\ww{\boldsymbol{w}}
\def\qq{\boldsymbol{q}}
\def\uu{\boldsymbol{u}}
\def\HH{\boldsymbol{H}}
\def\LL{\boldsymbol{L}}

\def\heff{\boldsymbol{h}_{\textrm{\textup{eff}}}}

\def\Cmesh{C_{\textbf{\textup{mesh}}}}
\def\ellex{\ell_{\textrm{\textup{ex}}}}

\def\pphi{\boldsymbol{\phi}}
\def\ppsi{\boldsymbol{\psi}}

\def\mmu{\boldsymbol{\mu}}
\def\nnu{\boldsymbol{\nu}}
\def\ppi{\boldsymbol{\pi}}
\def\Ppi{\boldsymbol{\Pi}}

\def\vvarphi{\boldsymbol{\varphi}}

\def\matrixA{{\bf A}}
\def\matrixB{{\bf B}}

\def\matrixH{{\bf H}}
\def\matrixI{{\bf I}}

\def\matrixL{{\bf L}}
\def\matrixM{{\bf M}}

\def\matrixP{{\bf P}}
\def\matrixR{{\bf R}}
\def\matrixQ{{\bf Q}}
\def\matrixS{{\bf S}}
\def\matrixT{{\bf T}}

\def\vectorB{{\bf b}}
\def\vectorE{{\bf e}}
\def\vectorM{{\bf m}}
\def\vectorR{{\bf r}}

\def\vectorV{{\bf v}}
\def\vectorW{{\bf w}}
\def\vectorX{{\bf x}}
\def\vectorY{{\bf y}}

\def\0{\boldsymbol{0}}

\def\sign{{\rm sign}}

\def\courantspace{\boldsymbol{\mathcal{S}}}
\def\magnetizationset{\boldsymbol{\mathcal{M}}}
\def\tangentspace{\boldsymbol{\mathcal{K}}_h[\mhn]}
\def\tps#1{\boldsymbol{\mathcal{K}}_h[#1]}

\def\mhn{\mm_h^n}
\def\vh{\vv_h}

\def\weight{W}

\def\Cmh0{{\tilde{C}}_{{0}}}
\newcounter{statement}
\newenvironment{statement}[2][!]{%
\vskip3mm
\hrule
\hrule
\hrule
\vskip1mm
\noindent%
\refstepcounter{statement}%
\bf#2~\thestatement%
\ifthenelse{\equal{#1}{!}}{.\ }{~(#1).\ }%
\it%
}{%
\vskip1mm
\hrule
\hrule
\hrule
\vskip2mm
}

\newenvironment{theorem}[1][!]{\begin{statement}[#1]{Theorem}}{\end{statement}}
\newenvironment{lemma}[1][!]{\begin{statement}[#1]{Lemma}}{\end{statement}}

\newenvironment{proposition}[1][!]{\begin{statement}[#1]{Proposition}}{\end{statement}}
\newenvironment{corollary}[1][!]{\begin{statement}[#1]{Corollary}}{\end{statement}}
\newenvironment{remark}[1][!]{\begin{statement}[#1]{Remark}}{\end{statement}}
\newenvironment{algorithm}[1][!]{\begin{statement}[#1]{Algorithm}}{\end{statement}}

\mdfdefinestyle{theoremstyle}{% linecolor=red,linewidth=2pt,% frametitlerule=true,% frametitlebackgroundcolor=gray!20, innertopmargin=\topskip,
} 
\def\subsection#1{%
 \vskip2mm
 \refstepcounter{subsection}%
 {\bf\arabic{section}.\arabic{subsection}.~#1.~}%
}

\usepackage{fancyhdr}
\advance\footskip0.4cm
\advance\textheight-0.4cm
\calclayout
\author{Johannes~Kraus}
\author{Carl-Martin~Pfeiler}
\author{Dirk~Praetorius}
\author{Michele~Ruggeri}
\author{Bernhard~Stiftner}
\address{Faculty of Mathematics, University of Duisburg-Essen, Thea-Leymann-Strasse 9, 45127 Essen, Germany}
\email{johannes.kraus@uni-due.de}
\address{Institute for Analysis and Scientific Computing, TU Wien, Wiedner Hauptstrasse 8-10, 1040 Vienna, Austria}
\email{carl-martin.pfeiler@asc.tuwien.ac.at}
\email{dirk.praetorius@asc.tuwien.ac.at}
\email{bernhard.stiftner@asc.tuwien.ac.at}
\address{Faculty of mathematics, University of Vienna, Oskar-Morgenstern-Platz 1, 1090 Vienna, Austria}
\email{michele.ruggeri@univie.ac.at}
\thanks{\emph{Acknowledgements.}
The authors acknowledge support from the Vienna Science and Technology Fund WWTF (grant MA14-44) and the Austrian Science Fund FWF (grants DK W1245 and SFB F65).}
\title{\vspace*{-15mm}Iterative solution and preconditioning \\ for the tangent plane scheme \\ in computational micromagnetics}
\begin{document}
%%%%%%%%%%%%%%%%%%%%
\begin{abstract}
\vspace*{-1mm}
The tangent plane scheme is a time-marching scheme for the numerical solution of the nonlinear parabolic Landau--Lifshitz--Gilbert equation (LLG), which describes the time evolution of ferromagnetic configurations. Exploiting the geometric structure of LLG, the tangent plane scheme requires only the solution of one linear variational form per time-step, which is posed in the discrete tangent space determined by the nodal values of the current magnetization. We develop an effective solution strategy for the arising constrained linear systems, which is based on appropriate Householder reflections. We derive possible preconditioners, which are (essentially) independent of the time-step, and prove that the preconditioned GMRES algorithm leads to linear convergence. Numerical experiments underpin the theoretical findings.
\end{abstract}
%%%%%%%%%%%%%%%%%%%%
\maketitle
\thispagestyle{fancy}
%%%%%%%%%%%%%%%%%%%%
\vspace*{-8mm}

%%%%%%%%%%%%%%%%%%%%%%%%%%%%%%%%%%%%%%%%%%%%%%%%%%%%%%%%%%%%%%%%%%%%
\section{Introduction}
%%%%%%%%%%%%%%%%%%%%%%%%%%%%%%%%%%%%%%%%%%%%%%%%%%%%%%%%%%%%%%%%%%%%

%-------------------------------------------------------------------
\subsection{Landau--Lifshitz--Gilbert equation}
%-------------------------------------------------------------------
The Landau--Lifshitz--Gilbert equation (LLG) describes time-dependent micromagnetic phenomena in a ferromagnetic domain $\Omega \subset \R^3$ being bounded and Lipschitz with boundary $\partial \Omega$~\cite{Gilbert55,LL08}.
It reads
\begin{subequations}\label{eq:LLG}
\begin{align}
\partial_t \mm &= - \mm \times \heff(\mm) + \alpha\, \mm \times \partial_t \mm && \textrm{in } (0,T) \times \Omega, \label{eq:LLG1}\\
\partial_{\nn} \mm & = \0 && \textrm{on } (0,T) \times \partial \Omega, \label{eq:LLG2}\\
\mm(0) & = \mm^0 && \textrm{in } \Omega,
\end{align}
\end{subequations}
where the unknown $\mm: (0,T) \times \Omega \rightarrow \R^3$ is the magnetization, $\heff(\mm)$ is the $\mm$-dependent effective field, $\alpha \in (0,1]$ is the Gilbert damping constant, $T>0$ is the final time, and $\mm^0 \in \HH^1(\Omega) := (H^1(\Omega))^3$ with $|\mm^0| = 1$ a.e.\ in $\Omega$ is the initial configuration.
The effective field comprises several contributions, which correspond to different phenomena in micromagnetism.
In usual applications (cf., e.g.,~\cite{HS98}), one has $\heff(\mm) := \ellex^2 \Delta \mm  + \ppi(\mm) + \ff$, where
%\begin{itemize}
%\item 
$\ellex^2 \Delta \mm$ is the exchange contribution with the exchange length $\ellex > 0$,
%\item 
$\ppi: \HH^1(\Omega) \rightarrow \LL^2(\Omega) :=  (L^2(\Omega))^3$ is a bounded operator, which collects all $\mm$-dependent lower-order terms such as the stray field or the magnetocrystalline anisotropy contribution,
%\item 
$\ff \in C^1([0,T],\LL^2(\Omega))$ is the applied external field.
%\end{itemize}

Taking the scalar product with $\mm$ in~\eqref{eq:LLG1}, we note the PDE inherent constraints
\begin{align}\label{eq:orthogonal}
\frac12 \, \partial_t |\mm|^2 = \mm \cdot \partial_t \mm = 0 \quad \textrm{and thus} \quad
|\mm|=1 \quad \textrm{a.e. in } (0,T) \times \Omega.
\end{align}
In particular, $\partial_t \mm(t)$ belongs to the tangent space of $\mm(t)$ for almost all $t \in (0,T)$. Using vector identities and~\eqref{eq:orthogonal}, one can prove that~\eqref{eq:LLG1} is equivalent to
\begin{align}\label{eq:alternativeform}
\alpha\, \partial_t \mm + \mm \times \partial_t \mm = \heff(\mm) - (\heff(\mm) \cdot \mm) \, \mm,
\end{align}
which is a linear equation in $\vv:= \partial_t \mm$. 

%-------------------------------------------------------------------
\subsection{Tangent plane scheme (TPS)}
%-------------------------------------------------------------------
The idea of TPS~\cite{AJ06} can roughly be described as follows: At time $t_n$, the magnetization $\mm(t_n)$ is discretized by some lowest-order finite element approximation $\mm_h^n$ in space. Discretizing~\eqref{eq:alternativeform} by a Galerkin approach in the discrete tangent space at $\mm_h^n$, we obtain an approximation $\vv_h^n \approx \vv(t_n)$. Up to nodal normalization, $\mm_h^n + k \vv_h^n$ then yields an approximation of the magnetization at time $t_{n+1} := t_n + k$. Although LLG is nonlinear, TPS thus solves only one linear system per time-step for $\vv_h^n$, yet, in the discrete tangent space.

%-------------------------------------------------------------------
\subsection{State of the art}
%-------------------------------------------------------------------
TPS with explicit time-stepping was first analyzed in~\cite{AJ06} with a refined analysis in~\cite{BKP08}, which requires a CFL condition for convergence towards a weak solution in the sense of~\cite{AS92}. The work~\cite{Alouges08} proposed TPS with an implicit time-stepping. This yields unconditional convergence of the algorithm towards a weak solution. While the algorithm of~\cite{Alouges08} is formulated for the exchange field only, it was extended to general stationary lower-order contributions in~\cite{AKT12,BSF+14} and chiral magnetic skyrmion dynamics in~\cite{HPP+18}.
Moreover, TPS was extended to the coupling of LLG with other evolution equations such as the full Maxwell system~\cite{BPP15}, the eddy current equation~\cite{LT13,LPP+15,FT17}, the conservation of elastic momentum~\cite{BPP+14} to model magnetostrictive effects, or a spin diffusion equation~\cite{AHP+14,ARB+15}. In the mentioned works, TPS is formally of first order in time. Recently, TPS was modified into a (formally) second-order in time scheme in~\cite{AKS+14} with extensions in~\cite{DPP+17,HPP+18}. 

%-------------------------------------------------------------------
\subsection{Contributions}
%-------------------------------------------------------------------
So far, the efficient solution of the linear system in the discrete tangent space for the computation of $\vv_h^n \approx \vv(t_n):\Omega \rightarrow \R^3$ has not been discussed in the literature, yet. Here, the main difficulty is the time-dependent ansatz space resulting in a time-dependent system matrix. This also aggravates the construction of 
suitable and 
effective preconditioners, which, if possible, should not depend on the time-step, or, at least, only need an update every once in a while (after several time-steps).

We construct a linear system in $\R^{2N}$, where $N \in \N$ is the number of nodes of the underlying finite element discretization. The corresponding system matrix is positive definite, but non-symmetric and depends on the time-step.
We present and analyze various preconditioners, including a stationary approach (i.e., independent of the time-step) as well as Jacobi-type approximations. In the worst case, the number of necessary updates of the preconditioner to attain optimal convergence of the GMRES algorithm~\cite{SS86,Saad03} depends on the mesh-size $h$. However, under certain assumptions on the discrete magnetization $\mm_h^n \approx \mm(t_n)$, the number of necessary updates 
%for our preconditioners 
is also independent of $h$.

%-------------------------------------------------------------------
\subsection{Outline}
%-------------------------------------------------------------------
This paper is organized as follows: Section~\ref{section:preliminaries} introduces the basic notation and gives a precise formulation of 
%the tangent plane scheme 
TPS (Algorithm~\ref{alg:abtps}). In Section~\ref{section:tangentspace}, we provide a basis for the discrete tangent space and derive the prototype linear system, which has to be solved in each time-step (Theorem~\ref{theorem:howtosolve}). Section~\ref{section:precond} proposes symmetric and positive definite preconditioners for the latter linear system. The two main results (Theorem~\ref{theorem:theoretical_convergence} and Theorem~\ref{theorem:practical_convergence}) prove that the corresponding GMRES algorithms converge linearly. These results also provide estimates of the corresponding residual reduction factors (in certain energy norms) and show under which assumptions these estimates are independent of the discretization parameters. A corresponding linear convergence result for a stationary preconditioning approach (Corollary~\ref{corollary:stationaryprecond}) is a by-product of Theorem~\ref{theorem:theoretical_convergence}. Finally, we also discuss Jacobi-type approximations of our preconditioners (Section~\ref{subsection:jacobi}). Our theoretical results are underpinned by numerical experiments in Section~\ref{section:numerics}. The proofs of Theorem~\ref{theorem:howtosolve}, Theorem~\ref{theorem:theoretical_convergence}, and Theorem~\ref{theorem:practical_convergence} are postponed to Section~\ref{section:proofs}.

%%%%%%%%%%%%%%%%%%%%%%%%%%%%%%%%%%%%%%%%%%%%%%%%%%%%%%%%%%%%%%%%%%%%
\section{Preliminaries}\label{section:preliminaries}
%%%%%%%%%%%%%%%%%%%%%%%%%%%%%%%%%%%%%%%%%%%%%%%%%%%%%%%%%%%%%%%%%%%%

%-------------------------------------------------------------------
\subsection{General notation}
%-------------------------------------------------------------------
For any dimension $d \in \N$ (clear from the context) and vectors $\vectorX,\vectorY\in\R^d$, let $\vectorX \cdot \vectorY$ denote the Euclidean scalar product with the corresponding norm $|\vectorX|^2:=\vectorX\cdot\vectorX$. The induced matrix norm reads $\| \matrixA \| := \sup_{x \in \R^{d} \setminus \{ \0 \}} |\matrixA \vectorX| / |\vectorX|$. Moreover, we denote by $\vectorE_i$ the $i$-th unit vector and by $\matrixI$ the identity matrix in $\R^d$. To abbreviate notation, we follow the {\tt Matlab} syntax: For vectors, $\vectorX_1, \dots , \vectorX_n \in \R^{d}$, we write $[ \vectorX_1 , \dots, \vectorX_n ] \in \R^{d \times n}$ for the matrix whose $j$-th column is $\vectorX_j$. We use bold letters for vector-valued spaces, e.g., $\LL^2(\Omega)=(L^2(\Omega))^3$. By slight abuse of notation, we write $\norm{\cdot}{\LL^{2}(\Omega)}$ simultaneously for the $L^2$-norm on $(L^2(\Omega))^3$ and $(L^2(\Omega))^{3 \times 3}$.
Similarly, we write $\dual{\cdot}{\cdot}{\Omega}$ for all $L^2$-scalar products including vector-valued spaces. We write $\partial_k$ for the derivative with respect to $\xx_k$, where $k \in \{1,2,3\}$. We write $C$ for generic constants (clear from the context and, in particular, independent of the discretization parameters). For $a,b \in \R^+$ with $a \leq C b$, we write $a \lesssim b$. If $a \lesssim b$ and $b \lesssim a$, we write $a \simeq b$.

%-------------------------------------------------------------------
\subsection{Discretization}
%-------------------------------------------------------------------
For the temporal discretization of LLG, let $M \in \N$ and $k := T/M$.
Let $t_n := kn$ with $n \in \{0,\dots,M\}$ be the uniform time-steps and let $t_{n+1/2} := (t_{n+1} + t_n) / 2$.
For the spatial discretization, let $\TT_h$ be a $\Cmesh$-quasi-uniform triangulation of $\Omega$ into tetrahedra $K \in \TT_h$ with mesh-size $h>0$, i.e., there exists $\Cmesh>0$ such that
\begin{align*}
h \, \leq  \, |K|^{1 / 3}  \, \leq  \, \operatorname{diam}(K) \leq \Cmesh \, h \quad \textrm{for all } K \in \TT_h,
\end{align*}
where $|K|$ denotes the volume of the element $K$, while $\operatorname{diam}(K)$ is its diameter.
Let
\begin{align}
 \courantspace_h := (\SS_h)^3
 \quad \textrm{with} \quad
 \SS_h:=\set{v_h\in C(\Omega)}{\textrm{for all } T\in\TT_h\quad v_h|_T\text{ is affine}}
% \quad\text{and}\quad
% \magnetizationset_h := \set{\wh\in\courantspace_h}{\textrm{for all } z\in\NN_h\quad|\wh(z)|=1}
\end{align}
be the lowest-order FEM space. We denote by $\NN_h$ the set of nodes of $\TT_h$ and define $N := \# \NN_h$. For all $d \in \N$, a scaling argument (see, e.g.,~\cite[Lemma~3.9]{Bartels2015}) yields that 
\begin{align}\label{eq:scalingargument}
\norm{\vvarphi_h}{\LL^2(\Omega)}^2 \leq  h^3 \, \sum_{i=1}^N | \vvarphi_h(z_i) |^2  \leq (d+2) \, \norm{\vvarphi_h}{\LL^2(\Omega)}^2
\quad \textrm{for all } \vvarphi_h \in (\SS_h)^d.
\end{align}
Moreover, let
\begin{align}
\magnetizationset_h := \set{\vvarphi_h\in\courantspace_h}{|\vvarphi_h(z)|=1 \quad \textrm{for all } z\in\NN_h}. 
\end{align}
For some fixed $\mmu_h\in\magnetizationset_h$, define the discrete tangent space
\begin{align}\label{eq:Khn}
 \tps{\mmu_h} := \set{\vvarphi_h\in\courantspace_h}{\vvarphi_h(z)\cdot\mmu_h(z)=0\quad \textrm{for all } z\in\NN_h}.
\end{align}
Note that $\dim \courantspace_h = 3N$ and $\dim \tps{\mmu_h} = 2N$.

%-------------------------------------------------------------------
\subsection{Tangent plane scheme (TPS)}
%-------------------------------------------------------------------
We unify the formulation of the first-order TPS~\cite{Alouges08,AKT12,BSF+14} and the second-order TPS~\cite{AKS+14,DPP+17}. To this end, let $\weight_k: \R \rightarrow \R_{>0}$ and $\beta: \R_{>0} \rightarrow \R_{>0}$ satisfy
\begin{align}\label{eq:stabilizations}
\lim\limits_{k \rightarrow 0} \norm{ \weight_k- \alpha }{L^{\infty}(\R)} = 0,
\quad 
\inf\limits_{k > 0} \inf\limits_{s \in \R} W_k(s) \geq \frac{\alpha}{2} > 0,
\quad \textrm{and} \quad
\lim_{k \rightarrow 0} \beta(k) k = 0.
\end{align}
Then, the general algorithm takes the following form:
%\pagebreak

\begin{algorithm}[General TPS]\label{alg:abtps}
{\bf Input:} $\mm_h^{-1} := \mm_h^0\in\magnetizationset_h$. \\
{\bf Loop:}  For $0 \leq n < M$, iterate the following steps {\rm (a)}--{\rm (c)}: \\
{\rm(a)} Compute 
\begin{subequations}
\begin{align}\label{eq:weight_input}
\lambda_h^{n} := - \ellex^2 \, |\nabla \mm_h^n |^2 + \big( \ff(t_n) + \ppi(\mm_h^n) \big) \cdot \mm_h^n.
\end{align}
{\rm(b)} Find $\vv_h^n \in \tangentspace$ such that
\begin{align}\label{eq:tps2_variational}
\begin{split}
&\dual{\weight_k(\lambda_h^n) \, \vv_h^n}{\vvarphi_h}{\Omega} + \dual{\mm_h^n\times\vv_h^n}{\vvarphi_h}{\Omega} + 
\beta(k) k \, \dual{\nabla\vv_h^n}{\nabla\vvarphi_h}{\Omega} 
\\
& \qquad = - \ellex^2 \dual{\nabla \mm_h^n}{\nabla \vvarphi_h}{\Omega} + 
\dual{\boldsymbol{L}_h(\mm_h^n,\mm_h^{n-1})}{\vvarphi_h}{\Omega}
\quad \textrm{for all } \vvarphi_h \in \tangentspace;
\end{split}
\end{align} see Remark~\ref{remark:aboudtps2}~{\rm (iii)} for details on $\weight_k$, $\beta$, and $\boldsymbol{L}_h(\mm_h^n,\mm_h^{n-1}) \in \HH^1(\Omega)$. \\
{\rm(c)}  Define $\mm_h^{n+1} \in \magnetizationset_h$ by
\begin{align}\label{eq:makestep}
\mm_h^{n+1}(z) := \frac{\mm_h^{n}(z) + k \vv_h^{n}(z)}{|\mm_h^{n}(z) + k \vv_h^{n}(z)|} \quad \textrm{for all nodes } z\in\NN_h.
\end{align}
\end{subequations}
{\bf Output:} Approximations $\mm_h^n\approx\mm(t_n)$. \qed
\end{algorithm}

\begin{remark}\label{remark:aboudtps2}
{\rm (i)} With $\weight_k(\cdot) \geq \alpha/2 > 0$ from~\eqref{eq:stabilizations}, the bilinear form on the left-hand side of~\eqref{eq:tps2_variational} is continuous and elliptic on $\HH^1(\Omega)$. Therefore, the Lax--Milgram theorem guarantees existence and uniqueness of the solution $\vv_h^n \in \tangentspace$ to~\eqref{eq:tps2_variational}.

{\rm (ii)} To see that~\eqref{eq:makestep} is well-defined, note that $\mm_h^0 \in \magnetizationset_h$ and induction on $n$ prove that
\begin{align*}
|\mm_h^n(z) + k \vv_h^n(z) |^2 \stackrel{\eqref{eq:Khn}}{=} |\mm_h^n(z)|^2 + k^2 |\vv_h^n(z)|^2 \geq 1 \quad \textrm{for all nodes } z \in \NN_h .
\end{align*}

{\rm (iii)} The first-order (TPS1) and second-order (TPS2) tangent plane schemes differ solely in $\weight_k$, $\beta$ and $\boldsymbol{L}_h(\mm_h^n,\mm_h^{n-1})$ in the definition of~\eqref{eq:tps2_variational}. For TPS1 from~\cite{Alouges08,AKT12,BSF+14}, we have
\begin{subequations}\label{eq:TPS1}
\begin{align}
\weight_k & \, \equiv \alpha, \\
\beta(k) & \, \equiv \ellex^2 \, \Theta \quad \textrm{with} \quad \Theta \in (0,1], \label{eq:TPS1_beta}\\
\boldsymbol{L}_h(\mm_h^n,\mm_h^{n-1}) & := \ppi(\mm_h^n) + \ff(t_n),
\end{align} 
\end{subequations}
and step {\rm (a)} in Algorithm~\ref{alg:abtps} is omitted. For TPS2 from~\cite{AKS+14,DPP+17}, let
\begin{subequations}\label{eq:TPS2}
\begin{align}
\rho(k) := | k \log(k) | \quad \textrm{and} \quad M(k) := | k \log(k) |^{-1};
\end{align}
see, e.g.,~\cite[eq.(11c)]{DPP+17} for the precise requirements. Then, we have
\begin{align}
\weight_k(s) &:= 
\begin{cases}
\displaystyle\alpha + \frac{k}{2} \, \min\{s,M(k)\} \quad&\text{for } s\ge 0 \,,\\
\displaystyle\frac{\alpha}{1+\frac{k}{2\alpha}\min\{-s,M(k)\}} \quad&\text{for } s<0 ,
\end{cases} \\
\beta(k) & := \frac{\ellex^2}{2} \, (1 + \rho(k)), \\
\boldsymbol{L}_h(\mm_h^n,\mm_h^{n-1}) &:= \frac32 \, \ppi(\mm_h^n) - \frac12 \, \ppi(\mm_h^{n-1}) + \ff\big( \, t_n + \frac{k}{2} \, \big).
\end{align}
\end{subequations}
Note that, for sufficiently small $k>0$, the latter choices of $\weight_k$ and $\rho$ satisfy the assumptions from~\eqref{eq:stabilizations}; see~\cite{AKS+14,DPP+17} for details.

{\rm (iv)} For TPS1 from {\rm(iii)}, the projection step~{\rm (c)} of Algorithm~\ref{alg:abtps} can be omitted. The integrator then remains unconditionally convergent~\cite{AHP+14}. If there exists a smooth (and hence unique~\cite{DS14}) strong solution of LLG, the a priori analysis in~\cite{FT17} guarantees first-order convergence in space and time.
\end{remark}

%-------------------------------------------------------------------
\subsection{Linear algebra}\label{section:linalg}
%-------------------------------------------------------------------
We suppose a numbering of the nodes, i.e., $\NN_h=\{z_1,\dots,z_N\}$. Let $\varphi_j\in\SS_h$ be the nodal hat function associated with $z_j$, i.e., $\varphi_j(z_k) = \delta_{jk}$, where $\delta_{jk}$ is Kronecker's delta. We then consider the following basis of $\courantspace_h$: Define
\begin{align}\label{eq:basis3D}
 \pphi_{3(j-1)+\ell} := \varphi_j\,\boldsymbol{e}_\ell
 \quad\text{for all }j=1,\dots,N \textrm{ and all } \ell=1,2,3.
\end{align}
Given $\mm_h^n \in \magnetizationset_h$, we then define $\matrixM,\matrixM_{k}[\mm_h^n],\matrixL,\matrixS[\mm_h^n]\in\R^{3N\times3N}$ as follows:
\begin{itemize}
\item $\matrixM_{ij} := \dual{\pphi_{i}}{\pphi_{j}}{\Omega}$ is the (symmetric and positive definite) mass matrix;
\item $\big(\matrixM_{k}[\mm_h^n]\big)_{ij} := \dual{(\weight_k(\lambda_h^n)/\alpha) \, \pphi_{i}}{\pphi_{j}}{\Omega}$ is the (symmetric and positive definite) weighted mass matrix, where $\lambda_h^n$ stems from~\eqref{eq:weight_input} and thus depends on $\mm_h^n$;
\item $\matrixL_{ij} := \dual{\nabla\pphi_i}{\nabla\pphi_j}{\Omega}$ is the (symmetric and positive semidefinite) stiffness matrix;
\item $(\matrixS[\mm_h^n])_{ij} := \dual{\mhn\times\pphi_i}{\pphi_j}{\Omega}$ is the (skew-symmetric) cross product matrix.
\end{itemize}
Moreover, we set
\begin{subequations}\label{eq:blockfrom}
\begin{align}\label{eq:defai}
M_{ij} := \dual{\varphi_i}{\varphi_j}{\Omega} \in \R
\quad \textrm{and} \quad
L_{ij} := \dual{\nabla \varphi_i}{\nabla \varphi_j}{\Omega} \in \R \quad 
\textrm{for } i,j=1,\ldots,N,
\end{align}
and note the block forms
\begin{align}
\matrixM
\, = \, 
%\begin{pmatrix} M_{11} \, \matrixI_{3 \times 3} & M_{12} \, \matrixI_{3 \times 3} & \cdots & M_{1N} \, \matrixI_{3 \times 3} \\
% M_{21} \, \matrixI_{3 \times 3} & M_{22} \, \matrixI_{3 \times 3} &\ddots&\vdots\\
% \vdots&\ddots&\ddots& M_{(N-1)N} \, \matrixI_{3 \times 3} \\
% M_{N1} \, \matrixI_{3 \times 3} &\cdots& M_{N(N-1)} \, \matrixI_{3 \times 3} & M_{NN} \, \matrixI_{3 \times 3}
% \end{pmatrix}.
\begin{pmatrix} M_{11} \, \matrixI_{3 \times 3} & \cdots & M_{1N} \, \matrixI_{3 \times 3} \\
 \vdots&\ddots& \vdots\\
 M_{N1} \, \matrixI_{3 \times 3} & \cdots  & M_{NN} \, \matrixI_{3 \times 3}
 \end{pmatrix}
 \quad \textrm{and} \quad
 \matrixL = 
 \begin{pmatrix} L_{11} \, \matrixI_{3 \times 3} & \cdots & L_{1N} \, \matrixI_{3 \times 3} \\
 \vdots&\ddots& \vdots\\
 L_{N1} \, \matrixI_{3 \times 3} & \cdots  & L_{NN} \, \matrixI_{3 \times 3}
 \end{pmatrix}.
\end{align}
\end{subequations}

If we replace $\tangentspace$ with $\courantspace_h$ in~\eqref{eq:tps2_variational}, the left-hand side of~\eqref{eq:tps2_variational} gives rise to the matrix
\begin{align}\label{eq:systemmatrix}
\matrixA_{k}[\mm_h^n] := \alpha \, \matrixM_{k}[\mm_h^n] + \beta(k) k \, \matrixL - \matrixS[\mm_h^n] \in \R^{3N \times 3N}. %\quad \textrm{where} \quad \beta(k) := \frac{\ellex^2}{2} \big( 1 + \rho(k) \big).
\end{align}
The right-hand side of~\eqref{eq:tps2_variational} gives rise to the vector $\vectorB[\mm_h^n]\in\R^{3N}$ with
\begin{align}\label{eq:righthandside}
\big( \vectorB [\mm_h^n] \big)_j & := 
- \ellex^2 \dual{\nabla \mm_h^n}{\nabla \pphi_j}{\Omega} + \dual{\boldsymbol{L}_h(\mm_h^n,\mm_h^{n-1})}{\vvarphi_h}{\Omega}.
\end{align}
Note that the matrix $\matrixA_{k}[\mm_h^n]$ is positive definite and hence regular, but not symmetric. Finally, define the 2D-equivalent to the basis from~\eqref{eq:basis3D} by 
\begin{align}\label{eq:basis2D}
 \ppsi_{2(j-1)+\ell} := \varphi_j\,\boldsymbol{e}_\ell
 \quad\text{for all }j=1,\dots,N \textrm{ and all } \ell=1,2.
\end{align}

%%%%%%%%%%%%%%%%%%%%%%%%%%%%%%%%%%%%%%%%%%%%%%%%%%%%%%%%%%%%%%%%%%%%
\section{The tangent space problem}\label{section:tangentspace} 
%%%%%%%%%%%%%%%%%%%%%%%%%%%%%%%%%%%%%%%%%%%%%%%%%%%%%%%%%%%%%%%%%%%%

In this section, we present a strategy, which translates the solution of the discrete variational formulation~\eqref{eq:tps2_variational} to a linear system in $\R^{2N} \cong \tangentspace$. To that end, we use Householder matrices: Given $\vectorM \in \R^3$ with $|\vectorM| = 1$, define $\widetilde{\matrixH}[\vectorM] \in \R^{3 \times 3}$ by
\begin{subequations}\label{eq:householder}
\begin{align}\label{eq:tildeH}
\widetilde{\matrixH}[\vectorM]
:=
\begin{cases}
 \matrixI - 2 \vectorW \vectorW^T, \quad \textrm{where} \quad
\vectorW := 
\frac{\vectorM + \vectorE_3}{|\vectorM + \vectorE_3|}
& \textrm{for } \vectorM \neq -\vectorE_3, \\
\big[ \, \vectorE_1, \vectorE_2, - \vectorE_3 \, \big]
& \textrm{for } \vectorM = -\vectorE_3.
\end{cases}
\end{align}
Then, $\widetilde{\matrixH}[\vectorM]$ is orthonormal with $\widetilde{\matrixH}[\vectorM] = {\widetilde{\matrixH}[\vectorM]}^T = {\widetilde{\matrixH}[\vectorM]}^{-1}$ and maps $\vectorE_3$ to $-\vectorM$. Define
\begin{align}\label{eq:standard_choice}
\matrixH[\vectorM] := \big[ \, \widetilde{\matrixH}[\vectorM] \vectorE_1,  \widetilde{\matrixH}[\vectorM] \vectorE_2 \, \big] \in \R^{3 \times 2},
\end{align}
i.e., $\operatorname{span}(\matrixH[\vectorM]) \bot \vectorM$ and $\operatorname{span}(\matrixH[\vectorM])$ mimics $\tangentspace$ nodewise. Moreover, for any orthogonal matrix $\matrixT_n \in \R^{3 \times 3}$ with $\matrixT_n = \matrixT_n^{-1} = \matrixT_n^T$, the matrix
\begin{align}\label{eq:thereplacement}
\matrixT_n \, \matrixH[\matrixT_n \vectorM] \in \R^{3 \times 2} \quad \textrm{instead of} \quad \matrixH[\vectorM] \in \R^{3 \times 2},
\end{align}
\end{subequations}
also satisfies $\operatorname{span}(\matrixT_n\matrixH[\matrixT_n\vectorM]) \bot \vectorM$. Hence, $\operatorname{span}(\matrixT_n\matrixH[\matrixT_n\vectorM])$ still mimics $\tangentspace$ nodewise. The following theorem provides a linear system in $\R^{2N} \cong \tangentspace$ for the solution to~\eqref{eq:tps2_variational}. The proof is postponed to Section~\ref{subsection:proof_howtosolve} below.

\begin{theorem}\label{theorem:howtosolve}
Recall $\matrixA_{k}[\mm_h^n]\in\R^{3N\times3N}$ and $\vectorB[\mm_h^n]\in\R^{3N}$ from Section~\ref{section:linalg}. Define 
\begin{align}\label{eq:defQ}
\matrixQ[\mm_h^n] := 
 \begin{pmatrix} \matrixT_n \matrixH [\matrixT_n \mm_h^n(z_1)] &\0&\cdots&\0\\
 \0& \matrixT_n \matrixH [\matrixT_n \mm_h^n(z_2)] &\ddots&\vdots\\
 \vdots&\ddots&\ddots&\0\\
 \0&\cdots&\0& \matrixT_n \matrixH [\matrixT_n \mm_h^n(z_N)]
 \end{pmatrix}\in\R^{3N\times2N},
\end{align}
where $\matrixH(\cdot)$ stems from~\eqref{eq:householder}. Then, the matrix $\matrixQ[\mm_h^n]^T\matrixA_{k}[\mm_h^n]\matrixQ[\mm_h^n] \in \R^{2N \times 2N}$ is positive definite and, in particular, regular. Moreover, the unique solution $\vectorX \in \R^{2N}$ of
\begin{align}\label{eq:linearsystem:constrainedQ}
 \big(\, \matrixQ[\mm_h^n]^T\matrixA_{k}[\mm_h^n]\matrixQ[\mm_h^n] \, \big) \, 
 \vectorX = 
 \matrixQ[\mm_h^n]^T \, \vectorB[\mm_h^n],
\end{align}
and the unique solution $\vh^n\in\tangentspace$ of the variational formulation~\eqref{eq:tps2_variational} satisfy
\begin{align}\label{eq:variationalform:vh}
 \vh^n = \sum_{j=1}^{3N} \vectorV_j \pphi_j
 \quad\text{with}\quad
 \vectorV:=\matrixQ[\mm_h^n] \, \vectorX\in\R^{3N}.
\end{align}
%where $\vectorX\in\R^{2N}$ is the solution of the linear system~\eqref{eq:linearsystem:constrainedQ}.
\end{theorem}

\begin{remark}\label{remark:howtosolve}
For the validity of Theorem~\ref{theorem:howtosolve}, it is only relevant that $\matrixH[\mm_h^n(z_i)]$ has orthonormal columns and that $\operatorname{span}(\matrixH[\mm_h^n(z_i)]) \bot \mm_h^n(z_i)$ for all $i=1,\dots,N$. Given $\vectorM \in \R^3$, with $|\mm|=1$, alternative strategies from~\cite[Lemma 6.1.2]{Ruggeri16} are, e.g.,
\begin{itemize}
\item either to set $\vectorW := 
\frac{\vectorM + \sigma \vectorE_3}{|\vectorM + \sigma \vectorE_3|}$, where $\sigma = \sign(\vectorM_3)$;
\item or to use the transformation matrix of the rotation around the axis $\vectorE_3\times\vectorM$ by an angle $\varphi$ such that $\cos\varphi = \vectorM\cdot\vectorE_3$ and $\sin\varphi = \lvert\vectorM\times\vectorE_3\rvert$.
\end{itemize}
Moreover, even different strategies for $\matrixH[\mm_h^n(z_i)]$ at different nodes $z_i$ are possible.
\end{remark}

%%%%%%%%%%%%%%%%%%%%%%%%%%%%%%%%%%%%%%%%%%%%%%%%%%%%%%%%%%%%%%%%%%%%
\section{Preconditioning} \label{section:precond}
%%%%%%%%%%%%%%%%%%%%%%%%%%%%%%%%%%%%%%%%%%%%%%%%%%%%%%%%%%%%%%%%%%%%

To solve the tangent space system~\eqref{eq:linearsystem:constrainedQ}, we aim to choose a preconditioner $\matrixP \in \R^{2N \times 2N}$ and employ the GMRES algorithm~\cite{SS86,Saad03} to the preconditioned system
\begin{align}\label{eq:preconditionedsystem}
\matrixP \matrixA_{\matrixQ}[\mm_h^n] \vectorX := \matrixP \matrixQ[\mm_h^n]^T \matrixA_{k}[\mm_h^n] \matrixQ[\mm_h^n] \vectorX = \matrixP \matrixQ[\mm_h^n]^T \vectorB[\mm_h^n] =: \matrixP \vectorB_{\matrixQ}[\mm_h^n] .
\end{align}
In the following sections, we discuss possible choices for $\matrixP$. We rely on the symmetric part of $\matrixA_{\matrixQ}[\mm_h^n]$, where we replace $\weight_k$ by the parameter
\begin{align}\label{eq:alphaP}
\alpha_{\matrixP} > 0.
\end{align}
In particular, this includes the case $\alpha_{\matrixP} = \alpha$. Note that GMRES requires only the action of the preconditioner $\matrixP$ on a vector. Moreover, recall that $\matrixQ[\mm_h^n]$ from~\eqref{eq:defQ} implicitly depends on the arbitrary but fixed matrix $\matrixT_n \in \R^{3 \times 3}$ from~\eqref{eq:thereplacement}. We refer to Section~\ref{section:idealaxis} below for the possible construction of the matrix $\matrixT_n$, for given $\mm_h^n \in \magnetizationset_h$. 

%-------------------------------------------------------------------
\subsection{Theoretical preconditioner}\label{subsection:theoreticalprecond}
%-------------------------------------------------------------------
For $\mmu_h \in \magnetizationset_h$, we first consider
\begin{align}\label{eq:forma_constrained_preconditioner}
 \matrixP_{\matrixQ}[\mmu_h] := 
 \big( \, \matrixQ[\mmu_h]^T
 (\alpha_{\matrixP} \matrixM+\beta(k)k\,\matrixL)
 \matrixQ[\mmu_h] \, \big)^{-1}
 \in\R^{2N\times2N}.
\end{align}
To analyse the preconditioned GMRES algorithm, we define the energy scalar product
\begin{align}\label{eq:energyscalar}
\energydual{\vectorX}{\vectorY}{\mmu_h} & := \vectorX \cdot \big( \, \matrixP_{\matrixQ}[\mmu_h] \, \big)^{-1}\vectorY
\quad \textrm{for all } \vectorX, \vectorY \in \R^{2N}
\end{align}
and denote the induced norm by $|\!|\!|\cdot|\!|\!|_{\mmu_h}$. The following theorem shows that (the time-step dependent) $\matrixP_{\matrixQ}[\mm_h^n]$ is a suitable choice and that, in practice, one can keep and reuse $\matrixP_{\matrixQ}[\mm_h^n]$ for several time-steps. Its proof is postponed to Section~\ref{proof:theoretical_convergence}.

\begin{theorem}\label{theorem:theoretical_convergence}
Let $\alpha_{\matrixP} \geq \alpha$. Let $\mmu_h \in \magnetizationset_h$ be arbitrary. Consider the preconditioned GMRES algorithm with the preconditioner $\matrixP_{\matrixQ}[\mmu_h]$ from~\eqref{eq:forma_constrained_preconditioner} for the solution of~\eqref{eq:preconditionedsystem} with the initial guess $\vectorX^{(0)} \in \R^{2N}$. For $\ell \in \N_0$, let $\vectorX^{(\ell)} \in \R^{2N}$ denote the GMRES iterates with the corresponding residuals
\begin{align*}
\vectorR^{(\ell)} := \matrixP_{\matrixQ}[\mmu_h] \vectorB_{\matrixQ}[\mm_h^n] 
- \matrixP_{\matrixQ}[\mmu_h] \matrixA_{\matrixQ}[\mm_h^n] \, \vectorX^{(\ell)} \in \R^{2N}.
\end{align*}
Then, the following two assertions {\rm (i)}--{\rm (ii)} hold true:

{\rm (i)} There exists a constant $C>1$, which depends only on $\Cmesh$, such that
\begin{align}\label{eq:estimate:theoretical_1}
|\!|\!|\vectorR^{(\ell)}|\!|\!|_{\mmu_h} \, 
\leq \, 
\bigg[ \, 1 -  C^{-1} \Big( \, \frac{\alpha}{\alpha_{\matrixP}} \, \Big)^2 \, \bigg( \, 1 
+ \frac{1}{\alpha_{\matrixP}} 
+ \frac{ \norm{\weight_k - \alpha_{\matrixP}}{L^{\infty}(\Omega)} }{ \alpha_{\matrixP} } \, \bigg)^{-2} \mathbb{F}^{-4} \, \bigg]^{\ell/2} \,
|\!|\!|\vectorR^{(0)}|\!|\!|_{\mmu_h}
\end{align} 
for all $\ell \in \N_0$, where
\begin{align}\label{eq:factor:theoretical_1}
\mathbb{F} \, := \,
1 + 
\frac{\beta(k)k}{\alpha_{\matrixP} h^2}
\max_{z \in \NN_h} \matrixnorm{\matrixH[\matrixT_n\mm_h^n(z)] - \matrixH[\matrixT_n\mmu_h(z)]}^2  \,  \geq  \, 1.
\end{align}

{\rm (ii)} If, additionally, $1 + (\matrixT_n \mm_h^n(z))_3 \geq \gamma > 0$ and $1 + (\matrixT_n\mmu_h(z))_3 \geq \gamma > 0$ for all nodes $z \in \NN_h$, the statement of {\rm (i)} holds with the $h$-independent factor
\begin{align}
\begin{split}
\mathbb{F} \, & := \,
1 
+ 
\gamma^{-2} \, \norm{\mm_h^n - \mmu_h}{\LL^\infty(\Omega)}^2  
+  
\frac{\beta(k) k}{\alpha_{\matrixP} \gamma^{2}} \, \norm{\nabla \mm_h^n - \nabla \mmu_h}{\LL^\infty(\Omega)}^2 \\
& \qquad + 
\frac{\beta(k) k}{\alpha_{\matrixP} \gamma^{6}} \,
\big( \, \norm{\nabla \mm_h^n}{\LL^\infty(\Omega)}^2 + \norm{\nabla \mmu_h}{\LL^\infty(\Omega)}^2 \, \big) \,
\norm{\mm_h^n - \mmu_h}{\LL^{\infty}(\Omega)}^2 \,  \geq \, 1.
\end{split}
\end{align}
\end{theorem}

%-------------------------------------------------------------------
\subsection{Stationary preconditioning}\label{subsection:stationary}
%-------------------------------------------------------------------
We consider a preconditioner which is independent of the time-step.
Similarly to $\matrixM, \matrixL \in \R^{3N \times 3N}$ from Section~\ref{section:linalg}, define the matrices $\matrixM^{\textrm{2D}}, \matrixL^{\textrm{2D}} \in \R^{2N \times 2N}$, which correspond to the nodal basis $(\ppsi_i)_{i=1}^{2N}$ of $(\SS_h)^2$ from~\eqref{eq:basis2D}:% as follows:
\begin{itemize}
\item $\matrixM_{ij}^{\textrm{2D}} \!:=\! \dual{\ppsi_{i}}{\ppsi_{j}}{\Omega}$ is the (symmetric and positive definite) mass matrix,
\item $\matrixL_{ij}^{\textrm{2D}} \!:=\! \dual{\nabla\ppsi_i}{\nabla\ppsi_j}{\Omega}$ is the (symmetric and positive semidefinite) stiffness matrix.
\end{itemize}
Then, consider the stationary preconditioner
\begin{align}\label{eq:defP2D}
\matrixP^{\textrm{2D}} := \big( \alpha_{\matrixP} \matrixM^{\textrm{2D}} + \beta(k)k \, \matrixL^{\textrm{2D}} \big)^{-1} \in \R^{2N \times 2N}.
\end{align}
Denote the corresponding energy scalar product by
\begin{align*}
\energydual{\vectorX}{\vectorY}{} & := \vectorX \cdot (\alpha_{\matrixP} \matrixM^{\rm{2D}} + \beta(k) k \matrixL^{\rm{2D}} ) \vectorY
\quad \textrm{for all } \vectorX, \vectorY \in \R^{2N},
\end{align*}
and the induced norm by $|\!|\!|\cdot|\!|\!|$. The following corollary discusses the performance of $\matrixP^{\textrm{2D}}$. It is a direct consequence of Theorem~\ref{theorem:theoretical_convergence}{\rm (ii)}.

\begin{corollary}\label{corollary:stationaryprecond}
Let $\alpha_{\matrixP} \geq \alpha$. Consider the preconditioned GMRES algorithm with $\matrixP^{\rm{2D}}$ from~\eqref{eq:defP2D} for the solution of~\eqref{eq:preconditionedsystem} with the initial guess $\vectorX^{(0)} \in \R^{2N}$. For $\ell \in \N_0$, let $\vectorX^{(\ell)} \in \R^{2N}$ denote the GMRES iterates with the corresponding residuals
\begin{align*}
\vectorR^{(\ell)} := \matrixP^{\rm{2D}} \vectorB_{\matrixQ}[\mm_h^n] 
- \matrixP^{\rm{2D}} \matrixA_{\matrixQ}[\mm_h^n] \, \vectorX^{(\ell)} \in \R^{2N}.
\end{align*}
Let $1 + \big(\matrixT_n\mm_h^n(z)\big)_3 \geq \gamma > 0$ for all nodes $z \in \NN_h$. Then, there exists a constant $C>1$, which depends only on $\Cmesh$, such that
\begin{align}\label{eq:estimate:stationary_1}
|\!|\!|\vectorR^{(\ell)}|\!|\!| \, 
\leq \, 
\bigg[ \, 1 - C^{-1} \Big( \, \frac{\alpha}{\alpha_{\matrixP}} \, \Big)^2 \bigg( \, 1 
+ \frac{1}{\alpha_{\matrixP}} 
+ \frac{ \norm{\weight_k - \alpha_{\matrixP}}{L^{\infty}(\Omega)} }{ \alpha_{\matrixP} } \, \bigg)^{-2} \mathbb{F}^{-4} \, \bigg]^{\ell/2} \,
|\!|\!|\vectorR^{(0)}|\!|\!|
\end{align} 
for all $\ell \in \N_0$, where
\begin{align}
\begin{split}
\mathbb{F} \, & := \,
1 
+ 
\gamma^{-2} + 
\frac{\beta(k) k}{\alpha_{\matrixP} \gamma^{6}} \, \norm{\nabla \mm_h^n}{\LL^\infty(\Omega)}^2 \,  \geq \, 1.
\end{split}
\end{align}
\end{corollary}

\begin{proof}
Recall that $\matrixT_n = \matrixT_n^{-1} = \matrixT_n^T$. For constant $\mmu_h := \matrixT_n \vectorE_3 \in \magnetizationset_h$, we get that
\begin{subequations}\label{eq:concludestationary}
\begin{align}
\matrixT_n \matrixH[\matrixT_n \mmu_h(z_i)] = \matrixT_n \matrixH[\vectorE_3]
= \matrixT_n [  \vectorE_1, \vectorE_2 ] \in \R^{3 \times 2}
\end{align}
and thus
\begin{align}
( \matrixT_n \matrixH[\matrixT_n \mmu_h(z_i)] )^T (\matrixT_n \matrixH[\matrixT_n \mmu_h(z_j)] )
=  [  \vectorE_1, \vectorE_2 ]^T \matrixT_n \matrixT_n [  \vectorE_1, \vectorE_2 ] = \matrixI_{2 \times 2}.
\end{align}
\end{subequations}
for all $i,j = 1, \dots, N$. Together with the block forms from~\eqref{eq:blockfrom}, this yields that
\begin{align*}
\matrixP_{\matrixQ}[\mmu_h] \stackrel{\eqref{eq:forma_constrained_preconditioner}}{=} 
\big( \, \matrixQ[\mmu_{h}]^T \big( \, \alpha_{\matrixP} \matrixM + \beta(k)k \matrixL \, \big) \matrixQ[\mmu_h]  \, \big)^{-1}
\stackrel{\eqref{eq:concludestationary}}{=} 
\big( \alpha_{\matrixP} \matrixM^{\textrm{2D}} + \beta(k)k \matrixL^{\textrm{2D}} \big)^{-1} 
\stackrel{\eqref{eq:defP2D}}{=} 
\matrixP^{\textrm{2D}}.
\end{align*}
Since $\nabla \mmu_h = \0$ a.e.\ in $\Omega$ and $\norm{\mm_h^n - \mmu_h}{\LL^{\infty}(\Omega)} \leq 2$, Theorem~\ref{theorem:theoretical_convergence}{\rm (ii)} proves the result.
\end{proof}

%-------------------------------------------------------------------
\subsection{Practical preconditioner}\label{subsection:practicalprecond}
%-------------------------------------------------------------------
For general problems of type~\eqref{eq:linearsystem:constrainedQ}, the work~\cite{NS96} proposes (without a proof) to consider the practical preconditioner
\begin{align}\label{eq:pracital_constrained_preconditioner}
 \widetilde{\matrixP}_{\matrixQ}[\mm_h^n] := 
 \matrixQ[\mm_h^n]^T \big(\alpha_{\matrixP}\matrixM+\beta(k) k\matrixL\big)^{-1} \matrixQ[\mm_h^n] \in\R^{2N\times2N}
\end{align}
to approximate the theoretical preconditioner $\matrixP[\mm_h^n]$ from~\eqref{eq:forma_constrained_preconditioner}. We note that unlike $\matrixP_{\matrixQ}[\mmu_h]$, it makes no sense to consider $\widetilde{\matrixP}_{\matrixQ}[\mmu_h]$ for $\mmu_h \neq \mm_h^n$, since the preconditioned GMRES algorithm exploits only the matrix-vector product with the preconditioner and $\matrixQ[\mm_h^n]$ is computed anyway. For the analysis, we define the energy scalar product
\begin{align}\label{eq:energyscalarN}
\energydual{\vectorX}{\vectorY}{\mm_h^n}^{\prime} & := \vectorX \cdot \big( \, \widetilde{\matrixP}_{\matrixQ}[\mm_h^n] \, \big)^{-1}\vectorY
\quad \textrm{for all } \vectorX, \vectorY \in \R^{2N},
\end{align}
and denote the corresponding norm by $|\!|\!|\cdot|\!|\!|_{\mm_h^n}^{\prime}$. The following theorem discusses the performance of preconditioned GMRES with the practical preconditioner $\widetilde{\matrixP}_{\matrixQ}[\mm_h^n]$. Its proof is postponed to Section~\ref{proof:practical_convergence}. 
\pagebreak

\begin{theorem}\label{theorem:practical_convergence}
Let $\alpha_{\matrixP} \geq \alpha$. Consider the preconditioned GMRES algorithm with the preconditioner $\widetilde{\matrixP}_{\matrixQ}[\mm_h^n]$ from~\eqref{eq:pracital_constrained_preconditioner} for the solution of~\eqref{eq:preconditionedsystem} with the initial guess $\widetilde{\vectorX}^{(0)} \in \R^{2N}$. For $\ell \in \N_0$, let $\widetilde{\vectorX}^{(\ell)} \in \R^{2N}$ denote the GMRES iterates with the corresponding residuals
\begin{align*}
\widetilde{\vectorR}^{(\ell)} := \widetilde{\matrixP}_{\matrixQ}[\mm_h^n] \vectorB_{\matrixQ}[\mm_h^n] 
- \widetilde{\matrixP}_{\matrixQ}[\mm_h^n] \matrixA_{\matrixQ}[\mm_h^n] \, \widetilde{\vectorX}^{(\ell)} \in \R^{2N}.
\end{align*}
Then, the following two assertions {\rm (i)}--{\rm (ii)} hold true:

{\rm (i)} There exists a constant $C>1$, which depends only on $\Cmesh$, such that
\begin{align}\label{eq:estimate:practical__new_1}
|\!|\!|\widetilde{\vectorR}^{(\ell)}|\!|\!|_{\mm_h^n}^{\prime} \, 
\leq \, 
\bigg[ \, 1 - C^{-1} \Big( \, \frac{\alpha}{\alpha_{\matrixP}} \, \Big)^2  \bigg( \, 1 
+ \frac{1}{\alpha_{\matrixP}} 
+ \frac{ \norm{\weight_k - \alpha_{\matrixP}}{L^{\infty}(\Omega)} }{ \alpha_{\matrixP} } \, \bigg)^{-2} \widetilde{\mathbb{F}}^{-1} \, \bigg]^{\ell/2} \,
|\!|\!|\widetilde{\vectorR}^{(0)}|\!|\!|_{\mm_h^n}^{\prime}
\end{align} 
for all $\ell \in \N_0$, where
\begin{align}\label{eq:factor:practical_1}
\widetilde{\mathbb{F}} \, := \, 1  + \frac{\beta(k)k}{\alpha_{\matrixP} h^2} \, \geq \, 1.
\end{align}

{\rm (ii)} If, additionally, $1 + (\matrixT_n\mm_h^n(z))_3 \geq \gamma > 0$ for all nodes $z \in \NN_h$, the statement of {\rm (i)} holds with the $h$-independent factor
\begin{align}
\widetilde{\mathbb{F}} \, & 
:= \,  1 + \frac{\beta(k)k}{\alpha_{\matrixP} \gamma^4} \norm{\nabla \mm_h^n}{\LL^\infty(\Omega)}^2
 \, \geq \, 1.
\end{align}
\end{theorem}

%-------------------------------------------------------------------
\subsection{Jacobi-type preconditioner}\label{subsection:jacobi}
%-------------------------------------------------------------------
Consider the following approximation to the stationary preconditioner $\matrixP^{\rm 2D}$ from~\eqref{eq:defP2D}: Recalling $M_{ij}, L_{ij} \in \R$ from~\eqref{eq:defai}, we set
\begin{align}\label{eq:defjac}
\matrixP^{\textrm{jac}}_i \, := \, \big( \alpha_{\matrixP} M_{ii} + \beta(k)k L_{ii} \big)^{-1} 
\, \matrixI_{2 \times 2}
\in \R^{2 \times 2} \quad \textrm{for all } i=1,\ldots,N
\end{align}
and define the stationary Jacobi-type preconditioner
\begin{align}
\matrixP^{\textrm{jac}} := 
 \begin{pmatrix} \matrixP^{\textrm{jac}}_1 &\0&\cdots&\0\\
 \0& \matrixP^{\textrm{jac}}_2 &\ddots&\vdots\\
 \vdots&\ddots&\ddots&\0\\
 \0&\cdots&\0& \matrixP^{\textrm{jac}}_N
 \end{pmatrix}\in\R^{2N\times2N}.
\end{align}
Given $\mm_h^n \in \magnetizationset_h$, the matrix $\matrixP_{\matrixQ}^{\textrm{jac}}[\mm_h^n] \in \R^{2N \times 2N}$
\begin{align*}
%\matrixP_{\matrixQ}^{\textrm{jac}}[\mm_h^n] \in \R^{2N \times 2N}, \textrm {where }
\Big[ \, \big( \, \matrixP_{\matrixQ}^{\textrm{jac}}[\mm_h^n] \, \big)^{-1} \Big]_{ij} := 
\begin{cases}
\big( \matrixQ[\mm_h^n]^T (\alpha_{\matrixP} \matrixM + \beta(k)k \matrixL) \matrixQ[\mm_h^n] \big)_{ii} & \textrm{for } i = j, \\
0 & \textrm{else,}
\end{cases}
\end{align*}
is the Jacobi-type approximation of of $\matrixP_{\matrixQ}[\mm_h^n]$ from~\eqref{eq:forma_constrained_preconditioner}.
Similarly, the matrix
\begin{align*}
\matrixQ[\mm_h^n]^T \widetilde{\matrixP}^{\textrm{jac}} \matrixQ[\mm_h^n] \in \R^{2N \times 2N}, 
\quad \textrm{where }
\Big[ \, \big( \, \widetilde{\matrixP}^{\textrm{jac}} \, \big)^{-1} \, \Big]_{ij} := 
\begin{cases}
\big( \alpha_{\matrixP} \matrixM + \beta(k)k \matrixL \big)_{ii} & \textrm{for } i = j, \\
0 & \textrm{else,}
\end{cases}
\end{align*}
is the Jacobi-type approximation of $\widetilde{\matrixP}_{\matrixQ}[\mm_h^n]$ from~\eqref{eq:pracital_constrained_preconditioner}. The following proposition states that all three definitions of Jacobi-type preconditioners coincide.

\begin{proposition}\label{prop:Jacobi}
It holds that
\begin{align*}
 \matrixP^{\rm{jac}} = \matrixP_{\matrixQ}^{\rm{jac}}[\mm_h^n] = \matrixQ[\mm_h^n]^T \widetilde{\matrixP}^{\rm{jac}} \matrixQ[\mm_h^n].
\end{align*}
\end{proposition}

\begin{proof}
The statement follows from the representations of $\matrixM$ and $\matrixL$ from~\eqref{eq:blockfrom} as well as the block-diagonal definition of $\matrixQ[\mm_h^n]$, where the blocks $\matrixT_n\matrixH[\matrixT_n\mm_h^n(z_i)]$ have orthonormal columns for all nodes $z_i \in \NN_h$.
\end{proof}

%-------------------------------------------------------------------
\subsection{Practical computation of $\matrixT_n$}\label{section:idealaxis}
%-------------------------------------------------------------------
With $\gamma > 0$ being as large as possible, Theorem~\ref{theorem:theoretical_convergence}~{\rm (ii)}, Corollary~\ref{corollary:stationaryprecond}, and Theorem~\ref{theorem:practical_convergence}~{\rm (ii)} require that
\begin{align}\label{eq:ideal_axis}
1 + \mm_h^n(z) \cdot \matrixT_n \vectorE_3 \geq \gamma > 0 \quad \textrm{for all nodes } z \in \NN_h.
\end{align}
For given $\mm_h^n \in \magnetizationset_h$, we thus aim to choose the matrix $\matrixT_n \in \R^{3 \times 3}$ from~\eqref{eq:thereplacement} such that $\gamma$ is large. To this end, define $d_{\ell}^{\star} \in [0,2]$ by
\begin{subequations}\label{eq:define_ds}
\begin{align}
d_{\ell}^{+} & := 1 - \max_{z \in \NN_h} (\mm_h^n(z))_{\ell} ,
\quad  d_{\ell}^{-} := 1 + \, \min_{z \in \NN_h} (\mm_h^n(z))_{\ell},
\quad \textrm{for all }  \ell \in \{1,2,3\}.
\end{align}
\end{subequations}
In addition, define $\matrixT_{\ell}^{\star} \in \R^{3 \times 3}$ with $\matrixT_n = \matrixT_n^{-1} = \matrixT_n^T$ by
\begin{subequations}\label{eq:define_Ts}
\begin{align}
\matrixT_1^+ & := [-\vectorE_3,\vectorE_2,-\vectorE_1], 
\quad
\matrixT_1^- := [\vectorE_3,\vectorE_2,\vectorE_1], \\
\matrixT_2^+ &:= [\vectorE_1,-\vectorE_3,-\vectorE_2], 
\quad
\matrixT_2^- := [\vectorE_1,\vectorE_3,\vectorE_2], \\ 
\matrixT_3^+ &:= [\vectorE_1,\vectorE_2,-\vectorE_3], 
\quad \,\,\,\,\,
\matrixT_3^- := [\vectorE_1,\vectorE_2,\vectorE_3].
\end{align}
\end{subequations}
For all $\ell \in \{1,2,3\}$ and $\star \in \{ + , - \}$, it holds that
\begin{align}\label{eq:gamma_dl}
1 + \mm_h^n(z) \cdot \matrixT_{\ell}^{\star} \vectorE_3 
\, \geq \, d_{\ell}^{\star} \, \in \, [0,2] \quad \textrm{for all nodes } z \in \NN_h.
\end{align}
Hence,~\eqref{eq:ideal_axis} holds with $\gamma \in [0,2]$ being the maximum $d_{\ell}^{\star}$, and $\matrixT_n = \matrixT_{\ell}^{\star}$ in~\eqref{eq:define_Ts}. We note that other (more involved) strategies are possible. Finally, we note that our construction leads to $\gamma = 0$, if and only if 
\begin{align*}
\{ \pm \vectorE_1, \pm \vectorE_2, \pm \vectorE_3 \} \subset \{ \mm_h^n(z) : z \in \NN_h \}.
\end{align*}

%%%%%%%%%%%%%%%%%%%%%%%%%%%%%%%%%%%%%%%%%%%%%%%%%%%%%%%%%%%%%%%%%%%%
\section{Numerics} \label{section:numerics}
%%%%%%%%%%%%%%%%%%%%%%%%%%%%%%%%%%%%%%%%%%%%%%%%%%%%%%%%%%%%%%%%%%%%

In this section, we underpin our theoretical findings with numerical experiments. To this end, we employ our 
\texttt{C++}-code for computational micromagnetics, which is based on the FEM library \texttt{NGSolve}~\cite{ngs}. For BEM computations, we employ the \texttt{BEM++} library \cite{SBA+15}. Moreover, we couple \texttt{NGSolve} and \texttt{BEM++} via \texttt{ngbem}~\cite{Rieder}.

We always employ the first-order tangent plane scheme, i.e., Algorithm~\ref{alg:abtps} with $\weight_k$, $\beta$, and $\boldsymbol{L}_h(\mm_h^n,\mm_h^{n-1})$ as in~\eqref{eq:TPS1}, where, in particular, we always set $\Theta = 1$ in~\eqref{eq:TPS1_beta}.

To solve the (preconditioned) linear system~\eqref{eq:linearsystem:constrainedQ}, we employ the (preconditioned) GMRES algorithm~\cite{SS86,Saad03}. Our implementation is based on the template routine from \texttt{Netlib}~\cite{gmres}, where we employ the iteration tolerance $\varepsilon = 10^{-14}$. To save memory, we restart GMRES after every $200$ iterations. Note that this is commonly referred to as restarted GMRES; cf., e.g.,~\cite[Algorithm~6.11]{Saad03}. As initial value for the GMRES iteration, we always choose $\vectorX_0 = \0$.

The experiments of this section focus on the (possibly $h$-independent) iteration numbers of the preconditioned GMRES algorithm. To evaluate the inverse matrices in the preconditioners from Section~\ref{section:precond} (e.g., in~\eqref{eq:forma_constrained_preconditioner}), we always solve the corresponding linear system via Gaussian elimination.

\begin{remark}\label{remark:complexity}
The corresponding 'inverse' matrices in the stationary preconditioner $\matrixP^{\textrm{2D}}$ from~\eqref{eq:defP2D} as well as the practical preconditioner $\widetilde{\matrixP}_{\matrixQ}[\mm_h^n]$ from~\eqref{eq:pracital_constrained_preconditioner} are block-diagonal matrices and consist of the same $N \times N$-stationary matrix block. Moreover, $\matrixQ[\mm_h^n]$ from~\eqref{eq:defQ} in $\widetilde{\matrixP}_{\matrixQ}[\mm_h^n]$ has a block-diagonal form and is explicitly available at each time-step. Hence, $\matrixP^{\textrm{2D}}$ and $\widetilde{\matrixP}_{\matrixQ}[\mm_h^n]$ have a similar computational complexity. In contrast to that, the theoretical preconditioner $\matrixP_{\matrixQ}[\mm_h^n]$ from~\eqref{eq:forma_constrained_preconditioner} has to be rebuilt at every time-step.
\end{remark}

%%%%%%%%%%%%%%%%%%%%%%%%%%%
%%%%%%%%%%%%%%%%%%%%
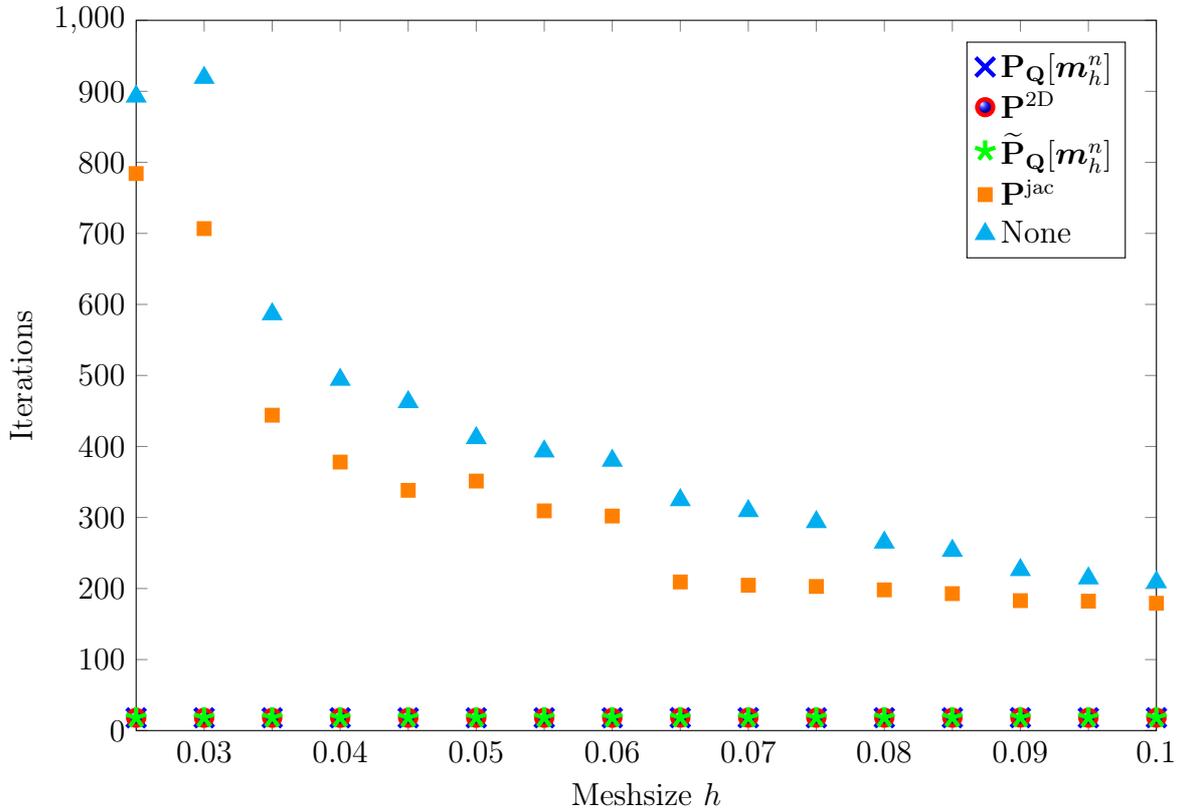
\begin{figure}
\begin{tikzpicture}
\pgfplotstableread{plots/cube/iterations.dat}{\data}
\begin{axis}[
height = 110mm,
width = 150mm,
xlabel={Meshsize $h$},
ylabel={Iterations},
xmin=0.025,
xmax=0.1,
xtick={0.025,0.03,0.035,0.04,0.045,0.05,0.055,0.06,0.065,0.07,0.075,0.08,0.085,0.09,0.095,0.1},
xticklabels={,0.03,,0.04,,0.05,,0.06,,0.07,,0.08,,0.09,,0.1},
ymin=0,
ymax=1000,
legend pos=north east,
every axis legend/.append style={nodes={right}}],
%scaled ticks=false, tick label style={/pgf/number format/fixed},
]
\addplot[only marks, blue, mark=x, mark size=5,ultra thick] table[x=h, y=implicit] {\data};
\addplot[only marks, red, mark=ball, mark size=3,ultra thick] table[x=h, y=twoD] {\data};
\addplot[only marks, green, mark=star, mark size=4,ultra thick] table[x=h, y=explicit] {\data};
\addplot[only marks, orange, mark=square*, ultra thick] table[x=h, y=jacobi] {\data};
\addplot[only marks, cyan, mark=triangle*, mark size=3,ultra thick] table[x=h, y=none] {\data};
\legend{$\matrixP_{\matrixQ}[\mm_h^n]$, $\matrixP^{\textrm{2D}}$, $\widetilde{\matrixP}_{\matrixQ}[\mm_h^n]$, $\matrixP^{\textrm{jac}}$, None};
\end{axis}
\end{tikzpicture}
\caption{Experiment of Section~\ref{subsection:cube}: Average number of GMRES iterations.}
\label{fig:academic}
\end{figure}
%%%%%%%%%%%%%%%%%%%%

%%%%%%%%%%%%%%%%%%%%%%%%%%%
%
%-------------------------------------------------------------------
\subsection{An academic example}\label{subsection:cube}
%-------------------------------------------------------------------
We investigate the iteration number for different choices $\matrixP$ in~\eqref{eq:forma_constrained_preconditioner} with respect to the mesh-size $h$. To this end, we adapt the setting of~\cite[Section 6.1]{PRS17}: On $\Omega := (0,1)^3$, we employ the initial value and applied field
\begin{align*}
\mm^0 := (1,0,0)^T 
\quad \textrm{and} \quad 
\ff(x_1,x_2,x_3) := 10 \, (\sin (x_1) , \cos (x_1) ,0)^T,
\end{align*}
respectively. We set $T=1$, $\alpha = 0.5$, $\ellex^2 = 10 \equiv \beta(k)$; see also~\eqref{eq:TPS1_beta}. The lower-order contributions $\boldsymbol{\pi}$ consist only of the stray field. With this configuration, the magnetization is expected to align itself in the direction of (the non-constant) applied field $\ff$. For time discretization, we fix $k=10^{-2}$. For space discretization, we employ the meshes generated by the \texttt{NGS/Py}~\cite{ngs} embedded module {\tt Netgen} with the mesh-sizes
\begin{align*}
h \, \in \, \big\{ \, \frac{j}{2} \, \cdot10^{-2} : j = 5,6,\dots,20 \, \big\}.
%\{ 2.5\cdot10^{-2}, 3\cdot10^{-2}, 3.5\cdot10^{-2}, \dots , 9.0\cdot10^{-2}, 9.5\cdot10^{-2}, 10^{-1} \}.
\end{align*}
We employ the different preconditioners from Section~\ref{section:precond} with $\alpha_{\matrixP} = 1$ or do not use preconditioning for the iterative solution of the underlying linear system~\eqref{eq:linearsystem:constrainedQ}. Moreover, we always fix $\matrixT_n := \matrixI_{3 \times 3}$, i.e., we always employ the standard choice~\eqref{eq:standard_choice}.

In Figure~\ref{fig:academic}, we plot the average number of GMRES iterations. As expected, no preconditioning (None) requires the most iterations. The Jacobi preconditioner $\matrixP^{\textrm{jac}}$ brings a slight improvement. However, (None) and $\matrixP^{\textrm{jac}}$ are both not robust with respect to the mesh-size $h$. In contrast to that, the theoretical ($\matrixP_{\matrixQ}[\mm_h^n]$), the stationary ($\matrixP^{\textrm{2D}}$), and the practical ($\widetilde{\matrixP}_{\matrixQ}[\mm_h^n]$) preconditioner from Section~\ref{subsection:theoreticalprecond}--\ref{subsection:practicalprecond}, respectively, require significantly less iterations. Moreover, all three latter options are robust with respect to the mesh-size $h$. However, $\matrixP_{\matrixQ}[\mm_h^n]$ is computationally more expensive than $\matrixP^{\textrm{2D}}$ and $\widetilde{\matrixP}_{\matrixQ}[\mm_h^n]$; cf. Remark~\ref{remark:complexity}.
In conclusion, the latter experiment suggests to use either the stationary preconditioner $\matrixP^{\textrm{2D}}$ or the practical preconditioner $\widetilde{\matrixP}_{\matrixQ}[\mm_h^n]$.

%%%%%%%%%%%%%%%%%%%%%%%%%%%%%%%%
%%%%%%%%%%%%%%%%%%%%
\begin{figure}

\centering
\begin{subfigure}{0.32\textwidth}
\scalebox{0.4}{
\begin{tikzpicture}
\pgfplotstableread{plots/mumag4/Alpha/itnr.dat}{\data}
\begin{axis}[
width = 120mm,
xlabel={Physical time (\si{\nano\second})},
ylabel={Iterations},
xmin=0.0,
xmax=3,
ymin=0,
ymax=700,
legend style={at={(axis cs:0.2,220)},anchor=south west,nodes={right}},
]
\addplot[blue, mark size=1,ultra thick] table[x=t, y=itnr_implicit] {\data};
\addplot[red, mark size=1,ultra thick] table[x=t, y=itnr_2D] {\data};
\addplot[green, mark size=1,ultra thick, dashed] table[x=t, y=itnr_explicit] {\data};
\addplot[orange, mark size=1,ultra thick] table[x=t, y=itnr_jacobi] {\data};
\addplot[cyan, mark size=1,ultra thick] table[x=t, y=itnr_none] {\data};
\legend{$\matrixP_{\matrixQ}[\mm_h^n]$, $\matrixP^{\textrm{2D}}$, $\widetilde{\matrixP}_{\matrixQ}[\mm_h^n]$, $\matrixP^{\textrm{jac}}$, None};
\end{axis}
\end{tikzpicture}
}
\caption{$\alpha = 0.02$, $\alpha_{\matrixP} = 1$}
\end{subfigure}
\begin{subfigure}{0.32\textwidth}
\scalebox{0.4}{
\begin{tikzpicture}
\pgfplotstableread{plots/mumag4/Alpha/itnr_bigalpha.dat}{\data}
\begin{axis}[
width = 120mm,
xlabel={Physical time (\si{\nano\second})},
ylabel={Iterations},
xmin=0.0,
xmax=3,
ymin=0,
ymax=200,
legend style={at={(axis cs:2.2,80)},anchor=south west,nodes={right}},
]
\addplot[blue, mark size=1,ultra thick] table[x=t, y=itnr_implicit_BA] {\data};
\addplot[red, mark size=1,ultra thick] table[x=t, y=itnr_2D_BA] {\data};
\addplot[green, mark size=1,ultra thick, dashed] table[x=t, y=itnr_explicit_BA] {\data};
\addplot[orange, mark size=1,ultra thick] table[x=t, y=itnr_jacobi_BA] {\data};
\addplot[cyan, mark size=1,ultra thick] table[x=t, y=itnr_none_BA] {\data};
%\legend{$\matrixP_{\matrixQ}[\mm_h^n]$, $\matrixP^{\textrm{2D}}$, $\widetilde{\matrixP}_{\matrixQ}[\mm_h^n]$, $\matrixP^{\textrm{jac}}$, None};
\end{axis}
\end{tikzpicture}
}
\caption{$\alpha = \alpha_{\matrixP} = 1$}
\end{subfigure}
\begin{subfigure}{0.32\textwidth}
\scalebox{0.4}{
\begin{tikzpicture}
\pgfplotstableread{plots/mumag4/Alpha/itnr_smallalphaP.dat}{\data}
\begin{axis}[
width = 120mm,
xlabel={Physical time (\si{\nano\second})},
ylabel={Iterations},
xmin=0.0,
xmax=3,
ymin=0,
ymax=1400,
%legend pos=outer north east,
%every axis legend/.append style={nodes={right}}],
%scaled ticks=false, tick label style={/pgf/number format/fixed},
]
\addplot[blue, mark size=1,ultra thick] table[x=t, y=itnr_implicit_SAP] {\data};
\addplot[red, mark size=1,ultra thick] table[x=t, y=itnr_2D_SAP] {\data};
\addplot[green, mark size=1,ultra thick, dashed] table[x=t, y=itnr_explicit_SAP] {\data};
\addplot[orange, mark size=1,ultra thick] table[x=t, y=itnr_jacobi_SAP] {\data};
\addplot[cyan, mark size=1,ultra thick] table[x=t, y=itnr_none_SAP] {\data};
%\legend{$\matrixP^{\textrm{2D}}$, $\widetilde{\matrixP}_{\matrixQ}[\mm_h^n]$, $\matrixP^{\textrm{jac}}$, None};
\end{axis}
\end{tikzpicture}
}
\caption{$\alpha = \alpha_{\matrixP} = 0.02$}
\end{subfigure}
\caption{Experiment of Section~\ref{subsection:mumag4}: GMRES iterations over time.}
\label{fig:mumag4_itnr_dpr}
\end{figure}
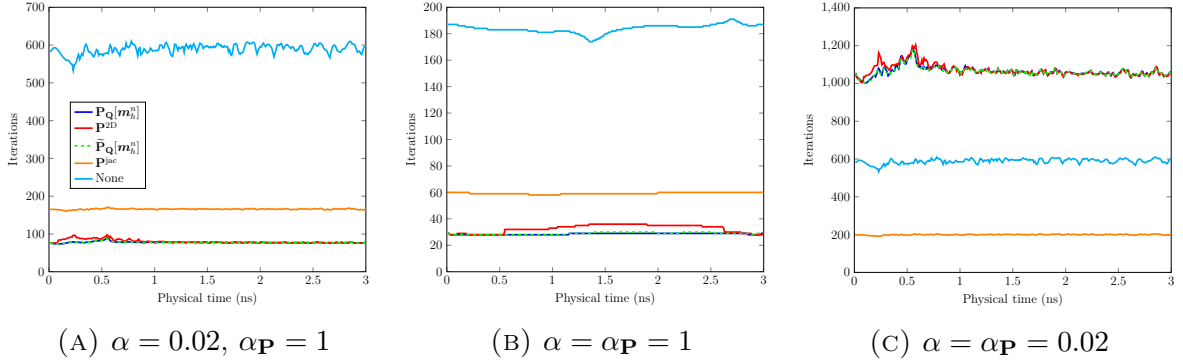
%%%%%%%%%%%%%%%%%%%%

%%%%%%%%%%%%%%%%%%%%%%%%%%%%%%%%
%
%-------------------------------------------------------------------
\subsection{$\boldsymbol\mu$MAG standard problem \#4}\label{subsection:mumag4}
%-------------------------------------------------------------------
We investigate the practical applicability of the different choices for the preconditioner $\matrixP$ from Section~\ref{section:precond} to a physically relevant example. To this end, we consider the $\mu$-MAG standard problem \#4~\cite{mumag}, which simulates the switching of the magnetization in a thin permalloy layer
\begin{align*}
\Omega := ( -250 \si{\nano\meter},250 \si{\nano\meter}) \times ( -62.5 \si{\nano\meter}, 62.5\si{\nano\meter}) \times ( -1.5 \si{\nano\meter},1.5 \si{\nano\meter}).
\end{align*}
The dynamics is described by LLG (stated in physical SI units), which reads
\begin{subequations}\label{eq:LLG_physical}
\begin{align}
\partial_t \MM &= -\gamma_0 \, \MM \times \mathbf{H}_{\textrm{eff}}(\MM) \, + \, \frac{\alpha}{M_s} \, \MM \times \partial_t \MM & &\textrm{ in } \left(0,T\right) \times \Omega =: \Omega_T, \\
\partial_{\mathbf{n}} \MM &= \0 && \textrm{ on } \left(0,T\right) \times \partial \Omega, \\
\MM(0) &= \MM^0 && \textrm{ in } \Omega. \label{eq:LLG_physical3}
\end{align}
Here, $\gamma_0 =2.21 \cdot 10^5\si{\newton}/\si{\ampere}^2$ is the gyromagnetic ratio $M_{\mathrm{s}}:= 8.0\cdot 10^5\si{\ampere}/\si{\meter}$ is the saturation magnetization, and the sought magnetization $\MM$ satisfies $|\MM|=M_s$ a.e.\ in $\Omega_T$. Moreover, we employ the physical end time $T := 3 \si{\nano\second}$. The physical effective field reads
\begin{align}
\mathbf{H}_{\textrm{eff}}(\MM) \, := \, \frac{2A}{\mu_0 M_s^2} \, \Delta \MM \, + \, \ppi(\MM) \, + \, \mathbf{H}_{\textrm{ext}}.
\end{align}
Here, $A := 1.3\cdot 10^{-11} \si{\joule\per\meter}$ is the exchange constant of permalloy, $\mu_0=4 \pi \cdot 10^{-7} \si{\newton/\ampere^2}$ is the permeability of vacuum, $\ppi(\MM)$ consists only of the stray field, and
\begin{align}
\mu_0 \mathbf{H}_{\textrm{ext}} \, := \, (-35.5  , -6.3 , 0 )^T \, \si{\milli\tesla}
\end{align}
\end{subequations}
is a constant applied field corresponding to the second $\mu$-MAG \#4 configuration.

Our numerical simulation is based on the non-dimensional form~\eqref{eq:LLG} of~\eqref{eq:LLG_physical}: With $L = 10^{-9}$ and the temporal and spatial rescaling $t \mapsto (\gamma_0 M_s t)/L$ and $\xx \mapsto \xx/L$, respectively, the dimensionless unknown $\mm := \MM /M_s$ satisfies LLG~\eqref{eq:LLG} in the (not relabelled) domain
\begin{align*}
\Omega := (-250,250) \times (-62.5,62.5) \times (-1.5,1.5),
\end{align*}
and the parameters
\begin{align}\label{eq:ellex_mumag4}
\ellex^2 = \frac{2A}{\mu_0 M_{\mathrm{s}}^2 L^2} \approx 32.3283
\quad \textrm{and} \quad 
\ff := (-28250, -5013.4, 0 )^T.
\end{align}

For time and space discretization, we choose the physical time-step size $\Delta t = 0.1 \si{\pico\second}$ and the physical mesh-size $\Delta \xx = 5\si{\nano\meter}$. With the above rescaling, we use Algorithm~\ref{alg:abtps} with the actual numerical discretization parameters
\begin{align}\label{eq:rescaling}
k = \gamma_0 M_{\mathrm{s}} \Delta t \approx 0.017688, \quad \textrm{and} \quad h = \Delta x / L = 5. 
\end{align}
We employ the corresponding mesh generated by the \texttt{NGS/Py}~\cite{ngs} embedded module {\tt Netgen}, which consists of $17478$ elements and $6073$ nodes. We employ the preconditioners from Section~\ref{section:precond} or do not use preconditioning for the iterative solution of the underlying linear system~\eqref{eq:linearsystem:constrainedQ}. Moreover, we employ the adaptive strategy for $\matrixT_n$ from Section~\ref{section:idealaxis}.

%-------------------------------------------------------------------
\subsubsection{$\mu$-MAG \#4 configuration}\label{subsection:standard_mumag4}
%-------------------------------------------------------------------
 As specified by $\mu$-MAG \#4, we choose $\alpha = 0.02$ for the Gilbert damping constant. For preconditioning, we choose $\alpha_{\matrixP}=1$. In Figure~\ref{fig:mumag4_itnr_dpr}, we plot the required GMRES iterations for different preconditioners over time. No preconditioning (None) requires the most iterations. The theoretical ($\matrixP_{\matrixQ}[\mm_h^n]$), the stationary ($\matrixP^{\textrm{2D}}$), the practical ($\widetilde{\matrixP}_{\matrixQ}[\mm_h^n]$), and the Jacobi preconditioner ($\matrixP^{\textrm{jac}}$) significantly improve the iteration number. Here, $\matrixP_{\matrixQ}[\mm_h^n]$, $\matrixP^{\textrm{2D}}$, and $\widetilde{\matrixP}_{\matrixQ}[\mm_h^n]$ require the fewest iterations and $\widetilde{\matrixP}_{\matrixQ}[\mm_h^n]$ performs even slightly better than $\matrixP_{\matrixQ}[\mm_h^n]$ and $\matrixP^{\textrm{2D}}$. Moreover, recall that $\widetilde{\matrixP}_{\matrixQ}[\mm_h^n]$ and $\matrixP^{\textrm{2D}}$ have similar computational complexity, while $\matrixP_{\matrixQ}[\mm_h^n]$ has to be rebuilt at every time-step; cf. Remark~\ref{remark:complexity}. This experiment suggests either the stationary preconditioner $\matrixP^{\textrm{2D}}$ or the practical preconditioner $\widetilde{\matrixP}_{\matrixQ}[\mm_h^n]$. In the following three subsections, we extend this experiment to discuss the roles of $\alpha$ (see Section~\ref{subsection:alpha}), $\matrixT_n$ (see Section~\ref{subsection:mumag4_axis}), and $\alpha_{\matrixP}$ (see Section~\ref{subsection:choosealphaP}).

%-------------------------------------------------------------------
\subsubsection{The impact of $\alpha$}\label{subsection:alpha}
%-------------------------------------------------------------------
We repeat the experiment from Section~\ref{subsection:standard_mumag4} with $\alpha = \alpha_{\matrixP} = 1$. The new setting still simulates the switching dynamics of $\mu$-MAG \#4, but with a bigger Gilbert damping constant $\alpha$. In Figure~\ref{fig:mumag4_itnr_dpr}, we plot the required GMRES iterations for different preconditioners over time. Compared to the original Figure~\ref{subsection:standard_mumag4} with $\alpha = 0.02$ and $\alpha_{\matrixP} = 1$, all approaches require less iterations. 
This is in accordance with the results from Section~\ref{section:precond}, where larger $\alpha$ leads to better contraction of the residuals; cf., e.g.,~\eqref{eq:estimate:theoretical_1}.

%%%%%%%%%%%%%%%%%%%%%%%%%%%%%%
%%%%%%%%%%%%%%%%%%%%
\begin{figure}
\begin{tikzpicture}
\pgfplotstableread{plots/mumag4/Alpha/average_alphaP.dat}{\data}
\begin{axis}[
height = 60mm,
width = 150mm,
xlabel={$\alpha_\matrixP$},
ylabel={Iterations},
xmin=0.0,
xmax=4.0,
ymin=0,
ymax=300,
legend pos=north east,
every axis legend/.append style={nodes={right}}],
%scaled ticks=false, tick label style={/pgf/number format/fixed},
]
\addplot[blue, mark=x, mark size=3,ultra thick] table[x=alphaP, y=twoD] {\data};
\addplot[red, mark=x, mark size=3,ultra thick] table[x=alphaP, y=explicit] {\data};
\legend{$\matrixP_{\matrixQ}^{\textrm{2D}}$,$\widetilde{\matrixP}[\mm_h^n]$};
\end{axis}
\end{tikzpicture}
\caption{Experiment of Section~\ref{subsection:choosealphaP}: Average number of GMRES iterations ($\alpha = 0.02$).}
\label{fig:mumag4_alphaP}
\end{figure}
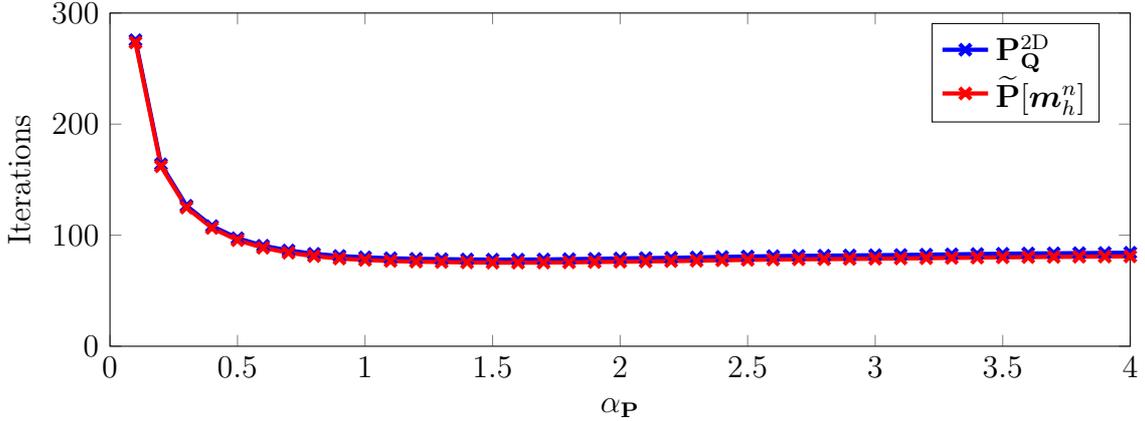
%%%%%%%%%%%%%%%%%%%%

%%%%%%%%%%%%%%%%%%%%%%%%%%%%%%
%
%-------------------------------------------------------------------
\subsubsection{The impact of $\alpha_{\matrixP}$}\label{subsection:choosealphaP}
%-------------------------------------------------------------------
First, we repeat the experiment with $\alpha_{\matrixP} = \alpha = 0.02$. Unlike Section~\ref{subsection:standard_mumag4} and Section~\ref{subsection:alpha}, we observe the following: In Figure~\ref{fig:mumag4_itnr_dpr}, we plot the required GMRES iterations for different preconditioners over time. The Jacobi-preconditioner ($\matrixP^{\textrm{jac}}$) requires less iterations than no precondioning (None). The theoretical ($\matrixP_{\matrixQ}[\mm_h^n]$), the stationary ($\matrixP^{\textrm{2D}}$), and the practical preconditioner ($\widetilde{\matrixP}_{\matrixQ}[\mm_h^n]$) fail completely. All three require significantly more GMRES iterations than no preconditioning (None). This might be due to the skew-symmetric part $\matrixS[\mm_h^n]$ in the (yet unconstrained system) matrix $\matrixA_k[\mm_h^n]$ from~\eqref{eq:systemmatrix}. Here, $\matrixS[\mm_h^n]$ is similar to a mass matrix, but unlike $\MM_k[\mm_h^n]$, it lacks the factor $\alpha$. To empirically determine a good choice for $\alpha_{\matrixP}$, we repeat our current experiment with the fixed preconditioners $\matrixP^{\textrm{2D}}$ and $\widetilde{\matrixP}_{\matrixQ}[\mm_h^n]$, and vary $\alpha_{\matrixP}$. In Figure~\ref{fig:mumag4_alphaP}, we plot the average number of the required GMRES iterations. As already observed in Figure~\ref{fig:mumag4_itnr_dpr}, bigger values of $\alpha_{\matrixP}$ result in significantly less iterations. However, $\alpha_{\matrixP}$ bigger than $1$, has little to no effect. In conclusion, we suggest to always choose $\alpha_{\matrixP} = 1$.

\vspace*{-1mm}
%-------------------------------------------------------------------
\subsubsection{Adaptive vs.\ fixed $\matrixT_n$}\label{subsection:mumag4_axis}
%-------------------------------------------------------------------
%%%%%%%%%%%%%%%%%%%%%%%%%%%%%%
%%%%%%%%%%%%%%%%%%%%
\begin{figure}
\hspace{-0.85cm}\begin{tikzpicture}
\pgfplotstableread{plots/mumag4/Adaptive/choice1.dat}{\choiceA}
\pgfplotstableread{plots/mumag4/Adaptive/choice2.dat}{\choiceB}
\pgfplotstableread{plots/mumag4/Adaptive/choice3.dat}{\choiceC}
\pgfplotstableread{plots/mumag4/Adaptive/choice4.dat}{\choiceD}
\pgfplotstableread{plots/mumag4/Adaptive/choice5.dat}{\choiceE}
\pgfplotstableread{plots/mumag4/Adaptive/choice6.dat}{\choiceF}
\begin{axis}[
height = 20mm,
width = 150mm,
xmin=0,
xmax=3,
xtick=\empty,
ymin=1,
ymax=1,
ytick=\empty,
%yticks={1}
%yticklabels={,},
%ylabel={},
%yticklabel style={/pgf/number format/precision=3},
%legend pos=outer north east,
%every axis legend/.append style={nodes={right}},
]
\addplot[only marks,red, mark size=1,thick] table[x=t1, y expr={1}] {\choiceA};
\addplot[only marks,green, mark size=1,thick] table[x=t2, y expr={1}] {\choiceB};
\addplot[only marks,orange, mark size=1,thick] table[x=t3, y expr={1}] {\choiceC};
\addplot[only marks,cyan, mark size=1,thick] table[x=t4, y expr={1}] {\choiceD};
\addplot[only marks,violet, mark size=1,thick] table[x=t5, y expr={1}] {\choiceE};
\addplot[only marks,brown, mark size=1,thick] table[x=t6, y expr={1}] {\choiceF};
\end{axis}
\end{tikzpicture}
\begin{tikzpicture}
\pgfplotstableread{plots/mumag4/Adaptive/gamma_and_ds.dat}{\data}
\begin{axis}[
height = 65mm,
width = 150mm,
xlabel={Phyisical time (\si{\nano\second})},
ylabel={},
xmin=0.0,
xmax=3,
ymin=0,
ymax=2,
legend pos=outer north east,
every axis legend/.append style={nodes={right}}],
%scaled ticks=false, tick label style={/pgf/number format/fixed},
]
\addplot[blue, mark size=1 ,ultra thick, dashed] table[x=t, y=gamma] {\data};
\addplot[red, mark size=1,thick] table[x=t, y=gamma_maxx] {\data};
\addplot[green, mark size=1,thick] table[x=t, y=gamma_maxy] {\data};
\addplot[orange, mark size=1,thick] table[x=t, y=gamma_maxz] {\data};
\addplot[cyan, mark size=1,thick] table[x=t, y=gamma_minx] {\data};
\addplot[violet, mark size=1,thick] table[x=t, y=gamma_miny] {\data};
\addplot[brown, mark size=1,thick] table[x=t, y=gamma_minz] {\data};
\legend{$d_{\textrm{adapt}}$, $d_1^+$, $d_2^+$, $d_3^+$, $d_1^-$, $d_2^-$, $d_3^-$};
\end{axis}
\end{tikzpicture}
\caption{Experiment of Section~\ref{subsection:mumag4_axis}: Evolution of $d_{\ell}^\star$ with $\ell \in \{1,2,3\}$ and $\star \in \{+,-\}$, and $d_{\textrm{adapt}} := \max_{\ell \in \{1,2,3\}} \{ d_{\ell}^+ , d_{\ell}^- \}$ over time ($\alpha = 0.02$, $\alpha_{\matrixP} = 1$). The above line indicates the current state of $d_{\textrm{adapt}}$.}
\label{fig:gamma_and_ds}
\end{figure}
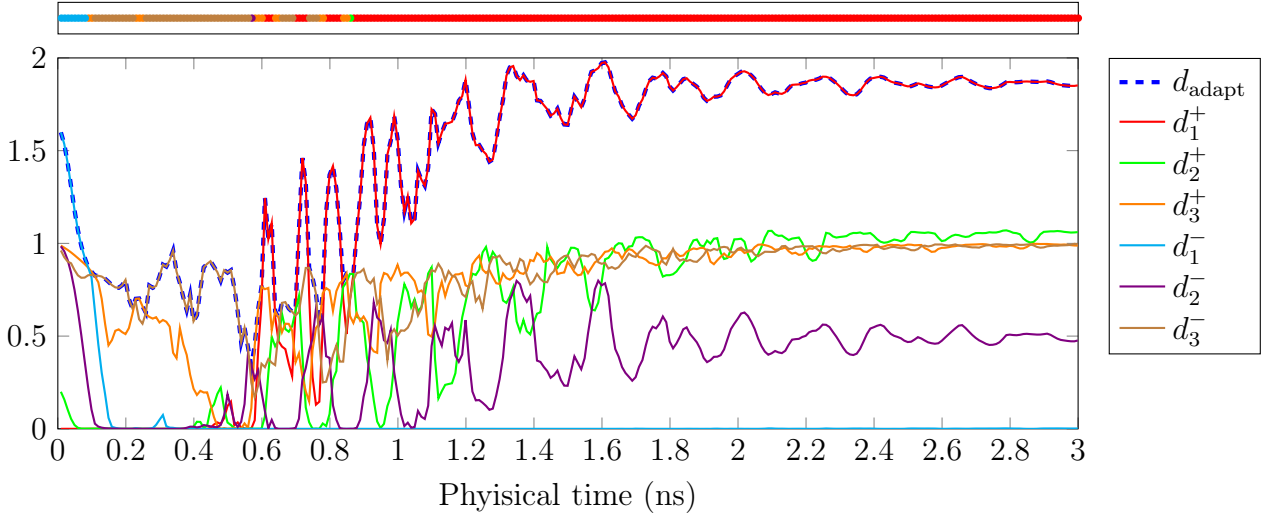
%%%%%%%%%%%%%%%%%%%%

%%%%%%%%%%%%%%%%%%%%
\begin{figure}
\begin{tikzpicture}
\pgfplotstableread{plots/mumag4/Adaptive/itnr_axis.dat}{\data}
\begin{axis}[
height = 100mm,
width = 150mm,
xlabel={Physical time (\si{\nano\second})},
ylabel={Iterations},
xmin=0.0,
xmax=3,
ymin=70,
ymax=200,
legend pos= north east,
every axis legend/.append style={nodes={right}}],
%scaled ticks=false, tick label style={/pgf/number format/fixed},
]
\addplot[only marks, mark size=0.2, blue,ultra thick] table[x=t, y=itnr_adaptive] {\data};
\addplot[red, mark size=1,thick, dashed] table[x=t, y=itnr_T1plus] {\data};
\addplot[green, mark size=1,thick, dashed] table[x=t, y=itnr_T2plus] {\data};
\addplot[orange, mark size=1,thick] table[x=t, y=itnr_T3plus] {\data};
\addplot[cyan, mark size=1,thick, dashed] table[x=t, y=itnr_T1minus] {\data};
\addplot[violet, mark size=1,thick, dashed] table[x=t, y=itnr_T2minus] {\data};
\addplot[brown, mark size=1,thick] table[x=t, y=itnr_T3minus] {\data};
\legend{Adaptive, $\matrixT_1^+$, $\matrixT_2^+$, $\matrixT_3^+$, $\matrixT_1^-$, $\matrixT_2^-$, $\matrixT_3^-$};
\end{axis}
\end{tikzpicture}
\caption{Experiment of Section~\ref{subsection:mumag4_axis}: GMRES iterations for the stationary preconditioner $\matrixP^{\textrm{2D}}$ over time ($\alpha = 0.02$, $\alpha_{\matrixP} = 1$).}
\label{fig:mumag4_choices_itnr}
\end{figure}
%%%%%%%%%%%%%%%%%%%%

%%%%%%%%%%%%%%%%%%%%
\begin{figure}
\begin{tikzpicture}
\pgfplotstableread{plots/mumag4/Adaptive/itnr_axis_EXPL.dat}{\data}
\begin{axis}[
height = 65mm,
width = 120mm,
xlabel={Physical time (\si{\nano\second})},
ylabel={Iterations},
xmin=0.0,
xmax=3,
ymin=70,
ymax=90,
legend pos=outer north east,
every axis legend/.append style={nodes={right}}],
%scaled ticks=false, tick label style={/pgf/number format/fixed},
]
\addplot[only marks, mark size=0.2, blue,ultra thick] table[x=t, y=itnr_adaptive] {\data};
\addplot[red, mark size=1,thick] table[x=t, y=itnr_T1plus] {\data};
\addplot[green, mark size=1,thick] table[x=t, y=itnr_T2plus] {\data};
\addplot[orange, mark size=1,thick] table[x=t, y=itnr_T3plus] {\data};
\addplot[cyan, mark size=1,thick] table[x=t, y=itnr_T1minus] {\data};
\addplot[violet, mark size=1,thick] table[x=t, y=itnr_T2minus] {\data};
\addplot[brown, mark size=1,thick] table[x=t, y=itnr_T3minus] {\data};
\legend{Adaptive, $\matrixT_1^+$, $\matrixT_2^+$, $\matrixT_3^+$, $\matrixT_1^-$, $\matrixT_2^-$, $\matrixT_3^-$};
\end{axis}
\end{tikzpicture}
\caption{Experiment of Section~\ref{subsection:mumag4_axis}: GMRES iterations for the practical preconditioner $\widetilde{\matrixP}_{\matrixQ}[\mm_h^n]$ over time ($\alpha = 0.02$, $\alpha_{\matrixP} = 1$).}
\label{fig:mumag4_choices_itnr_EXPL}
\end{figure}
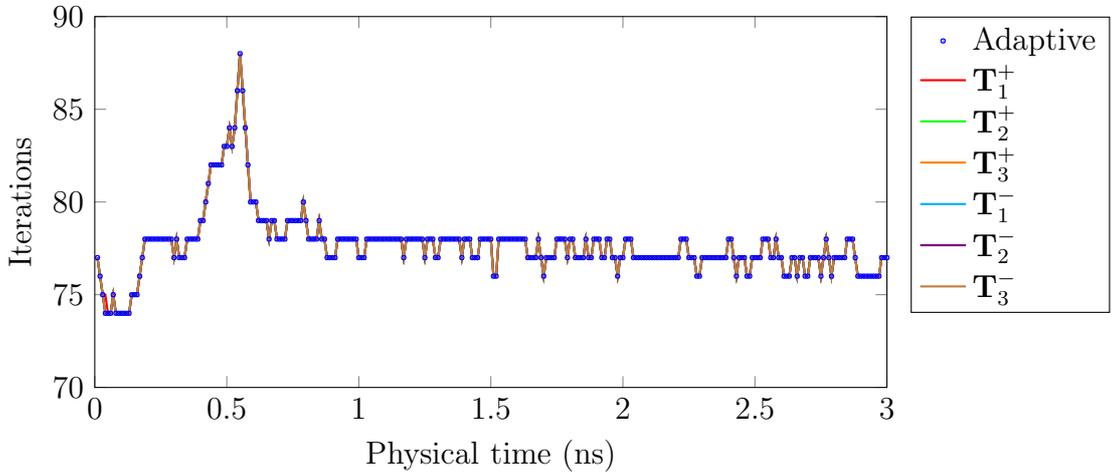
%%%%%%%%%%%%%%%%%%%%

%%%%%%%%%%%%%%%%%%%%%%%%%%%%%%
We extend the experiment from Section~\ref{subsection:standard_mumag4} to discuss the impact of the adaptive strategy for $\matrixT_n \in \R^{3\times3}$ from Section~\ref{section:idealaxis}. In Figure~\ref{fig:gamma_and_ds}, we plot the evolution of $d_{\ell}^{\star} \in [0,2]$ from~\eqref{eq:define_ds}, where $\ell \in \{1,2,3\}$ and $\star \in \{ + , - \}$ and $d_{\textrm{adapt}} := \max_{\ell \in \{1,2,3\} } \{ d_{\ell}^+ , d_{\ell}^- \}$. Recall from~\eqref{eq:gamma_dl} that we can choose $\gamma = d_{\textrm{adapt}}$ for adaptive $\matrixT_n$ and $\gamma = d_{\ell}^{\star}$ for fixed $\matrixT_n := \matrixT_{\ell}^{\star}$. For adaptive $\matrixT_n$, we always have in our example that $1 + (\matrixT_n\mm_h^n(z))_3 \geq \gamma > 0$, i.e., Corollary~\ref{corollary:stationaryprecond} and Theorem~\ref{theorem:practical_convergence}~{\rm (ii)} apply.

In Figure~\ref{fig:mumag4_choices_itnr}, we consider the stationary preconditioner $\matrixP^{\textrm{2D}}$. We plot the evolution of the GMRES iteration numbers with the corresponding $\matrixT_n$ (Adaptive) as well as fixed $\matrixT_n := \matrixT_{\ell}^{\star}$ from~\eqref{eq:define_Ts}, where $\ell \in \{1,2,3\}$ and $\star \in \{ + , - \}$. In Figure~\ref{fig:mumag4_choices_itnr_EXPL}, we repeat this experiment with the practical preconditioner $\widetilde{\matrixP}_{\matrixQ}[\mm_h^n]$ instead of $\matrixP^{\textrm{2D}}$. 

Adaptive $\matrixT_n$ is not always the best choice, however, it avoids the increased iteration number of a fixed $\matrixT_n$; see Figure~\ref{fig:mumag4_choices_itnr}. Yet, for the relation of the iteration number of fixed $\matrixT_n$ and the corresponding $d_{\ell}^{\star}$, the picture is not complete: In Figure~\ref{fig:mumag4_choices_itnr_EXPL}, all options appear to be equally good, even though, e.g., for fixed $\matrixT_n := \matrixT_1^-$, it holds that $d_1^-\approx 0$ most of the time; see Figure~\ref{fig:gamma_and_ds}. 
In conclusion, our experiment suggests to use adaptive $\matrixT_n$. However, a full understanding of the effect of the choice of $\matrixT_n$ might require further work.

%%%%%%%%%%%%%%%%%%%%%%%%%%%%%%%%
%%%%%%%%%%%%%%%%%%%%
\begin{figure}

\centering

\begin{subfigure}{0.32\textwidth}
\scalebox{0.4}{
\begin{tikzpicture}
\pgfplotstableread{plots/mumag5/Alpha/itnr.dat}{\data}
\begin{axis}[
width = 120mm,
xlabel={Physical time (\si{\nano\second})},
ylabel={Iterations},
xmin=0.0,
xmax=8,
ymin=0,
ymax=400,
%legend pos=outer north east,
legend style={at={(axis cs:0.5,140)},anchor=south west,nodes={right}},
every axis legend/.append style={nodes={right}}],
scaled ticks=false, tick label style={/pgf/number format/fixed},
]
\addplot[blue, mark size=1,ultra thick] table[x=t, y=itnr_implicit] {\data};
\addplot[red, mark size=1,ultra thick] table[x=t, y=itnr_2D] {\data};
\addplot[green, mark size=1,ultra thick, dashed] table[x=t, y=itnr_explicit] {\data};
\addplot[orange, mark size=1,ultra thick] table[x=t, y=itnr_jacobi] {\data};
\addplot[cyan, mark size=1,ultra thick] table[x=t, y=itnr_none] {\data};
\legend{$\matrixP_{\matrixQ}[\mm_h^n]$,$\matrixP^{\textrm{2D}}$, $\widetilde{\matrixP}_{\matrixQ}[\mm_h^n]$, $\matrixP^{\textrm{jac}}$, None};
\end{axis}
\end{tikzpicture}
}
\caption{$\alpha = 0.1$, $\alpha_{\matrixP} = 1$}
%\label{fig:mumag5_itnr}
\end{subfigure}
\begin{subfigure}{0.32\textwidth}
\scalebox{0.4}{
\begin{tikzpicture}
\pgfplotstableread{plots/mumag5/Alpha/itnr_bigalpha.dat}{\data}
\begin{axis}[
width = 120mm,
xlabel={Physical time (\si{\nano\second})},
ylabel={Iterations},
xmin=0.0,
xmax=8,
ymin=0,
ymax=400,
legend style={at={(axis cs:0.5,140)},anchor=south west,nodes={right}},
]
\addplot[blue, mark size=1,ultra thick] table[x=t, y=itnr_implicit_BA] {\data};
\addplot[red, mark size=1,ultra thick] table[x=t, y=itnr_2D_BA] {\data};
\addplot[green, mark size=1,ultra thick, dashed] table[x=t, y=itnr_explicit_BA] {\data};
\addplot[orange, mark size=1,ultra thick] table[x=t, y=itnr_jacobi_BA] {\data};
\addplot[cyan, mark size=1,ultra thick] table[x=t, y=itnr_none_BA] {\data};
%\legend{$\matrixP^{\textrm{2D}}$, $\widetilde{\matrixP}_{\matrixQ}[\mm_h^n]$, $\matrixP^{\textrm{jac}}$, None};
\end{axis}
\end{tikzpicture}
}
\caption{$\alpha = \alpha_{\matrixP} = 1$}
\end{subfigure}
\begin{subfigure}{0.32\textwidth}
\scalebox{0.4}{
\begin{tikzpicture}
\pgfplotstableread{plots/mumag5/Alpha/itnr_smallalphaP.dat}{\data}
\begin{axis}[
width = 120mm,
xlabel={Physical time (\si{\nano\second})},
ylabel={Iterations},
xmin=0.0,
xmax=8,
ymin=0,
ymax=400,
%legend pos=outer north east,
%every axis legend/.append style={nodes={right}}],
%scaled ticks=false, tick label style={/pgf/number format/fixed},
]
\addplot[blue, mark size=1,ultra thick] table[x=t, y=itnr_implicit_SAP] {\data};
\addplot[red, mark size=1,ultra thick] table[x=t, y=itnr_2D_SAP] {\data};
\addplot[green, mark size=1,ultra thick, dashed] table[x=t, y=itnr_explicit_SAP] {\data};
\addplot[orange, mark size=1,ultra thick] table[x=t, y=itnr_jacobi_SAP] {\data};
\addplot[cyan, mark size=1,ultra thick] table[x=t, y=itnr_none_SAP] {\data};
%\legend{$\matrixP^{\textrm{2D}}$, $\widetilde{\matrixP}_{\matrixQ}[\mm_h^n]$, $\matrixP^{\textrm{jac}}$, None};
\end{axis}
\end{tikzpicture}
}
\caption{$\alpha = \alpha_{\matrixP} = 0.1$}
\end{subfigure}
\caption{Experiment of Section~\ref{subsection:mumag5}: GMRES iterations over time.}
\label{fig:mumag5_itnr_variants}
\end{figure}
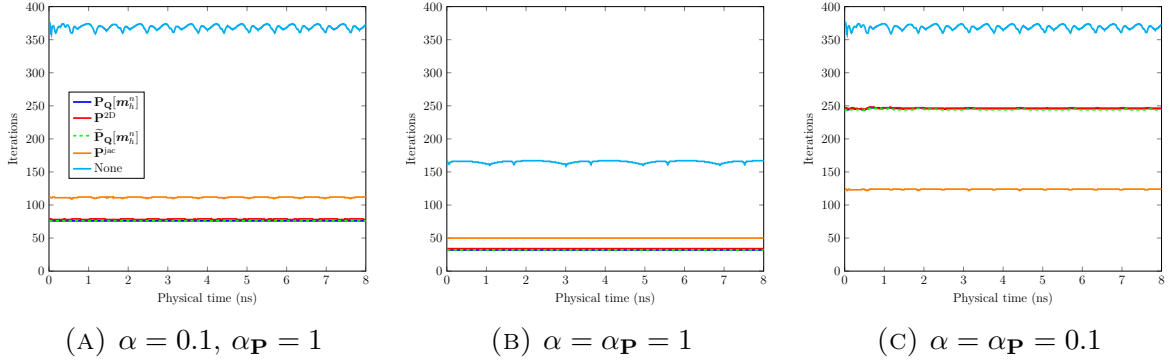
%%%%%%%%%%%%%%%%%%%%

%%%%%%%%%%%%%%%%%%%%
\begin{figure}
\begin{tikzpicture}
\pgfplotstableread{plots/mumag5/Adaptive/itnr_axis.dat}{\data}
\begin{axis}[
width = 100mm,
xlabel={Physical time (\si{\nano\second})},
ylabel={Iterations},
xmin=0.0,
xmax=8,
ymin=75,
ymax=105,
legend pos=outer north east,
every axis legend/.append style={nodes={right}}],
%scaled ticks=false, tick label style={/pgf/number format/fixed},
]
\addplot[only marks, mark size=0.2, blue,ultra thick] table[x=t, y=itnr_adaptive] {\data};
\addplot[red, mark size=1,thick, dashed] table[x=t, y=itnr_T1plus] {\data};
\addplot[green, mark size=1,thick, dashed] table[x=t, y=itnr_T2plus] {\data};
\addplot[orange, mark size=1,thick, dashed] table[x=t, y=itnr_T3plus] {\data};
\addplot[cyan, mark size=1,thick, dashed] table[x=t, y=itnr_T1minus] {\data};
\addplot[violet, mark size=1,thick, dashed] table[x=t, y=itnr_T2minus] {\data};
\addplot[brown, mark size=1,thick,dashed] table[x=t, y=itnr_T3minus] {\data};
\legend{Adaptive, $\matrixT_1^+$, $\matrixT_2^+$, $\matrixT_3^+$, $\matrixT_1^-$, $\matrixT_2^-$, $\matrixT_3^-$};
\end{axis}
\end{tikzpicture}
\caption{Experiment of Section~\ref{subsection:mumag5}: GMRES iterations for the stationary preconditioner $\matrixP^{\textrm{2D}}$ over time ($\alpha = 0.1$, $\alpha_{\matrixP} = 1$).}
\label{fig:mumag5_choices_itnr}
\end{figure}
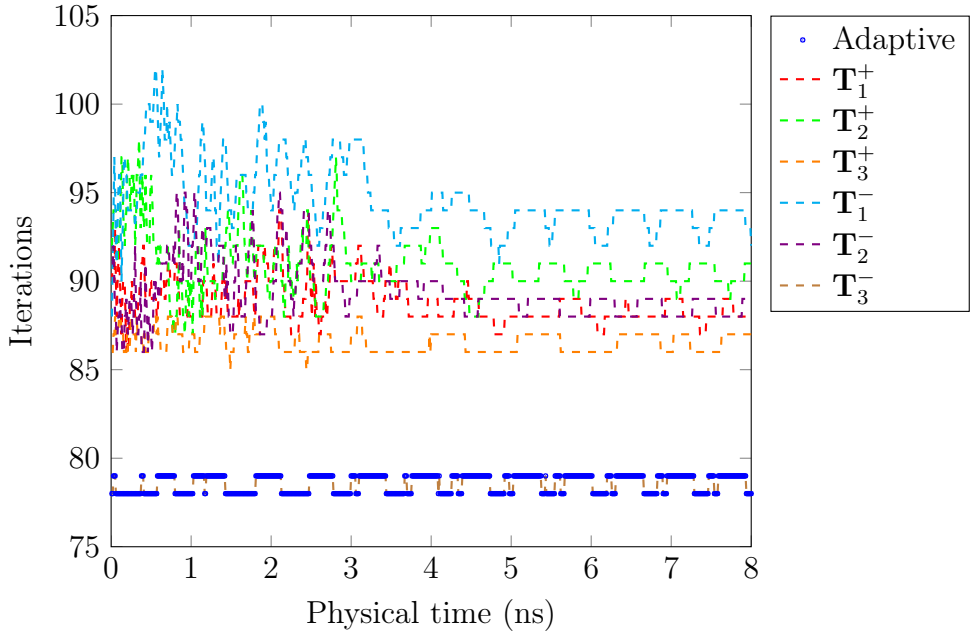
%%%%%%%%%%%%%%%%%%%%

%%%%%%%%%%%%%%%%%%%%%%%%%%%%%%%%
%
%-------------------------------------------------------------------
\subsection{$\boldsymbol\mu$MAG standard problem \#5 }\label{subsection:mumag5}
%-------------------------------------------------------------------
%We consider the $\mu$-MAG standard problem~\#5~\cite{mumag}. 
As for $\mu$-MAG \#4 in Section~\ref{subsection:mumag4}, the the $\mu$-MAG standard problem~\#5~\cite{mumag} relies on a physical formulation similar to~\eqref{eq:LLG_physical} with the end time $T = 8 \si{\nano\second}$ and the domain
%\begin{align*}
$\Omega := (-50 \si{\nano\meter},50 \si{\nano\meter}) \times (-50\si{\nano\meter},50\si{\nano\meter}) \times (-5\si{\nano\meter},5\si{\nano\meter}).$
%\end{align*}

The simulation is based on the non-dimensional form~\eqref{eq:LLG}, where we employ the same temporal and spatial rescaling as in Section~\ref{subsection:mumag4} with the saturation magnetization $M_{\mathrm{s}}:= 8.0\cdot 10^5\si{\ampere}/\si{\meter}$. In particular, we compute $\ellex^2$ as in~\eqref{eq:ellex_mumag4}, where $A := 1.3\cdot 10^{-11} \si{\joule\per\meter}$, $\mu_0=4 \pi \cdot 10^{-7} \si{\newton/\ampere^2}$, and $\alpha := 0.1$.
In extension to $\mu$Mag \#4, the $\mu$-MAG standard problem \#5 involves LLG~\eqref{eq:LLG} with an extended effective field 
\begin{align*}
\heff(\mm) := \ellex^2 \Delta \mm  + \ppi(\mm) + \ff + \mm \times ( \uu \cdot \nabla ) \mm + \beta ( \uu \cdot \nabla ) \mm,
\end{align*}
where the additional term is the Zhang-Li spin-torque term~\cite{ZL04,TNM+05}. Here, $\ff := \0$,  $\ppi(\mm)$ is the stray field, $\beta := 0.05$ is the constant of non-adiabacity, and $\uu := 1/(\gamma_0 M_s L) (72.17,0,0)^T$ is the spin current with the gyromagnetic ratio $\gamma_0 =2.21 \cdot 10^5\si{\newton}/\si{\ampere}^2$ and the scaling parameter $L = 10^{-9}$.

For time and space discretization, we choose the physical time-step size $\Delta t = 0.1 \si{\pico\second}$ and the physical mesh-size $\Delta x = 3\si{\nano\meter}$. The corresponding $k$ and $h$ are obtained as in~\eqref{eq:rescaling}. The corresponding mesh generated by the \texttt{NGS/Py}~\cite{ngs} embedded module {\tt Netgen} has $25 666$ elements and $5 915$ nodes. We employ our \texttt{C++}-code, the first order tangent plane scheme (see~\eqref{eq:TPS1}). GMRES is restarted every 200 iterations and an iteration tolerance $\varepsilon = 10^{-14}$.
We perform the same numerical experiments as for $\mu$MAG \#4. The results qualitatively agree with those of Section~\ref{subsection:mumag4}; see in Figure~\ref{fig:mumag5_itnr_variants} and Figure~\ref{fig:mumag5_choices_itnr}.

%%%%%%%%%%%%%%%%%%%%%%%%%%%%%%%%%%%%%%%%%%%%%%%%%%%%%%%%%%%%%%%%%%%%
\section{Proof of main results} \label{section:proofs}
%%%%%%%%%%%%%%%%%%%%%%%%%%%%%%%%%%%%%%%%%%%%%%%%%%%%%%%%%%%%%%%%%%%%

%-------------------------------------------------------------------
\subsection{Auxiliary mappings}\label{subsection:auxiliarymappings}
%-------------------------------------------------------------------
Recall the hat functions $(\varphi_i)_{i=1}^N$, the nodal basis $(\ppsi_i)_{i=1}^{2N}$ of $(\SS_h)^2$ from~\eqref{eq:basis2D}, and $(\pphi_i)_{i=1}^{3N}$ of $\courantspace_h = (\SS_h)^3$ from~\eqref{eq:basis3D}. Recall the definitions of $\matrixH[\cdot] \in \R^{3 \times 2}$ from~\eqref{eq:householder} and $\matrixQ[\cdot] \in \R^{3N \times 2N}$ from~\eqref{eq:defQ}. Given $\mmu_h \in \magnetizationset_h$, define the mappings
\begin{subequations}\label{eq:defPPhmh}
\begin{align}\label{eq:defPPhmh1}
\mathbb{P}_h[\mmu_h] &: 
(\SS_h)^2 \rightarrow \tps{\mmu_h} \subsetneqq \courantspace_h : 
\ww_h \mapsto \sum_{i=1}^{N} \big( \matrixT_n \matrixH[\matrixT_n\mmu_h(z_i)] \ww_h(z_i)  \big) \, \varphi_i, \\
\widetilde{\mathbb{P}}_h[\mmu_h] &: 
\R^{2N} \rightarrow \tps{\mmu_h} \subsetneqq \courantspace_h : 
\vectorX \mapsto \sum_{i=1}^{3N} \big( \matrixQ[\mmu_h]\vectorX \big)_i \, \pphi_i, \label{eq:defPPhmh2}
\end{align}
\end{subequations}
their ``transposed'' versions
\begin{subequations}\label{eq:defPPhmh_trans}
\begin{align}
\mathbb{P}_h^T[\mmu_h] &: 
\courantspace_h \rightarrow \big( \SS_h \big)^2 : 
\vv_h \rightarrow 
\sum_{i=1}^{N} \big( \matrixH[\mmu_h(z_i)]^T \matrixT_n \vv_h(z_i) \big) \, \varphi_i \label{eq:defPPhmh_trans1} \\
\widetilde{\mathbb{P}}_h^T[\mmu_h] &: 
\R^{3N} \rightarrow \big(\SS_h\big)^2 :
\vectorX \mapsto \sum_{i=1}^{3N} \big( \matrixQ[\mmu_h]^T \vectorX \big)_i \, \pphi_i, \label{eq:defPPhmh_trans2} 
\end{align}
\end{subequations}
and the compositions
\begin{subequations}\label{eq:deforthtangentspace}
\begin{align}
\!\!\!\!\!\!\!\!  \Ppi_h[\mmu_h] &: 
\courantspace_h \rightarrow \tps{\mmu_h} \subsetneqq \courantspace_h :
\vv_h \mapsto \sum_{i=1}^{N} \big( \matrixT_n \matrixH[\matrixT_n\mmu_h(z_i)] \matrixH[\matrixT_n\mmu_h(z_i)]^T \matrixT_n \vv_h(z_i)  \big) \, \varphi_i, 
\label{eq:deforthtangentspace1} \\
\!\!\!\!\!\!\!\!   \widetilde{\Ppi}_h[\mmu_h] &: 
\R^{3N} \rightarrow \tps{\mmu_h} \subsetneqq \courantspace_h :
\vectorX \mapsto \sum_{i=1}^{3N} \big( \matrixQ[\mmu_h] \matrixQ[\mmu_h]^T \vectorX \big)_i \, \pphi_i.\label{eq:deforthtangentspace2}
\end{align}
\end{subequations}
Note that $\Ppi_h[\mmu_h]$ is the nodewise orthogonal projection onto $\tps{\mmu_h}$. The following lemma discusses the relations of the mappings~\eqref{eq:defPPhmh}--\eqref{eq:deforthtangentspace}. 

\begin{lemma}\label{lemma:representations}
For any $\mmu_h \in \magnetizationset_h$, there hold the following assertions {\rm (i)}--{\rm (v)}:

{\rm (i)} For $\vectorX \in \R^{2N}$ and $\ww_h := \sum_{i=1}^{2N} \vectorX_i \ppsi_i \in (\SS_h)^2$, it holds that $\mathbb{P}_h[\mmu_h] \ww_h = \widetilde{\mathbb{P}}_h[\mmu_h] \vectorX$.

{\rm (ii)} For $\vectorY \in \R^{3N}$ and $\vv_h := \sum_{i=1}^{3N} \vectorY_i \pphi_i$, it holds that $\Ppi_h[\mmu_h] \vv_h = \widetilde{\Ppi}_h[\mmu_h] \vectorY$.

{\rm (iii)} For $\vv_h \in \courantspace_h$, it holds that $\mathbb{P}_h[\mmu_h] \circ \mathbb{P}_h^T[\mmu_h] \, \vv_h = \Ppi_h[\mmu_h] \, \vv_h$.

{\rm (iv)} For $\ww_h \in (\SS_h)^2$, it holds that $\mathbb{P}_h^T[\mmu_h] \circ \mathbb{P}_h[\mmu_h] \, \ww_h = \ww_h$.

{\rm (v)} $\mathbb{P}_h[\mmu_h]: (\SS_h)^2 \to \tps{\mmu_h}$ and $\widetilde{\mathbb{P}}_h[\mmu_h]: \R^{2N} \to\tps{\mmu_h}$ are isomorphisms.
\end{lemma}

\begin{proof}
{\rm (i)}--{\rm (ii)} follow by definition. {\rm (iii)}--{\rm (iv)} follow from the block-structure of $\matrixQ[\mmu_h]$, since $\matrixQ[\mmu_h]$ has orthonormal columns. Since $(\SS_h)^2 \cong \R^{2N} \cong \tps{\mmu_h}$, {\rm (iv)} proves that $\mathbb{P}_h[\mmu_h]$ is an isomorphism. Together with~{\rm (i)}, this also proves the statement about $\widetilde{\mathbb{P}}_h[\mmu_h]$. Altogether, this concludes the proof.
\end{proof}

In the following lemma, we prove discrete $\LL^2$-stabilities of the mappings~\eqref{eq:defPPhmh}--\eqref{eq:deforthtangentspace}.

\begin{lemma}\label{lemma:l2norms}
There exists $C>0$, which depends only on $\Cmesh$, such that the following assertions {\rm (i)}--{\rm (vi)} hold true:

{\rm (i)} For any $\mmu_h \in \magnetizationset_h$, it holds that
\begin{align*}
C^{-1} \norm{ \ww_h }{\LL^2(\Omega)} \leq \norm{ \mathbb{P}_h[\mmu_h] \ww_h }{\LL^2(\Omega)} \leq C
\norm{ \ww_h }{\LL^2(\Omega)}
\quad \textrm{for all } \ww_h \in (\SS_h)^2.
\end{align*}

{\rm (ii)} For any $\mmu_h \in \magnetizationset_h$, it holds that
\begin{align*}
\norm{\mathbb{P}_h^T[\mmu_h] \vv_h}{\LL^2(\Omega)} \leq
C \, \norm{\vv_h}{\LL^2(\Omega)}
\quad \textrm{for all } \vv_h \in \courantspace_h.
\end{align*}

{\rm (iii)} For any $\mmu_h \in \magnetizationset_h$, it holds that
\begin{align*}
\norm{ \Ppi_h[\mmu_h] \vv_h }{\LL^2(\Omega)} \leq C
\norm{ \vv_h }{\LL^2(\Omega)}
\quad \textrm{for all } \vv_h \in \courantspace_h.
\end{align*}

{\rm (iv)} For any $\mmu_h \in \magnetizationset_h$, it holds that
\begin{align*}
C^{-1} h^{3/2} |\vectorX| \leq \norm{ \widetilde{\mathbb{P}}_h[\mmu_h] \vectorX }{\LL^{2}(\Omega)} \leq C h^{3/2} |\vectorX| \quad \textrm{for all } \vectorX \in \R^{2N}.
\end{align*}

{\rm (v)} For any $\mmu_h, \nnu_h \in \magnetizationset_h$, it holds that
\begin{align*}
C^{-1} \, \norm{\widetilde{\mathbb{P}}_h[\nnu_h] \vectorX}{\LL^2(\Omega)} \leq \norm{\widetilde{\mathbb{P}}_h[\mmu_h] \vectorX}{\LL^2(\Omega)} \leq
C \, \norm{\widetilde{\mathbb{P}}_h[\nnu_h] \vectorX}{\LL^2(\Omega)}
\quad \textrm{for all } \vectorX \in \R^{2N}.
\end{align*}

{\rm (vi)} For any $\mmu_h, \nnu_h \in \magnetizationset_h$, it holds that
\begin{align*}
\norm{ \widetilde{\mathbb{P}}_h[\mmu_h] \vectorX - \widetilde{\mathbb{P}}_h[\nnu_h] \vectorX }{\LL^{2}(\Omega)} \leq C h^{3/2} \,
\max_{z \in \NN_h} \matrixnorm{\matrixH[\matrixT_n\mmu_h(z)] - \matrixH[\matrixT_n\nnu_h(z)]} \, | \vectorX |
%\matrixnorm{\matrixQ[\mmu_h] - \matrixQ[\nnu_h]} |\vectorX| 
\quad \textrm{for all } \vectorX \in \R^{2N}.
\end{align*}
\end{lemma}

\begin{proof}
Throughout the proof, recall that $\matrixT_n = \matrixT_n^{-1} = \matrixT_n^T$. For the proof of {\rm (i)}, and {\rm (iv)}--{\rm (vi)}, let $\vectorX \in \R^{2N}$ and define $\ww_h := \sum_{i=1}^{2N} \vectorX_i \ppsi_i$. %To prove~{\rm 

The matrices $\matrixH[\mmu_h(z_i)] \in \R^{3 \times 2}$ have orthonormal columns. Lemma~\ref{lemma:representations}~{(i)} yields that
\begin{align*}
& \norm{ \widetilde{\mathbb{P}}_h[\mmu_h] \vectorX }{\LL^{2}(\Omega)}^2
 \stackrel{\phantom{\eqref{eq:scalingargument}}}{=}  
\norm{ \mathbb{P}_h[\mmu_h] \ww_h }{\LL^{2}(\Omega)}^2   
\notag \\ 
& \quad \stackrel{\eqref{eq:scalingargument}}{\simeq} 
h^3 \sum_{i=1}^N | \matrixT_n \matrixH[\matrixT_n\mmu_h(z_i)] \ww_h(z_i) |^2
\stackrel{\phantom{\eqref{eq:scalingargument}}}{=}    
h^3 \sum_{i=1}^N | \ww_h(z_i) |^2 \stackrel{\eqref{eq:scalingargument}}{\simeq} \norm{ \ww_h }{\LL^{2}(\Omega)}^2.
\end{align*}
This proves~{\rm (i)}, and~{\rm (iv)} follows from $\sum_{i=1}^N |\ww(z_i)|^2 = \sum_{i=1}^{2N} |\vectorX_i|^2$. {\rm (v)} is a direct consequence of~{\rm (iv)}. For the proof of~{\rm (vi)}, note that 
\begin{align*}
\norm{ \widetilde{\mathbb{P}}_h[\mmu_h] \vectorX - \widetilde{\mathbb{P}}_h[\nnu_h] \vectorX}{\LL^{2}(\Omega)}^2 
& \stackrel{\eqref{eq:scalingargument}}{\simeq} h^3 \sum_{i=1}^N 
\Big| \, \matrixT_n \big( \, \matrixH[\matrixT_n\mmu_h(z_i)] - \matrixH[\matrixT_n\nnu_h(z_i)] \, \big) \, \ww_h(z_i) \, \Big|^2 \\
& \leq 
h^3 \max_{z \in \NN_h} \matrixnorm{\matrixH[\matrixT_n\mmu_h(z)] - \matrixH[\matrixT_n\nnu_h(z)]}^2 \, | \vectorX |^2 .
\end{align*}
This proves~{\rm (vi)}. For the proof of~{\rm (ii)}--{\rm (iii)}, let $\vv_h \in \courantspace_h$. Since the matrices $\matrixH[\mmu_h(z_i)] \in \R^{3 \times 2}$ have orthonormal columns, we obtain that
\begin{align*}
\norm{ \mathbb{P}_h^T[\mmu_h] \vv_h }{\LL^{2}(\Omega)}^2 
\stackrel{\eqref{eq:scalingargument}}{\simeq} 
h^3 \sum_{i=1}^N |\matrixH[\matrixT_n\mmu_h(z_i)]^T \matrixT_n \vv_h(z_i) |^2 \leq 
h^3 \sum_{i=1}^N |\vv_h(z_i) |^2 \stackrel{\eqref{eq:scalingargument}}{\simeq}  
\norm{\vv_h }{\LL^{2}(\Omega)}^2.
\end{align*}
This proves~{\rm (ii)}. Together with~{\rm (i)}--{\rm (ii)} and Lemma~\ref{lemma:representations}~{\rm (iii)}, this also proves~{\rm (iii)}. Altogether, this concludes the proof.
\end{proof}

In the following lemma, we prove discrete $\HH^1$-stability properties of the mappings~\eqref{eq:defPPhmh}--\eqref{eq:deforthtangentspace}. Note that (in contrast to Lemma~\ref{lemma:representations} and Lemma~\ref{lemma:l2norms}) the following lemma builds on the explicit definition of the Householder matrices~\eqref{eq:householder}.

\begin{lemma}\label{lemma:superfancylemma}
Let $\mmu_h \in \magnetizationset_h$ with $1 + (\mmu_h(z))_3 \geq \gamma > 0$ for all $z \in \NN_h$. Then, there exists $C > 1$, which depends only on $\Cmesh$, such that the following assertions {\rm (i)}--{\rm (iii)} hold true:

{\rm (i)} For all $\ww_h \in \big( \SS_h \big)^2$, it holds that
\begin{align*}
\norm{\nabla \big( \, \mathbb{P}_h[\mmu_h] \ww_h \, \big)}{\LL^2(\Omega)} \leq C 
\gamma^{-2} \, \norm{\nabla \mmu_h}{\LL^{\infty}(\Omega)} \, \norm{\ww_h}{\LL^{2}(\Omega)} + C \, \norm{\nabla \ww_h }{\LL^{2}(\Omega)}.
\end{align*}

{\rm (ii)} For all $\vv_h \in \courantspace_h$, it holds that
\begin{align*}
\norm{\nabla \big( \, \mathbb{P}_h^T[\mmu_h] \vv_h \, \big)}{\LL^2(\Omega)} \leq C 
\gamma^{-2} \, \norm{\nabla \mmu_h}{\LL^{\infty}(\Omega)} \, \norm{\vv_h}{\LL^{2}(\Omega)} + C \, \norm{\nabla \vv_h }{\LL^{2}(\Omega)}.
\end{align*}

{\rm (iii)} For all $\vv_h \in \courantspace_h$, it holds that
\begin{align*}
\norm{\nabla \big( \, \Ppi_h[\mmu_h] \vv_h \, \big)}{\LL^2(\Omega)} \leq C 
\gamma^{-2} \, \norm{\nabla \mmu_h}{\LL^{\infty}(\Omega)} \, \norm{\vv_h}{\LL^{2}(\Omega)} + C \, \norm{\nabla \vv_h }{\LL^{2}(\Omega)}.
\end{align*}
\end{lemma}

\begin{proof}
First, we prove~{\rm (i)}. We split the proof into the following six steps.

{\bf Step 1.} We derive a handier representation of $\mathbb{P}_h[\mmu_h]$. Let $\mu_1,\mu_2,\mu_3 \in \SS_h$ such that $\matrixT_n \mmu_h := ( \mu_1,\mu_2,\mu_3)^T$. Since functions in $\SS_h$ attain their minimum in one of the nodes, we obtain, in particular, that 
\begin{align}\label{eq:transferred_assumption}
1 + \mu_3 = 1 + (\matrixT_n \mmu_h)_3 \geq \gamma > 0 \quad \textrm{in } \Omega .
\end{align}
Hence, we can interpret
\begin{align}\label{eq:matricesR}
\matrixR_1[\matrixT_n \mmu_h] := 
\begin{pmatrix}
1 & 0 \\
0 & 1 \\
- \mu_1 & -\mu_2
\end{pmatrix}, 
\quad
\matrixR_2[\matrixT_n \mmu_h]:= \frac{1}{1+\mu_3}
\begin{pmatrix}
\mu_1^2 & \mu_1 \mu_2 \\
\mu_1 \mu_2 & \mu_2^2 \\
0 & 0
\end{pmatrix}
\end{align}
and
\begin{align}\label{eq:defRs}
\qq[\mmu_h] := \matrixT_n \matrixR_1[\matrixT_n\mmu_h] - \matrixT_n\matrixR_2[\matrixT_n\mmu_h]
\end{align}
as functions $\qq[\mmu_h], \matrixR_1[\matrixT_n \mmu_h], \matrixR_2[\matrixT_n \mmu_h]:\Omega \rightarrow \R^{3 \times 2}$. With the definition of the Householder matrices~\eqref{eq:householder}, an elementary calculation shows that $\matrixT_n\matrixH[\matrixT_n\mmu_h(z_i)] = \qq[\mmu_h(z_i)] $ for all $i = 1,\dots,N$. With $\boldsymbol{\mathcal{I}}_h$ being the vector-valued nodal interpolant onto $\courantspace_h$, we get
\begin{align}\label{eq:defPPhmhasQ}
\mathbb{P}_h[\mmu_h] \ww_h  \stackrel{\eqref{eq:defPPhmh}}{=} \boldsymbol{\mathcal{I}}_h ( \qq[\mmu_h] \ww_h ) \quad \textrm{for all } \ww_h \in (\SS_h)^2.
\end{align}

{\bf Step 2.} We derive preliminary estimates for $\matrixR_1[\matrixT_n\mmu_h]$ and $\matrixR_2[\matrixT_n\mmu_h]$ from~\eqref{eq:matricesR}. To this end, recall that $\matrixT_n = \matrixT_n^{-1} = \matrixT_n^T$. Lemma~\ref{lemma:auxiliary1}~{\rm (i)} yields that 
\begin{subequations}\label{eq:propertiesRs}
\begin{align}\label{eq:propertiesRs_zero}
\norm{\matrixR_1[\matrixT_n\mmu_h]}{\LL^{\infty}(\Omega)} \lesssim 1 
\quad \textrm{and} \quad
\norm{ \matrixR_2[\matrixT_n\mmu_h] }{\LL^{\infty}(\Omega)} \lesssim 1.
\end{align}
Let $k \in \{1,2,3\}$. Lemma~\ref{lemma:auxiliary1}~{\rm (ii)} yields that
\begin{align}\label{eq:propertiesRs_first_1}
\norm{\partial_{k} \big[  \matrixR_1[\matrixT_n\mmu_h] \big]}{\LL^{\infty}(\Omega)} 
\stackrel{\eqref{eq:transferred_assumption}}{\lesssim} \norm{\nabla \matrixT_n\mmu_h}{\LL^\infty(\Omega)} 
= \norm{\matrixT_n \nabla \mmu_h}{\LL^\infty(\Omega)} = \norm{\nabla \mmu_h}{\LL^\infty(\Omega)},
\end{align}
as well as
\begin{align}
\label{eq:propertiesRs_first_2}
\norm{\partial_{k} \big[  \matrixR_2[\matrixT_n\mmu_h] \big]}{\LL^{\infty}(\Omega)} 
\stackrel{\eqref{eq:transferred_assumption}}{\lesssim} \gamma^{-1} \norm{\nabla \matrixT_n \mmu_h}{\LL^\infty(\Omega)} = \gamma^{-1} \norm{\nabla \mmu_h}{\LL^\infty(\Omega)}.
\end{align}
Let $\ell,k \in \{1,2,3\}$. The definition~\eqref{eq:matricesR} and Lemma~\ref{lemma:auxiliary1}~{\rm (iii)} yield that, elementwise, 
\begin{align}\label{eq:propertiesRs_second}
\partial_{\ell} \partial_{k} \big[  \matrixR_1[\matrixT_n\mmu_h] \big] = \0
\end{align}
as well as
\begin{align}
\max_{K \in \TT_h}  \norm{\partial_\ell \partial_{k} \big[ \matrixR_2[\matrixT_n\mmu_h] \big]}{\LL^{\infty}(K)} \stackrel{\eqref{eq:transferred_assumption}}{\lesssim}
\gamma^{-2} \norm{\nabla \matrixT_n\mmu_h}{\LL^\infty(\Omega)}^2 =
\gamma^{-2} \norm{\nabla \mmu_h}{\LL^\infty(\Omega)}^2.
\end{align}
\end{subequations}

{\bf Step 3.} For $\ww_h \in (\SS_h)^2$, we estimate $\norm{ \nabla [ \mathbb{P}_h[\mmu_h] \ww_h ] }{\LL^2(\Omega)}$. To that end, note that $\qq[\mmu_h] \ww_h|_K \in (H^2(K))^3$ for all elements $K$. We exploit the elementwise approximation properties of the nodal interpolant $\boldsymbol{\mathcal{I}}_h$ and obtain that
\begin{align}
& \norm{\nabla \big[ \mathbb{P}_h[\mmu_h] \vv_h \big]}{\LL^{2}(\Omega)} 
\stackrel{\eqref{eq:defPPhmh}}{=}
\norm{\nabla \big[ \, \boldsymbol{\mathcal{I}}_h ( \qq[\mmu_h] \ww_h ) \, \big] }{\LL^{2}(\Omega)}
\notag \\
& 
\stackrel{\phantom{\eqref{eq:defPPhmh}}}{\lesssim}
\norm{\nabla \big[ \,  \qq[\mmu_h] \ww_h \, \big] }{\LL^{2}(\Omega)} 
+
\norm{\nabla \big[ \,  (\boldsymbol{1} - \boldsymbol{\mathcal{I}}_h) ( \qq[\mmu_h] \ww_h ) \, \big] }{\LL^{2}(\Omega)} \notag \\
& \stackrel{\phantom{\eqref{eq:defPPhmh}}}{\lesssim}
\norm{\nabla \big[ \,  \qq[\mmu_h] \ww_h \, \big] }{\LL^{2}(\Omega)} 
+
h \,
\Big( \, 
\sum_{K \in \TT_h}
\norm{D^2 \big[ \, ( \qq[\mmu_h] \ww_h ) \, \big] }{\LL^{2}(K)}^2
\, \Big)^{1/2}
 =: \, T_1 + h \, T_2. \label{eq:T1T2}
\end{align}

{\bf Step 4.} We estimate $T_1$. Let $k \in \{1,2,3\}$. The product rule yields that
\begin{align}
& \partial_{k} \big[   \qq[\mmu_h] \ww_h  \big] 
\stackrel{\phantom{\eqref{eq:defRs}}}{=}
\partial_{k} \big[  \qq[\mmu_h] \big] \ww_h  + \qq[\mmu_h] \partial_{k} \ww_h \notag \\
& \stackrel{\eqref{eq:defRs}}{=}
\partial_{k} \big[  \matrixT_n\matrixR_1[\matrixT_n\mmu_h] \big] \ww_h - 
\partial_{k} \big[  \matrixT_n\matrixR_2[\matrixT_n\mmu_h] \big] \ww_h 
\, + \, 
\matrixT_n\matrixR_1[\matrixT_n\mmu_h] \partial_{k} \ww_h - \matrixT_n\matrixR_2[\matrixT_n\mmu_h] \partial_{k} \ww_h \notag \\
& \stackrel{\phantom{\eqref{eq:defRs}}}{=}
\matrixT_n \partial_{k} \big[ \matrixR_1[\matrixT_n\mmu_h] \big] \ww_h - 
\matrixT_n \partial_{k} \big[ \matrixR_2[\matrixT_n\mmu_h] \big] \ww_h 
\, + \, 
\matrixT_n\matrixR_1[\matrixT_n\mmu_h] \partial_{k} \ww_h 
- \matrixT_n\matrixR_2[\matrixT_n\mmu_h] \partial_{k} \ww_h .
\label{eq:productrule1}
\end{align}
Note that $\gamma\leq 1 + \mu_3 \leq 1 + |\mu_3| \leq 2$ and recall that $\matrixT_n = \matrixT_n^{-1} = \matrixT_n^T$. With the estimates from~\eqref{eq:propertiesRs} and with $1 \leq 2/\gamma$, the latter equation yields that
\begin{align*}%\label{eq:part1}
T_1
& \stackrel{\eqref{eq:T1T2}}{\lesssim}
\gamma^{-1} \, \norm{\nabla \matrixT_n \mmu_h}{\LL^{\infty}(\Omega)} \norm{\ww_h}{\LL^{2}(\Omega)}
 + \norm{\nabla \ww_h}{\LL^{2}(\Omega)} \\
& \stackrel{\phantom{\eqref{eq:T1T2}}}{=}
\gamma^{-1} \, \norm{\nabla \mmu_h}{\LL^{\infty}(\Omega)} \norm{\ww_h}{\LL^{2}(\Omega)}
 + \norm{\nabla \ww_h}{\LL^{2}(\Omega)}. 
\end{align*}

{\bf Step 5.} We estimate $T_2$. Let $\ell,k \in \{1,2,3\}$. Elementwise, it holds that 
\begin{align*}
\partial_{\ell} \matrixT_n \partial_k \matrixR_1[\matrixT_n \mmu_h] = \matrixT_n \partial_{\ell} \partial_k \matrixR_1[\matrixT_n \mmu_h] \stackrel{\eqref{eq:propertiesRs_second}}{=} \0
\end{align*}
as well as $\partial_{\ell} \partial_k \ww_h = \0$. Together with the product rule, this yields that
\begin{align*}
\partial_{\ell} \partial_{k} \big[ \,  \qq[\mmu_h] \ww_h \, \big] 
& \stackrel{\eqref{eq:productrule1}}{=} 
\partial_{k} \big[ \matrixT_n \matrixR_1[\matrixT_n\mmu_h] \big] \partial_{\ell} \ww_h 
- 
\partial_{\ell} \matrixT_n \partial_{k} \big[  \matrixR_2[\matrixT_n\mmu_h] \big] \ww_h
- 
\partial_{k} \big[ \matrixT_n \matrixR_2[\matrixT_n\mmu_h] \big] \partial_{\ell} \ww_h \notag \\
& \qquad + \partial_{\ell} \matrixT_n \big[ \matrixR_1[ \matrixT_n\mmu_h] \big] \partial_k \ww_h 
- \partial_{\ell} \big[ \matrixT_n \matrixR_2[ \matrixT_n \mmu_h] \big] \partial_k \ww_h \notag \\
& \stackrel{\phantom{\eqref{eq:productrule1}}}{=} 
\matrixT_n \partial_{k} \big[ \matrixR_1[\matrixT_n \mmu_h] \big] \partial_{\ell} \ww_h 
- 
\matrixT_n \partial_{\ell} \partial_{k} \big[ \matrixR_2[\matrixT_n \mmu_h] \big] \ww_h
- 
\matrixT_n \partial_{k} \big[ \matrixR_2[ \matrixT_n \mmu_h] \big] \partial_{\ell} \ww_h  \notag \\
& \qquad + \matrixT_n \partial_{\ell} \big[ \matrixR_1[ \matrixT_n \mmu_h] \big] \partial_k \ww_h 
- \matrixT_n \partial_{\ell} \big[ \matrixR_2[ \matrixT_n \mmu_h] \big] \partial_k \ww_h.
\end{align*}
Recall that $\matrixT_n = \matrixT_n^{-1} = \matrixT_n^T$. With~\eqref{eq:propertiesRs} and $1 \leq 2/\gamma$, the latter equation yields that
\begin{align*}
T_2 
& \stackrel{\eqref{eq:T1T2}}{\lesssim} \, \gamma^{-1} \, \norm{\nabla \matrixT_n \mmu_h}{\LL^{\infty}(\Omega)} \, \norm{\nabla \ww_h}{\LL^{2}(\Omega)} + 
\gamma^{-2} \, \norm{\nabla \matrixT_n \mmu_h}{\LL^{\infty}(\Omega)}^2 \, \norm{\ww_h}{\LL^{2}(\Omega)} \\
& \stackrel{\phantom{\eqref{eq:T1T2}}}{=} \, \gamma^{-1} \, \norm{\nabla \mmu_h}{\LL^{\infty}(\Omega)} \, \norm{\nabla \ww_h}{\LL^{2}(\Omega)} + 
\gamma^{-2} \, \norm{\nabla \mmu_h}{\LL^{\infty}(\Omega)}^2 \, \norm{\ww_h}{\LL^{2}(\Omega)}.
\end{align*}

{\bf Step 6.} We combine {\bf Step 3}--{\bf Step 5}. For all $\ww_h \in (\SS_h)^2$, this yields that
\begin{align*}
  \norm{\nabla \big[ \mathbb{P}_h[\mmu_h] \ww_h \big]}{\LL^{2}(\Omega)} 
% \notag \\ & 
& \lesssim 
\, \gamma^{-1} \, 
\norm{\nabla \mmu_h}{\LL^{\infty}(\Omega)} \, 
\norm{\ww_h}{\LL^{2}(\Omega)}
 + \norm{\nabla \ww_h}{\LL^{2}(\Omega)} \notag \\
& \qquad +  h \, \gamma^{-1} \, 
\norm{\nabla \mmu_h}{\LL^{\infty}(\Omega)} \, 
\norm{\nabla \ww_h}{\LL^{2}(\Omega)} + 
h \, \gamma^{-2} \, \norm{\nabla \mmu_h}{\LL^{\infty}(\Omega)}^2 \, 
\norm{\ww_h}{\LL^{2}(\Omega)}.
\end{align*}
With an inverse estimate and $1 \leq 2/\gamma$, the latter equation yields for all $\ww_h \in (\SS_h)^2$ that
\begin{align}\label{eq:coreestimate1}
\!\!\!\! \norm{\nabla \big[ \mathbb{P}_h[\mmu_h] \ww_h \big]}{\LL^{2}(\Omega)} \lesssim 
\gamma^{-2} \, 
\norm{\nabla \mmu_h}{\LL^{\infty}(\Omega)} \, 
\norm{\ww_h}{\LL^{2}(\Omega)}
 + \norm{\nabla \ww_h}{\LL^{2}(\Omega)}.
\end{align}
This concludes the proof of~{\rm (i)}. 

For the proof of~{\rm (ii)}, let $\widetilde{\boldsymbol{\mathcal{I}}}_h$ be the nodal interpolant in 2D. With $\widetilde{\boldsymbol{\mathcal{I}}}_h$ instead of $\boldsymbol{\mathcal{I}}_h$ and $\qq[\mmu_h]^T: \Omega \rightarrow \R^{2 \times 3}$ instead of $\qq[\mmu_h]$, the proof of~{\rm(ii)} follows the lines of {\bf Step 1}--{\bf Step 5}. 
%This proves~{\rm (ii)}.

For the proof of~{\rm (iii)}, let $\vv_h \in \courantspace_h$ and $\ww_h := \mathbb{P}_h^T[\mmu_h] \vv_h \in \courantspace_h$. With Lemma~\ref{lemma:representations}~{\rm (iii)} and Lemma~\ref{lemma:l2norms}~{\rm (ii)}, we get that
\begin{align*}
& \norm{ \nabla [ \Ppi_h[\mmu_h] \vv_h ] }{\LL^2(\Omega)} 
 =
\norm{ \nabla \, \big[ \, \big( \mathbb{P}_h[\mmu_h] \circ \mathbb{P}_h^T[\mmu_h] \big) \, \vv_h \big] }{\LL^2(\Omega)} 
 \stackrel{\phantom{{\rm (ii)}}}{\lesssim}  
 \norm{ \nabla \, \big[ \, \mathbb{P}_h[\mmu_h] \, \ww_h \big] }{\LL^2(\Omega)} 
\\
& \qquad  \stackrel{{\rm (i)}}{\lesssim} 
  \gamma^{-2} \, \norm{\nabla \mmu_h}{\LL^{\infty}(\Omega)} \, \norm{\mathbb{P}_h^T[\mmu_h]\vv_h}{\LL^{2}(\Omega)} + 
 \norm{\nabla \big[ \mathbb{P}_h^T[\mmu_h]\vv_h \big] }{\LL^{2}(\Omega)}  \\
& \qquad \stackrel{\phantom{{\rm (ii)}}}{\lesssim}   
 \gamma^{-2} \norm{\nabla \mmu_h}{\LL^{\infty}(\Omega)} \, \norm{\vv_h}{\LL^{2}(\Omega)} + \norm{\nabla \big[ \mathbb{P}_h^T[\mmu_h]\vv_h \big] }{\LL^{2}(\Omega)} \\
& \qquad \stackrel{{\rm (ii)}}{\lesssim}  
 \gamma^{-2} \, \norm{\nabla \mmu_h}{\LL^{\infty}(\Omega)} \, \norm{\vv_h}{\LL^{2}(\Omega)} + \norm{\nabla \vv_h }{\LL^{2}(\Omega)}.
\end{align*}
This proves~{\rm (iii)} and concludes the proof.
\end{proof}

In the following lemma, we prove a discrete $\HH^1$-continuity of the mapping $\mathbb{P}_{h}(\cdot)$ from~\eqref{eq:defPPhmh}. Unlike Lemma~\ref{lemma:representations} and Lemma~\ref{lemma:l2norms}, the following lemma builds on the explicit definition of the Householder matrices~\eqref{eq:householder}.
\pagebreak

\begin{lemma}\label{lemma:superfancylemma2}
Let $\mmu_h, \nnu_h \in \magnetizationset_h$ with $1 + (\matrixT_n\mmu_h(z))_3 \geq \gamma > 0$ and $1 + (\matrixT_n\nnu_h(z))_3 \geq \gamma > 0$ for all $z \in \NN_h$. Then, there exists $C > 1$, which depends only on $\Cmesh$, such that 
\begin{align}\label{eq:superfancyestimate}
\begin{split}
& \norm{\nabla \big( \, ( \mathbb{P}_h[\mmu_h] - \mathbb{P}_h[\nnu_h] ) \, \ww_h \, \big)}{\LL^2(\Omega)} \\
& \quad \leq 
C \, \gamma^{-1} \, \norm{\nabla \mmu_h - \nabla \nnu_h}{\LL^\infty(\Omega)} \, \norm{\ww_h}{\LL^2(\Omega)} 
+ 
C \, \gamma^{-1} \, \norm{\mmu_h - \nnu_h}{\LL^\infty(\Omega)} \, \norm{\nabla \ww_h}{\LL^2(\Omega)} \\
& \qquad + C \, \gamma^{-3} \,
\big( \, \norm{\nabla \mmu_h}{\LL^\infty(\Omega)} + \norm{\nabla \nnu_h}{\LL^\infty(\Omega)} \, \big) \,
\norm{\mmu_h - \nnu_h}{\LL^{\infty}(\Omega)} \,
\norm{\ww_h}{\LL^2(\Omega)},
\end{split}
\end{align}
for all $\ww_h \in \big( \SS_h \big)^2$.
\end{lemma}

\begin{proof}
We split the proof into the following six steps. 

{\bf Step 1.} With the assumption $1 + (\matrixT_n\mmu_h(z))_3 \geq \gamma > 0$ and $1 + (\matrixT_n\nnu_h(z))_3 \geq \gamma > 0$ for all nodes $z \in \NN_h$, we use the definitions~\eqref{eq:matricesR} of $\matrixR_1(\cdot)$ and $\matrixR_2(\cdot)$ and interpret
\begin{subequations}\label{eq:defqq_12}
\begin{align}
\qq[\mmu_h] & := \matrixT_n\matrixR_1[\matrixT_n\mmu_h] - \matrixT_n\matrixR_2[\matrixT_n\mmu_h] , \\
\qq[\nnu_h] & := \matrixT_n \matrixR_1[\matrixT_n\nnu_h] - \matrixT_n\matrixR_2[\matrixT_n\nnu_h]
\end{align}
\end{subequations}
as functions $\qq[\mmu_h], \qq[\nnu_h]: \Omega \rightarrow \R^{3 \times 2}$. With $\boldsymbol{\mathcal{I}}_h$ being the vector-valued nodal interpolant onto $\courantspace_h$, recall from~\eqref{eq:defPPhmhasQ} that
\begin{align}\label{eq:diffrepresentation}
(\mathbb{P}_h[\mmu_h] - \mathbb{P}_h[\nnu_h]) \, \ww_h  =
\boldsymbol{\mathcal{I}}_h \big( \, (\qq[\mmu_h] - \qq[\nnu_h] ) \,  \ww_h \, \big)
\quad \textrm{for all } \ww_h \in (\SS_h)^2.
\end{align}

{\bf Step 2.} Recall that $\matrixT_n = \matrixT_n^{-1} = \matrixT_n^T$. With the definition~\eqref{eq:matricesR} of $\matrixR_1(\cdot)$, we get that
\begin{subequations}\label{eq:preliminaries_R1_diff}
\begin{align}
\norm{\matrixR_1[\matrixT_n\mmu_h] - \matrixR_1[\matrixT_n\nnu_h]}{\LL^{\infty}(\Omega)} & \lesssim
\norm{\matrixT_n\mmu_h - \matrixT_n\nnu_h}{\LL^{\infty}(\Omega)}
 = \norm{\mmu_h - \nnu_h}{\LL^{\infty}(\Omega)} .
\end{align}
With the product rule, we further get for all $\ell, k \in \{1,2,3\}$ that
\begin{align}
& \norm{\partial_k \matrixR_1[\matrixT_n\mmu_h] - \partial_k \matrixR_1[\matrixT_n\nnu_h]}{\LL^{\infty}(\Omega)}
\lesssim 
 \norm{\nabla \matrixT_n\mmu_h - \nabla \matrixT_n\nnu_h}{\LL^{\infty}(\Omega)} \notag \\
& \quad  = \norm{\matrixT_n \nabla \mmu_h - \matrixT_n \nabla \nnu_h}{\LL^{\infty}(\Omega)} 
  = \norm{\nabla \mmu_h - \nabla \nnu_h}{\LL^{\infty}(\Omega)}
\end{align}
as well as
\begin{align}
\partial_{\ell} \partial_k \matrixR_1[\matrixT_n\mmu_h] \, = \, \0 &=  \partial_{\ell} \partial_k \matrixR_1[\matrixT_n\nnu_h] . \label{eq:secondr1zero}
\end{align}
\end{subequations}
Moreover, define
\begin{subequations}\label{eq:defsigma}
\begin{align}
\begin{split}
\sigma(\matrixT_n \mmu_h, \matrixT_n \nnu_h) 
& := \norm{\nabla \matrixT_n \mmu_h}{\LL^\infty(\Omega)} + \norm{\nabla \matrixT_n \nnu_h}{\LL^\infty(\Omega)} \\
& \,= \norm{\matrixT_n \nabla \mmu_h}{\LL^\infty(\Omega)} + \norm{\matrixT_n \nabla \nnu_h}{\LL^\infty(\Omega)} \\
& \,=
\norm{\nabla \mmu_h}{\LL^\infty(\Omega)} + \norm{\nabla \nnu_h}{\LL^\infty(\Omega)}
= \sigma(\mmu_h,\nnu_h).
\end{split}
\end{align}
Note that an inverse inequality yields that
\begin{align}
h \, \sigma(\mmu_h, \nnu_h) \lesssim \norm{\mmu_h}{\LL^\infty(\Omega)} + \norm{\nnu_h}{\LL^\infty(\Omega)} = 2.
\end{align}
\end{subequations}
Lemma~\ref{lemma:auxiliary2} and the definition~\eqref{eq:matricesR} of $\matrixR_2(\cdot)$ then yield that
\begin{subequations}\label{eq:preliminaries_R2_diff}
\begin{align}
\!\!\!\!\!\!
\norm{\matrixR_2[\matrixT_n \mmu_h] - \matrixR_2[\matrixT_n \nnu_h]}{\LL^{\infty}(\Omega)} & \lesssim 
\gamma^{-1} \, \norm{\matrixT_n\mmu_h - \matrixT_n\nnu_h}{\LL^{\infty}(\Omega)} 
 = \gamma^{-1} \, \norm{\mmu_h - \nnu_h}{\LL^{\infty}(\Omega)} .
\end{align}
For all $\ell, k \in \{1,2,3\}$, we further get that
\begin{align}
& \norm{\partial_k \matrixR_2[\matrixT_n\mmu_h] - \partial_k \matrixR_2[\matrixT_n\nnu_h]}{\LL^{\infty}(\Omega)} \notag \\
& \lesssim \gamma^{-2} \,
\sigma(\matrixT_n\mmu_h,\matrixT_n\nnu_h) \,
\norm{\matrixT_n\mmu_h - \matrixT_n\nnu_h}{\LL^{\infty}(\Omega)} 
+ \gamma^{-1} \,
\norm{\nabla \matrixT_n\mmu_h - \nabla \matrixT_n \nnu_h}{\LL^{\infty}(\Omega)} \notag \\
& \! \stackrel{\eqref{eq:defsigma}}{=} \gamma^{-2} \,
\sigma(\mmu_h,\nnu_h) \,
\norm{\matrixT_n\mmu_h - \matrixT_n\nnu_h}{\LL^{\infty}(\Omega)} 
+ \gamma^{-1} \,
\norm{\matrixT_n \nabla \mmu_h - \matrixT_n \nabla \nnu_h}{\LL^{\infty}(\Omega)} \notag \\
& \! \stackrel{\phantom{\eqref{eq:defsigma}}}{=} \gamma^{-2} \,
\sigma(\mmu_h,\nnu_h) \,
\norm{\mmu_h - \nnu_h}{\LL^{\infty}(\Omega)} 
+ \gamma^{-1} \,
\norm{\nabla \mmu_h - \nabla \nnu_h}{\LL^{\infty}(\Omega)}
\end{align}
as well as
\begin{align}
& \norm{\partial_{\ell} \partial_k \matrixR_2[\matrixT_n\mmu_h] - \partial_{\ell} \partial_k \matrixR_2[\matrixT_n\nnu_h]}{\LL^{\infty}(\Omega)} \notag \\
& \quad \lesssim \gamma^{-3} \,
\sigma(\matrixT_n\mmu_h,\matrixT_n\nnu_h)^2 \,
\norm{\matrixT_n\mmu_h - \matrixT_n\nnu_h}{\LL^{\infty}(\Omega)} \notag \\
& \qquad + \gamma^{-2} \,
\sigma(\matrixT_n\mmu_h,\matrixT_n\nnu_h) \,
\norm{\nabla \matrixT_n\mmu_h - \nabla \matrixT_n\nnu_h}{\LL^{\infty}(\Omega)} \notag \\
& \quad \stackrel{\eqref{eq:defsigma}}{=} \gamma^{-3} \,
\sigma(\mmu_h,\nnu_h)^2 \,
\norm{\mmu_h - \nnu_h}{\LL^{\infty}(\Omega)} 
+ \gamma^{-2} \,
\sigma(\mmu_h,\nnu_h) \,
\norm{\nabla \mmu_h - \nabla \nnu_h}{\LL^{\infty}(\Omega)}.
\end{align}
\end{subequations}

{\bf Step 3.} Let $\ww_h \in (\SS_h)^2$. Standard estimates for the nodal interpolant $\boldsymbol{\mathcal{I}}_h$ yield that
\begin{align*}
& \norm{\nabla \big[ \, ( \mathbb{P}_h[\mmu_h] - \mathbb{P}_h[\nnu_h] ) \, \ww_h \, \big]}{\LL^2(\Omega)} 
\stackrel{\eqref{eq:diffrepresentation}}{=} 
\norm{\nabla  \boldsymbol{\mathcal{I}}_h \big( \, (\qq[\mmu_h] - \qq[\nnu_h] ) \,  \ww_h \, \big)}{\LL^{2}(\Omega)}
\\
& \quad \leq
\norm{\nabla \big( \, (\qq[\mmu_h] - \qq[\nnu_h] ) \,  \ww_h \, \big)}{\LL^{2}(\Omega)}
+
\norm{\nabla (\boldsymbol{1} - \boldsymbol{\mathcal{I}}_h) \big( \, (\qq[\mmu_h] - \qq[\nnu_h] ) \,  \ww_h \, \big)}{\LL^{2}(\Omega)} \\
& \quad \lesssim
\norm{\nabla \big( \, (\qq[\mmu_h] - \qq[\nnu_h] ) \,  \ww_h \, \big)}{\LL^{2}(\Omega)}
+ 
h \, \bigg( \, \sum_{K \in \TT_h} \norm{D^2 \big( \, (\qq[\mmu_h] - \qq[\nnu_h] ) \,  \ww_h \, \big)}{\LL^{2}(K)}^2 \, \bigg)^2 \\
& \quad =:  T_1 + h \, T_2 .
\end{align*}

{\bf Step 4.} We estimate $T_1$. Let $k \in \{ 1,2,3 \}$. With the product rule, we get that
\begin{align*}
& \partial_k \big( \, (\qq[\mmu_h] - \qq[\nnu_h] ) \,  \ww_h \, \big)
 \stackrel{\phantom{\eqref{eq:defqq_12}}}{=} 
\partial_k \big( \, (\qq[\mmu_h] - \qq[\nnu_h] ) \,  \, \big) \ww_h +
(\qq[\mmu_h] - \qq[\nnu_h] ) \,  \partial_k \ww_h \\
&\stackrel{\eqref{eq:defqq_12}}{=}
\partial_k \big( \, (\matrixT_n \matrixR_1[\matrixT_n \mmu_h] - \matrixT_n \matrixR_1[\matrixT_n \nnu_h] ) \,  \, \big) \ww_h -
\partial_k \big( \, (\matrixT_n \matrixR_2[\matrixT_n \mmu_h] - \matrixT_n \matrixR_2[\matrixT_n \nnu_h] ) \,  \, \big) \ww_h
 \\
& \qquad + (\matrixT_n \matrixR_1[\matrixT_n \mmu_h] - \matrixT_n \matrixR_1[\matrixT_n \nnu_h] ) \,  \partial_k \ww_h
 - (\matrixT_n \matrixR_2[\matrixT_n \mmu_h] - \matrixT_n \matrixR_2[\matrixT_n \nnu_h] ) \,  \partial_k \ww_h \\
& \stackrel{\phantom{\eqref{eq:defqq_12}}}{=}
\matrixT_n \partial_k \big( \, ( \matrixR_1[\matrixT_n \mmu_h] - \matrixR_1[\matrixT_n \nnu_h] ) \,  \, \big) \ww_h -
\matrixT_n \partial_k \big( \, ( \matrixR_2[\matrixT_n \mmu_h] - \matrixR_2[\matrixT_n \nnu_h] ) \,  \, \big) \ww_h
 \\
& \qquad + \matrixT_n( \matrixR_1[\matrixT_n \mmu_h] - \matrixR_1[\matrixT_n \nnu_h] ) \,  \partial_k \ww_h
 - \matrixT_n (\matrixR_2[\matrixT_n \mmu_h] - \matrixR_2[\matrixT_n \nnu_h] ) \,  \partial_k \ww_h.
\end{align*}
Recall that $\matrixT_n = \matrixT_n^{-1} = \matrixT_n^T$. With the estimates from {\bf Step 2}, we further get that
\begin{align*}
& \norm{\partial_k \big( \, (\qq[\mmu_h] - \qq[\nnu_h] ) \,  \ww_h \, \big)}{\LL^{2}(\Omega)} \\
& \quad \stackrel{\phantom{\eqref{eq:defqq_12}}}{\leq} 
\norm{\nabla \mmu_h - \nabla \nnu_h}{\LL^\infty(\Omega)} \, \norm{\ww_h}{\LL^2(\Omega)} +
 \gamma^{-2} \, \sigma(\mmu_h,\nnu_h) \, \norm{\mmu_h - \nnu_h}{\LL^\infty(\Omega)} \, \norm{\ww_h}{\LL^2(\Omega)}\\
& \qquad 
+ \gamma^{-1} \, \norm{\nabla \mmu_h - \nabla \nnu_h}{\LL^\infty(\Omega)} \, \norm{\ww_h}{\LL^2(\Omega)}
+ \norm{\mmu_h - \nnu_h}{\LL^\infty(\Omega)} \, \norm{\partial_k \ww_h}{\LL^2(\Omega)} 
 \\
& \qquad + \gamma^{-1} \, \norm{\mmu_h - \nnu_h}{\LL^\infty(\Omega)} \, \norm{\partial_k \ww_h}{\LL^2(\Omega)}.
\end{align*}
With $1 \leq 2/ \gamma$, we arrive at
\begin{align*}
T_1 & \lesssim \gamma^{-1} \, \norm{\nabla \mmu_h - \nabla \nnu_h}{\LL^\infty(\Omega)} \, \norm{\ww_h}{\LL^2(\Omega)} + \gamma^{-1} \, \norm{\mmu_h - \nnu_h}{\LL^\infty(\Omega)} \, \norm{\nabla \ww_h}{\LL^2(\Omega)} \\
& \qquad + \gamma^{-2} \, \sigma(\mmu_h,\nnu_h) \, \norm{\mmu_h - \nnu_h}{\LL^\infty(\Omega)} \, \norm{\ww_h}{\LL^2(\Omega)}.
\end{align*}

{\bf Step 5.} We estimate $T_2$. To this end, let $\ell, k \in \{ 1,2,3 \}$. Note that the second derivative of the piecewise affine function $\ww_h$ vanishes on each element $K \in \TT_h$. Moreover, recall from~\eqref{eq:secondr1zero}, that elementwise $\partial_\ell \partial_k \matrixR_1[\matrixT_n\mmu_h] = \0 = \partial_\ell \partial_k \matrixR_1[\matrixT_n\nnu_h]$. The product rule yields elementwise that
\begin{align*}
& \partial_\ell \partial_k \big( \, (\qq[\mmu_h] - \qq[\nnu_h] ) \,  \ww_h \, \big) \\
& \stackrel{\phantom{\eqref{eq:defqq_12}}}{=} 
\partial_{\ell} \partial_k \big( \, \qq[\mmu_h] - \qq[\nnu_h]  \,  \big) \ww_h + 
 \partial_k \big( \, \qq[\mmu_h] - \qq[\nnu_h]  \, \big) \partial_{\ell} \ww_h + 
  \partial_\ell \big( \, \qq[\mmu_h] - \qq[\nnu_h]  \, \big) \partial_k \ww_h \\
%\intertext{}
& \stackrel{\eqref{eq:defqq_12}}{=}
\partial_{\ell} \partial_k \big( \, \matrixT_n \matrixR_2[\matrixT_n \mmu_h] - \matrixT_n \matrixR_2[\matrixT_n\nnu_h]  \, \big) \ww_h + 
\partial_k \big( \, \matrixT_n \matrixR_1[\matrixT_n \mmu_h] - \matrixT_n \matrixR_1[\matrixT_n \nnu_h]  \, \big) \partial_{\ell} \ww_h \\
& \qquad  - 
\partial_k \big( \, \matrixT_n \matrixR_2[\matrixT_n \mmu_h] - \matrixT_n \matrixR_2[\matrixT_n \nnu_h]  \, \big) \partial_{\ell} \ww_h 
+
\partial_\ell \big( \, \matrixT_n \matrixR_1[\matrixT_n \mmu_h] -  \matrixT_n \matrixR_1[\matrixT_n\nnu_h]  \, \big) \partial_{k} \ww_h \\
& \qquad -
\partial_\ell \big( \, \matrixT_n \matrixR_2[\matrixT_n \mmu_h] - \matrixT_n \matrixR_2[ \matrixT_n \nnu_h]  \, \big) \partial_{k} \ww_h \\
%\intertext{}
& \stackrel{\phantom{\eqref{eq:defqq_12}}}{=}
\matrixT_n \partial_{\ell} \partial_k \big( \,  \matrixR_2[\matrixT_n \mmu_h] - \matrixR_2[\matrixT_n\nnu_h]  \, \big) \ww_h + 
\matrixT_n \partial_k \big( \,  \matrixR_1[\matrixT_n \mmu_h] - \matrixR_1[\matrixT_n \nnu_h]  \, \big) \partial_{\ell} \ww_h \\
& \qquad  - 
\matrixT_n \partial_k \big( \, \matrixR_2[\matrixT_n \mmu_h] -  \matrixR_2[\matrixT_n \nnu_h]  \, \big) \partial_{\ell} \ww_h 
+
\matrixT_n \partial_\ell \big( \,  \matrixR_1[\matrixT_n \mmu_h] - \matrixR_1[\matrixT_n\nnu_h]  \, \big) \partial_{k} \ww_h \\
& \qquad -
\matrixT_n \partial_\ell \big( \, \matrixR_2[\matrixT_n \mmu_h] - \matrixR_2[ \matrixT_n \nnu_h]  \, \big) \partial_{k} \ww_h.
\end{align*}
With the estimates from {\bf Step 2} and since $\matrixT_n = \matrixT_n^{-1} = \matrixT_n^T$, we get for all $K \in \TT_h$ that
\begin{align*}
& \norm{\partial_\ell \partial_k \big[ \, (\qq[\mmu_h] - \qq[\nnu_h] ) \,  \ww_h \, \big]}{\LL^2(K)} \\
& \lesssim
\gamma^{-3} \,
\sigma(\mmu_h,\nnu_h)^2 \,
\norm{\mmu_h - \nnu_h}{\LL^{\infty}(\Omega)} 
\norm{\ww_h}{\LL^2(K)}
+ \gamma^{-2} \,
\sigma(\mmu_h,\nnu_h) \,
\norm{\nabla \mmu_h - \nabla \nnu_h}{\LL^{\infty}(\Omega)} 
\norm{\ww_h}{\LL^2(K)} \\
& \quad + \norm{\nabla \mmu_h - \nabla \nnu_h}{\LL^\infty(\Omega)} \, \norm{\partial_\ell \ww_h}{\LL^2(K)} 
+ \gamma^{-1} \norm{\nabla \mmu_h - \nabla \nnu_h}{\LL^\infty(\Omega)} \, \norm{\partial_\ell \ww_h}{\LL^2(K)}
\\
& \quad + \gamma^{-2} \, \sigma(\mmu_h,\nnu_h) \, \norm{\mmu_h - \nnu_h}{\LL^{\infty}(\Omega)} \, 
\norm{\partial_\ell \ww_h}{\LL^2(K)} 
+ \norm{\nabla \mmu_h - \nabla \nnu_h}{\LL^\infty(\Omega)} \, \norm{\partial_k \ww_h}{\LL^2(K)} \\
& \quad + \gamma^{-1} \norm{\nabla \mmu_h - \nabla \nnu_h}{\LL^\infty(\Omega)} \, \norm{\partial_k \ww_h}{\LL^2(K)} + \gamma^{-2} \, \sigma(\mmu_h,\nnu_h) \, \norm{\mmu_h - \nnu_h}{\LL^{\infty}(\Omega)} \, 
\norm{\partial_k \ww_h}{\LL^2(K)}. 
\end{align*}
With $1 \leq 2 / \gamma$, the latter estimate simplifies to
\begin{align*}
T_2 & \lesssim 
\gamma^{-3} \,
\sigma(\mmu_h,\nnu_h)^2 \,
\norm{\mmu_h - \nnu_h}{\LL^{\infty}(\Omega)} 
\norm{\ww_h}{\LL^2(\Omega)} 
+ \gamma^{-2} \, \sigma(\mmu_h,\nnu_h) \, \norm{\mmu_h - \nnu_h}{\LL^{\infty}(\Omega)} \, 
\norm{\nabla \ww_h}{\LL^2(\Omega)}
\\
& \quad +
\gamma^{-2} \, \sigma(\mmu_h,\nnu_h) \,
\norm{\nabla \mmu_h - \nabla \nnu_h}{\LL^{\infty}(\Omega)} 
\norm{\ww_h}{\LL^2(K)}
+ \gamma^{-1} \norm{\nabla \mmu_h - \nabla \nnu_h}{\LL^\infty(\Omega)} \, \norm{\nabla \ww_h}{\LL^2(\Omega)}.
\end{align*}

{\bf Step 6.} We combine {\bf Step 3}--{\bf Step 5}. An inverse estimate and $1 \leq 2 / \gamma$ imply that
\begin{align*}
& \norm{\nabla \big[ \, ( \mathbb{P}_h[\mmu_h] - \mathbb{P}_h[\nnu_h] ) \, \ww_h \, \big]}{\LL^2(\Omega)} 
\leq T_1 + h T_2
\\
& \stackrel{\phantom{\eqref{eq:defsigma}}}{\lesssim}
\gamma^{-1} \, \norm{\nabla \mmu_h - \nabla \nnu_h}{\LL^\infty(\Omega)} \, \norm{\ww_h}{\LL^2(\Omega)} + \gamma^{-1} \, \norm{\mmu_h - \nnu_h}{\LL^\infty(\Omega)} \, \norm{\nabla \ww_h}{\LL^2(\Omega)} \\
& \quad + \gamma^{-2} \, \sigma(\mmu_h,\nnu_h) \, \norm{\mmu_h - \nnu_h}{\LL^\infty(\Omega)} \, \norm{\ww_h}{\LL^2(\Omega)}
+
h \gamma^{-3} \,
\sigma(\mmu_h,\nnu_h)^2 \,
\norm{\mmu_h - \nnu_h}{\LL^{\infty}(\Omega)} 
\norm{\ww_h}{\LL^2(\Omega)} \\
& \quad +
h \gamma^{-2} \, \sigma(\mmu_h,\nnu_h) \, \norm{\mmu_h - \nnu_h}{\LL^{\infty}(\Omega)} \, 
\norm{\nabla \ww_h}{\LL^2(\Omega)} +
h \gamma^{-2} \, \sigma(\mmu_h,\nnu_h) \, \norm{\nabla \mmu_h - \nabla \nnu_h}{\LL^{\infty}(\Omega)} \, 
\norm{\ww_h}{\LL^2(\Omega)} \\
& \quad
+ h \gamma^{-1} \norm{\nabla \mmu_h - \nabla \nnu_h}{\LL^\infty(\Omega)} \, \norm{\nabla \ww_h}{\LL^2(\Omega)} \\
& \stackrel{\eqref{eq:defsigma}}{\lesssim}
\gamma^{-1} \, \norm{\nabla \mmu_h - \nabla \nnu_h}{\LL^\infty(\Omega)} \, \norm{\ww_h}{\LL^2(\Omega)} 
+ \gamma^{-1} \, \norm{\mmu_h - \nnu_h}{\LL^\infty(\Omega)} \, \norm{\nabla \ww_h}{\LL^2(\Omega)} \\
& \quad + \gamma^{-3} \,
\sigma(\mmu_h,\nnu_h) \,
\norm{\mmu_h - \nnu_h}{\LL^{\infty}(\Omega)} 
\norm{\ww_h}{\LL^2(\Omega)}.
\end{align*}
This concludes the proof.
\end{proof}

%-------------------------------------------------------------------
\subsection{Energy norms}\label{subsection:energynorms}
%-------------------------------------------------------------------
From~\eqref{eq:energyscalar}, recall for given $\mmu_h \in \magnetizationset_h$ the energy-scalar product
\begin{align}\label{eq:energyscalar_inprove}
\energydual{\vectorX}{\vectorY}{\mmu_h} & := \vectorX \cdot \big( \, \matrixP_{\matrixQ}[\mmu_h] \, \big)^{-1}\vectorY
\quad \textrm{for all } \vectorX, \vectorY \in \R^{2N},
\end{align}
with the corresponding energy norm $|\!|\!|\cdot|\!|\!|_{\mmu_h}$.
With the definitions of the matrices from Section~\ref{section:linalg} as well as the definition~\eqref{eq:defPPhmh2} of $\widetilde{\mathbb{P}}_h(\cdot)$, it follows that
\begin{align}\label{eq:energynormsinprove}
\energydual{\vectorX}{\vectorY}{\mmu_h} & = 
\alpha_{\matrixP} \, \dual{\widetilde{\mathbb{P}}[\mmu_h]\vectorX}{\widetilde{\mathbb{P}}[\mmu_h]\vectorY}{\Omega}
+
\beta(k) k \, \dual{\nabla \widetilde{\mathbb{P}}[\mmu_h]\vectorX}{\nabla \widetilde{\mathbb{P}}[\mmu_h]\vectorY}{\Omega}.
\end{align}
This section collects equivalence results for varying arguments $\mmu_h$ in $\energydual{\cdot}{\cdot}{\mmu_h}$ and $|\!|\!|\cdot|\!|\!|_{\mmu_h}$.

\begin{lemma}\label{lemma:energynorms_brachial}
Let $\mmu_h, \nnu_h \in \magnetizationset_h$ and
\begin{subequations}\label{eq:defkappa_tilde_overall}
\begin{align}\label{eq:defkappa_tilde}
\widetilde{\kappa}(\mmu_h,\nnu_h,h^{-2}) := 
\bigg( \, 1 + 
\frac{\beta(k)k}{\alpha_{\matrixP} h^2}
\max_{z \in \NN_h} \matrixnorm{\matrixH[\matrixT_n\mmu_h(z)] - \matrixH[\matrixT_n\nnu_h(z)]}^2 \, \bigg)^{1/2} \geq 1.
\end{align}
Then, $\widetilde{\kappa}(\mmu_h,\nnu_h,h^{-2}) = \widetilde{\kappa}(\nnu_h,\mmu_h,h^{-2})$ and there exists a constant $C\geq1$, which depends only on $\Cmesh$, such that
\begin{align}
C^{-1} \, \widetilde{\kappa}(\mmu_h,\nnu_h,h^{-2})^{-1}
\, |\!|\!|\vectorX|\!|\!|_{\nnu_h}
\leq 
|\!|\!|\vectorX|\!|\!|_{\mmu_h} \leq 
C \,
\widetilde{\kappa}(\mmu_h,\nnu_h,h^{-2}) 
\, |\!|\!|\vectorX|\!|\!|_{\nnu_h} 
\quad \textrm{for all } \vectorX \in \R^{2N}.
\end{align}
\end{subequations}
\end{lemma}

\begin{proof}
Let $\vectorX \in \R^{2N}$. Since the symmetry $\widetilde{\kappa}(\mmu_h,\nnu_h,h^{-2}) = \widetilde{\kappa}(\nnu_h,\mmu_h,h^{-2})$ is obvious, we only have to show that $|\!|\!|\vectorX|\!|\!|_{\mmu_h} \lesssim 
\widetilde{\kappa}(\mmu_h,\nnu_h,h^{-2}) |\!|\!|\vectorX|\!|\!|_{\nnu_h}$. To this end, Lemma~\ref{lemma:l2norms}~{\rm (v)} and an inverse estimate yield that
\begin{align*}
&|\!|\!|\vectorX|\!|\!|_{\mmu_h}^2 
\stackrel{\eqref{eq:energynormsinprove}}{=}
\alpha_{\matrixP} \norm{\widetilde{\mathbb{P}}_h[\mmu_h] \vectorX}{\LL^2(\Omega)}^2 + 
\beta(k) k \,  \norm{\nabla \widetilde{\mathbb{P}}_h[\mmu_h] \vectorX}{\LL^2(\Omega)}^2 \\
& \quad \stackrel{\phantom{\eqref{eq:energynormsinprove}}}{\lesssim}
\alpha_{\matrixP} \norm{\widetilde{\mathbb{P}}_h[\nnu_h] \vectorX}{\LL^2(\Omega)}^2 + 
\beta(k) k \, \norm{\nabla \widetilde{\mathbb{P}}_h[\nnu_h] \vectorX}{\LL^2(\Omega)}^2 +
\beta(k) k \, \norm{\nabla \widetilde{\mathbb{P}}_h[\mmu_h] \vectorX - \nabla \widetilde{\mathbb{P}}_h[\nnu_h] \vectorX}{\LL^2(\Omega)}^2 \\
& \quad \stackrel{\phantom{\eqref{eq:energynormsinprove}}}{\lesssim} 
|\!|\!|\vectorX|\!|\!|_{\nnu_h}^2 + 
\beta(k) k h^{-2} 
\norm{\widetilde{\mathbb{P}}_h[\mmu_h] \vectorX - \widetilde{\mathbb{P}}_h[\nnu_h] \vectorX}{\LL^2(\Omega)}^2.
\end{align*}
With Lemma~\ref{lemma:l2norms}~{\rm (iv)} and~{\rm (vi)}, we estimate the last term by
\begin{align*}
& \beta(k) k h^{-2} \,
\norm{\widetilde{\mathbb{P}}_h[\mmu_h] \vectorX - \widetilde{\mathbb{P}}_h[\nnu_h] \vectorX}{\LL^2(\Omega)}^2
\lesssim
\beta(k) k h \, | \vectorX |^2
\max_{z \in \NN_h} \matrixnorm{\matrixH[\matrixT_n\mmu_h(z)] - \matrixH[\matrixT_n\nnu_h(z)]}^2  \\
& \quad \stackrel{\phantom{\eqref{eq:energynormsinprove}}}{\simeq}
\beta(k) k h^{-2} \, \norm{\widetilde{\mathbb{P}}_h[\nnu_h] \vectorX}{\LL^2(\Omega)}^2 \,
\max_{z \in \NN_h} \matrixnorm{\matrixH[\matrixT_n\mmu_h(z)] - \matrixH[\matrixT_n\nnu_h(z)]}^2 \\
& \quad \stackrel{\eqref{eq:energynormsinprove}}{\lesssim}
\beta(k) \alpha_{\matrixP}^{-1} k h^{-2} \, |\!|\!|\vectorX|\!|\!|_{\nnu_h}^2
\max_{z \in \NN_h} \matrixnorm{\matrixH[\matrixT_n\mmu_h(z)] - \matrixH[\matrixT_n\nnu_h(z)]}^2.
\end{align*}
This proves $|\!|\!|\vectorX|\!|\!|_{\mmu_h} \lesssim \widetilde{\kappa}(\mmu_h,\nnu_h,h^{-2}) |\!|\!|\vectorX|\!|\!|_{\nnu_h}$ and hence concludes the proof.
\end{proof}

For certain $\mmu_h, \nnu_h \in \magnetizationset_h$, the norm equivalence $|\!|\!|\cdot|\!|\!|_{\mmu_h} \simeq |\!|\!|\cdot|\!|\!|_{\nnu_h}$ holds independently of the mesh-size $h$.

\begin{lemma}\label{lemma:energynorms_new}
Let $\mmu_h, \nnu_h \in \magnetizationset_h$ with $1 + (\matrixT_n\mmu_h(z))_3 \geq \gamma > 0$ and $1 + (\matrixT_n\nnu_h(z))_3 \geq \gamma > 0$ for all nodes $z \in \NN_h$. Let
\begin{align}\label{eq:defkappa}
\begin{split}
\kappa(\mmu_h,\nnu_h) & := \Big[ \, 1 
+ 
\gamma^{-2} \, \norm{\mmu_h - \nnu_h}{\LL^\infty(\Omega)}^2  
+  
\frac{\beta(k) k}{\alpha_{\matrixP} \gamma^{2}} \, \norm{\nabla \mmu_h - \nabla \nnu_h}{\LL^\infty(\Omega)}^2 \\
& \qquad + 
\frac{\beta(k) k}{\alpha_{\matrixP} \gamma^{6}} \,
\big( \, \norm{\nabla \mmu_h}{\LL^\infty(\Omega)}^2 + \norm{\nabla \nnu_h}{\LL^\infty(\Omega)}^2 \, \big) \,
\norm{\mmu_h - \nnu_h}{\LL^{\infty}(\Omega)}^2 \,
\Big]^{1/2} \geq 1.
\end{split}
\end{align}
Then, $\kappa(\mmu_h,\nnu_h) = \kappa(\nnu_h,\mmu_h)$ and there exists $C\geq1$ depending only on $\Cmesh$ such that
\begin{align*}
C^{-1} \, \kappa(\mmu_h,\nnu_h)^{-1} \, |\!|\!|\vectorX|\!|\!|_{\nnu_h} 
\leq 
|\!|\!|\vectorX|\!|\!|_{\mmu_h} \leq 
C \, \kappa(\mmu_h,\nnu_h) \, |\!|\!|\vectorX|\!|\!|_{\nnu_h} \quad \textrm{for all } \vectorX \in \R^{2N}.
\end{align*}
\end{lemma}

\begin{proof}
Let $\vectorX \in \R^{2N}$. Since the symmetry $\kappa(\mmu_h,\nnu_h) = \kappa(\nnu_h,\mmu_h)$ is obvious, we only have to show that $|\!|\!|\vectorX|\!|\!|_{\mmu_h} \lesssim \kappa(\mmu_h,\nnu_h) |\!|\!|\vectorX|\!|\!|_{\nnu_h}$. With Lemma~\ref{lemma:representations}, we get that
\begin{subequations}\label{eq:composition_norms_new}
\begin{align}\label{eq:composition2_new}
\!\!\!\!\! \widetilde{\mathbb{P}}_h[\mmu_h] 
%= 
%\widetilde{\mathbb{P}}_h[\nnu_h] 
%+ \widetilde{\mathbb{P}}_h[\mmu_h] 
%- \widetilde{\mathbb{P}}_h[\nnu_h] 
= \widetilde{\mathbb{P}}_h[\nnu_h] + \big( \, \mathbb{P}_h[\mmu_h] - \mathbb{P}_h[\nnu_h] \, \big) 
\circ \mathbb{P}_h^T[\nnu_h] \circ \widetilde{\mathbb{P}}_h[\nnu_h].
\end{align}
With Lemma~\ref{lemma:superfancylemma2}, Lemma~\ref{lemma:l2norms}~{\rm (ii)}, and Lemma~\ref{lemma:superfancylemma}~{\rm (ii)}, we get that
\begin{align}\label{eq:composition_new_h1}
&
\norm{ \nabla
\big( \, \mathbb{P}_h[\mmu_h] - \mathbb{P}_h[\nnu_h] \, \big) 
\circ \mathbb{P}_h^T[\nnu_h] \circ \widetilde{\mathbb{P}}_h[\nnu_h] \, \vectorX
}{\LL^2(\Omega)} 
\notag \\
& \quad
\stackrel{\eqref{eq:superfancyestimate}}{\lesssim} \gamma^{-1} \, \norm{\nabla \mmu_h - \nabla \nnu_h}{\LL^\infty(\Omega)} \, \norm{\mathbb{P}_h^T[\nnu_h] \circ \widetilde{\mathbb{P}}_h[\nnu_h] \, \vectorX}{\LL^2(\Omega)} 
\notag \\
& \qquad + 
\, \gamma^{-1} \, \norm{\mmu_h - \nnu_h}{\LL^\infty(\Omega)} \, 
\norm{\nabla \big [ \, \mathbb{P}_h^T[\nnu_h] \circ \widetilde{\mathbb{P}}_h[\nnu_h] \, \vectorX \, \big]}{\LL^2(\Omega)} \notag \\
& \qquad + \, \gamma^{-3} \,
\big( \, \norm{\nabla \mmu_h}{\LL^\infty(\Omega)} + \norm{\nabla \nnu_h}{\LL^\infty(\Omega)} \, \big) \,
\norm{\mmu_h - \nnu_h}{\LL^{\infty}(\Omega)} \,
\norm{\mathbb{P}_h^T[\nnu_h] \circ \widetilde{\mathbb{P}}_h[\nnu_h] \, \vectorX}{\LL^2(\Omega)} \notag \\
& \quad \stackrel{\phantom{\eqref{eq:superfancyestimate}}}{\lesssim}
\, \gamma^{-1} \, \norm{\nabla \mmu_h - \nabla \nnu_h}{\LL^\infty(\Omega)} \, 
\norm{\widetilde{\mathbb{P}}_h[\nnu_h] \, \vectorX}{\LL^2(\Omega)}
 + \gamma^{-1} \, \norm{\mmu_h - \nnu_h}{\LL^\infty(\Omega)} \, 
\norm{\nabla \widetilde{\mathbb{P}}_h[\nnu_h] \, \vectorX}{\LL^2(\Omega)}
\notag \\
& \qquad + \gamma^{-3} \,
\big( \, \norm{\nabla \mmu_h}{\LL^\infty(\Omega)} + \norm{\nabla \nnu_h}{\LL^\infty(\Omega)} \, \big) \,
\norm{\mmu_h - \nnu_h}{\LL^{\infty}(\Omega)} \,
\norm{\widetilde{\mathbb{P}}_h[\nnu_h] \, \vectorX}{\LL^2(\Omega)}.
\end{align}
\end{subequations}
Recalling $\norm{ \mathbb{P}_h[\mmu_h] \vectorX }{\LL^2(\Omega)} \simeq \norm{ \mathbb{P}_h[\nnu_h] \vectorX }{\LL^2(\Omega)}$ from Lemma~\ref{lemma:l2norms}~{\rm (v)}, we obtain that
\begin{align*}
|\!|\!|\vectorX|\!|\!|_{\mmu_h}^2
& \stackrel{\eqref{eq:composition_norms_new}}{\lesssim} 
\alpha_{\matrixP} \,
\,
\norm{\widetilde{\mathbb{P}}_h[\nnu_h] \vectorX}{\LL^2(\Omega)}^2 + \beta(k) k \,
\Big[ \,
\gamma^{-2} \, \norm{\nabla \mmu_h - \nabla \nnu_h}{\LL^\infty(\Omega)}^2 \\
& \qquad +
\gamma^{-6} \,
\big( \, \norm{\nabla \mmu_h}{\LL^\infty(\Omega)}^2 + \norm{\nabla \nnu_h}{\LL^\infty(\Omega)}^2 \, \big) \,
\norm{\mmu_h - \nnu_h}{\LL^{\infty}(\Omega)}^2 \,
\Big] \,
\norm{\widetilde{\mathbb{P}}_h[\nnu_h] \vectorX}{\LL^2(\Omega)}^2 \\
&  \quad + \beta(k) k \, \Big[ \, 
1 + \gamma^{-2} \, \norm{\mmu_h - \nnu_h}{\LL^\infty(\Omega)}^2
 \, \Big] 
\norm{\nabla \widetilde{\mathbb{P}}_h[\nnu_h] \vectorX}{\LL^2(\Omega)}^2\\
%\intertext{}
& \stackrel{\eqref{eq:composition_norms_new}}{\leq}  \Big[ \, 1 
+ 
\gamma^{-2} \, \norm{\mmu_h - \nnu_h}{\LL^\infty(\Omega)}^2  
+  
\frac{\beta(k) k}{\alpha_{\matrixP} \gamma^{2}} \, \norm{\nabla \mmu_h - \nabla \nnu_h}{\LL^\infty(\Omega)}^2 \\
& \qquad + 
\frac{\beta(k) k}{\alpha_{\matrixP} \gamma^{6}} \,
\big( \, \norm{\nabla \mmu_h}{\LL^\infty(\Omega)}^2 + \norm{\nabla \nnu_h}{\LL^\infty(\Omega)}^2 \, \big) \,
\norm{\mmu_h - \nnu_h}{\LL^{\infty}(\Omega)}^2 \,
\Big] \, |\!|\!|\vectorX|\!|\!|_{\nnu_h}^2 .
\end{align*}
This proves $|\!|\!|\vectorX|\!|\!|_{\mmu_h} \lesssim \kappa(\nnu_h,\mmu_h) |\!|\!|\vectorX|\!|\!|_{\nnu_h}$ and concludes the proof.
\end{proof}

\begin{lemma}\label{lemma:spectral}
Let $\mmu_h, \nnu_h \in \magnetizationset_h$. There exists a constant $C>1$, which depends only on $\Cmesh > 0$, such that the following two assertions {\rm (i)}--{\rm (ii)} hold true:

{\rm (i)} With $\widetilde{\kappa}(\mmu_h,\nnu_h,h^{-2})$ from~\eqref{eq:defkappa_tilde}, it holds that, for all $\vectorX, \vectorY \in \R^{2N}$,
\begin{subequations}
\begin{align}\label{eq:ellipticity}
\vectorX \cdot \matrixA_{\matrixQ}[\mmu_h] \vectorX 
& \, \geq \,
C^{-1} \, \frac{\alpha}{\alpha_{\matrixP}} \,
\widetilde{\kappa}(\mmu_h,\nnu_h,h^{-2})^{-2} \,
|\!|\!|\vectorX|\!|\!|_{\nnu_h}^2 ,
\quad \textrm{and}\\
\label{eq:continuity}
\vectorX \cdot \matrixA_{\matrixQ}[\mmu_h] \vectorY 
& \, \leq \, 
C \, \bigg( \, 1 
+ \frac{1}{\alpha_{\matrixP}} 
+ \frac{ \norm{\weight_k - \alpha_{\matrixP}}{L^{\infty}(\Omega)} }{ \alpha_{\matrixP} } \, \bigg) \, 
\widetilde{\kappa}(\mmu_h,\nnu_h,h^{-2})^{2} \,
|\!|\!|\vectorX|\!|\!|_{\nnu_h}
|\!|\!|\vectorY|\!|\!|_{\nnu_h}.
\end{align}
\end{subequations}

{\rm (ii)} If, additionally, $1 + (\matrixT_n\mmu_h(z))_3 \geq \gamma > 0$ and $1 + (\matrixT_n\nnu_h(z))_3 \geq \gamma > 0$ for all nodes $z \in \NN_h$, the statement of {\rm (i)} holds with $\kappa(\mmu_h,\nnu_h)$ from~\eqref{eq:defkappa} instead of $\widetilde{\kappa}(\mmu_h,\nnu_h,h^{-2})$ from~\eqref{eq:defkappa_tilde}. In particular, the estimate then is independent of the mesh-size $h$.
\end{lemma}

\begin{proof}
First, we prove~{\rm (i)}. Let $\vectorX, \vectorY \in \R^{2N}$. Recall $\matrixA_{\matrixQ}[\mmu_h]$ from~\eqref{eq:preconditionedsystem} as well as $\matrixA_k[\mmu_h]$, $\matrixM_k[\mmu_h]$, $\matrixL$, $\matrixS[\mmu_h]$ from Section~\ref{section:linalg}. Since $\matrixS[\mmu_h]$ is skew-symmetric, it holds that
\begin{align*}
& \vectorX \cdot \matrixA_{\matrixQ}[\mmu_h] \vectorX 
\stackrel{\eqref{eq:preconditionedsystem}}{=}
\alpha \, \matrixQ[\mmu_h]\vectorX \cdot \matrixM_{k}[\mm_h^n] \matrixQ[\mmu_h]\vectorX 
+ \beta(k) k \, \matrixQ[\mmu_h]\vectorX \cdot \matrixL \matrixQ[\mmu_h]\vectorX \\
& \stackrel{\eqref{eq:defPPhmh1}}{=} 
\dual{\weight_k(\lambda_h^n)\, \widetilde{\mathbb{P}}_h[\mmu_h]\vectorX}{\widetilde{\mathbb{P}}_h[\mmu_h]\vectorX}{\Omega}
+ \beta(k) k \,
\dual{\nabla \widetilde{\mathbb{P}}_h[\mmu_h]\vectorX}{\nabla \widetilde{\mathbb{P}}_h[\mmu_h]\vectorX}{\Omega} \\
& \,\, \stackrel{\eqref{eq:stabilizations}}{\geq}
\frac{\alpha}{2} \,
\dual{\widetilde{\mathbb{P}}_h[\mmu_h]\vectorX}{\widetilde{\mathbb{P}}_h[\mmu_h]\vectorX}{\Omega}
+ \beta(k) k \,
\dual{\nabla \widetilde{\mathbb{P}}_h[\mmu_h]\vectorX}{\nabla \widetilde{\mathbb{P}}_h[\mmu_h]\vectorX}{\Omega} 
\stackrel{\eqref{eq:energynormsinprove}}{\geq} \frac{\alpha}{2 \alpha_{\matrixP}} \, |\!|\!|\vectorX|\!|\!|_{\mmu_h}^2.
\end{align*}
With the norm equivalence result from Lemma~\ref{lemma:energynorms_brachial}, we replace $|\!|\!|\vectorX|\!|\!|_{\mmu_h}$ with $|\!|\!|\vectorX|\!|\!|_{\nnu_h}$ and prove~\eqref{eq:ellipticity}. Similarly, we obtain that
\begin{align*}
\vectorX \cdot \matrixA_{\matrixQ}[\mmu_h] \vectorY 
%& \stackrel{\phantom{\eqref{eq:energynormsinprove}}}{=} 
%\dual{\weight_k(\lambda_h^n)\, \widetilde{\mathbb{P}}_h[\mmu_h]\vectorX}{\widetilde{\mathbb{P}}_h[\mmu_h]\vectorY}{\Omega}
%+ \beta(k) k \,
%\dual{\nabla \widetilde{\mathbb{P}}_h[\mmu_h]\vectorX}{\nabla \widetilde{\mathbb{P}}_h[\mmu_h]\vectorY}{\Omega} \\
%& \qquad 
%+ \dual{\mmu_h \times \widetilde{\mathbb{P}}_h[\mmu_h]\vectorX}{\widetilde{\mathbb{P}}_h[\mmu_h]\vectorY}{\Omega} \\
& \stackrel{\eqref{eq:energynormsinprove}}{=}  
\energydual{\vectorX}{\vectorY}{\mmu_h}
+ \dual{(\weight_k(\lambda_h^n) - \alpha_{\matrixP})\, \widetilde{\mathbb{P}}_h[\mmu_h]\vectorX}{\widetilde{\mathbb{P}}_h[\mmu_h]\vectorY}{\Omega} \\
& \qquad + \dual{\mmu_h \times \widetilde{\mathbb{P}}_h[\mmu_h]\vectorX}{\widetilde{\mathbb{P}}_h[\mmu_h]\vectorY}{\Omega} 
\\
& \stackrel{\phantom{\eqref{eq:energynormsinprove}}}{\leq} 
|\!|\!|\vectorX|\!|\!|_{\mmu_h} \, |\!|\!|\vectorY|\!|\!|_{\mmu_h} 
+ \big( \, 1 + \norm{\weight_k - \alpha_{\matrixP}}{L^{\infty}(\Omega)} \, \big) 
\, \norm{\widetilde{\mathbb{P}}_h[\mmu_h]\vectorX}{\LL^2(\Omega)} 
\, \norm{\widetilde{\mathbb{P}}_h[\mmu_h]\vectorY}{\LL^2(\Omega)} \\
& \stackrel{\eqref{eq:energynormsinprove}}{\leq}  
\, \bigg( \, 1 + \frac{1}{\alpha_{\matrixP}} + 
\frac{\norm{\weight_k - \alpha_{\matrixP}}{L^{\infty}(\Omega)}}{\alpha_{\matrixP}} \, \bigg) \,
|\!|\!|\vectorX|\!|\!|_{\mmu_h} \, |\!|\!|\vectorY|\!|\!|_{\mmu_h}.
\end{align*}
Again, with the norm equivalence result from Lemma~\ref{lemma:energynorms_brachial}, we prove~\eqref{eq:continuity}. This concludes the proof of~{\rm (i)}. The proof of~{\rm (ii)} follows the same lines but employs Lemma~\ref{lemma:energynorms_new} instead Lemma~\ref{lemma:energynorms_brachial}. Altogether, this concludes the proof.
\end{proof}

%-------------------------------------------------------------------
\subsection{Proof of Theorem~\ref{theorem:howtosolve}}\label{subsection:proof_howtosolve}
%-------------------------------------------------------------------
Since $\matrixA_{k}[\mm_h^n]$ is positive definite and $\matrixQ[\mm_h^n]$ has orthonormal co\-lumns, the system matrix in~\eqref{eq:linearsystem:constrainedQ} is also positive definite. Let $\vectorX \in \R^{2N}$ be the unique solution of~\eqref{eq:linearsystem:constrainedQ}. Then, it holds that
\begin{align}\label{eq:baseeq}
\matrixQ[\mm_h^n]^T \matrixA_{k}[\mm_h^n] \matrixQ[\mm_h^n] \vectorX \cdot \vectorY
 \stackrel{\eqref{eq:linearsystem:constrainedQ}}{=}
\matrixQ[\mm_h^n]^T\vectorB[\mm_h^n]\cdot \vectorY \quad \textrm{for all } \vectorY \in \R^{2N}.
\end{align}
We denote the bilinear form on the left-hand side of~\eqref{eq:tps2_variational} by $A_h(\cdot,\cdot)$ and the linear functional on the right-hand side of~\eqref{eq:tps2_variational} by $R(\cdot)$. The definition of $\matrixA_{k}[\mm_h^n]$ in Section~\ref{section:linalg} then yields that
\begin{align*}
& A_h(\widetilde{\mathbb{P}}_h[\mm_h^n] \vectorX , \widetilde{\mathbb{P}}_h[\mm_h^n] \vectorY ) 
= 
 \matrixA_{k}[\mm_h^n] \matrixQ[\mm_h^n] \vectorX \cdot \matrixQ[\mm_h^n] \vectorY =
\matrixQ[\mm_h^n]^T \matrixA_{k}[\mm_h^n] \matrixQ[\mm_h^n] \vectorX \cdot \vectorY \\
& \quad \stackrel{\eqref{eq:baseeq}}{=}
\matrixQ[\mm_h^n]^T\vectorB[\mm_h^n]\cdot \vectorY 
 = \vectorB[\mm_h^n]\cdot \matrixQ[\mm_h^n]\vectorY
= R(\widetilde{\mathbb{P}}_h[\mm_h^n] \vectorY) \quad \textrm{for all } \vectorY \in \R^{2N}.
\end{align*}
With Lemma~\ref{lemma:representations}~{(v)}, $\widetilde{\mathbb{P}}_h[\mm_h^n]$ is an isomorphism from $\R^{2N}$ to $\tps{\mm_h^n}$. Consequently, the function $\widetilde{\mathbb{P}}_h[\mm_h^n] \vectorX \in \tps{\mm_h^n}$ is a solution to~\eqref{eq:tps2_variational}. The representation formula~\eqref{eq:variationalform:vh} follows from~\eqref{eq:defPPhmh2}. This concludes the proof. \qed

%-------------------------------------------------------------------
\subsection{Proof of Theorem~\ref{theorem:theoretical_convergence}}\label{proof:theoretical_convergence}
%-------------------------------------------------------------------
First, we prove~{\rm (i)}. For a non-symmetric but positive definite system matrix, the fields-of-value analysis for the preconditioned GMRES algorithm (see, e.g.,~\cite[Theorem~3.2]{Starke97}) yields that
\begin{subequations}\label{eq:starke_estimate}
\begin{align}
|\!|\!|\vectorR^{(\ell)}|\!|\!|_{\mmu_h} \, 
\leq \, 
\Big( \, 1 - \gamma^{(1)} \,  \gamma^{(2)} \, \Big)^{\ell/2} \,
|\!|\!|\vectorR^{(0)}|\!|\!|_{\mmu_h},
\end{align} 
where
\begin{align}
\gamma^{(1)} & := 
\inf_{\vectorX \in \R^{2N} \setminus \{\0 \}} 
\frac{\vectorX \cdot \matrixA_{\matrixQ}[\mm_h^n] \vectorX}{\vectorX \cdot \big( \, \matrixP_{\matrixQ}[\mmu_h] \, \big)^{-1} \vectorX} > 0,
%= 
%\inf_{\vectorX \in \R^{2N} \setminus \{\0 \}} 
%\frac{|\!|\!|\vectorX|\!|\!|_{\mm_h^n}^2}{|\!|\!|\vectorX|\!|\!|_{\mmu_h}^2} 
\label{eq:starke_estimate_1} \\
\gamma^{(2)} & :=
\inf_{\vectorX \in \R^{2N} \setminus \{\0 \}} 
\frac{\vectorX \cdot \big( \, \matrixA_{\matrixQ}[\mm_h^n] \, \big)^{-1} \vectorX}{\vectorX \cdot \matrixP_{\matrixQ}[\mmu_h] \vectorX} >0.
\label{eq:starke_estimate_2}
\end{align}
\end{subequations}
To estimate $\gamma^{(1)}$ and $\gamma^{(2)}$ from below, recall $\widetilde{\kappa}(\mm_h^n,\mmu_h,h^{-2})$ from~\eqref{eq:defkappa_tilde} and exploit Lemma~\ref{lemma:spectral}~{\rm (i)}. This yields that
\begin{align*}
\gamma^{(1)} \, \stackrel{\eqref{eq:defkappa_tilde_overall}}{\gtrsim} \,
\frac{\alpha}{\alpha_{\matrixP}} \widetilde{\kappa}(\mm_h^n,\mmu_h,h^{-2})^{-2} \stackrel{\eqref{eq:factor:theoretical_1}}{=} \frac{\alpha}{\alpha_{\matrixP}} \, \mathbb{F}^{-1}.
\end{align*}
With Lemma~\ref{lemma:spectral}~{\rm (i)}, the matrices $\matrixB := \matrixA_{\matrixQ}[\mm_h^n]$ and $\matrixB_0 :=\big(\matrixP_{\matrixQ}[\mmu_h]\big)^{-1}$ satisfy the setting of Lemma~\ref{lemma:abstract_matrix_setting} with
\begin{align*}
c_1 \, \simeq \,
\frac{\alpha}{\alpha_{\matrixP}}
\widetilde{\kappa}(\mm_h^n,\mmu_h,h^{-2})^{-2},
\quad \textrm{and} \quad c_2 \, \simeq \, \bigg( \, 1 
+ \frac{1}{\alpha_{\matrixP}} 
+ \frac{ \norm{\weight_k - \alpha_{\matrixP}}{L^{\infty}(\Omega)} }{ \alpha_{\matrixP} } \, \bigg) \, 
\widetilde{\kappa}(\mm_h^n,\mmu_h,h^{-2})^{2}.
\end{align*}
Hence, Lemma~\ref{lemma:abstract_matrix_setting} yields that
\begin{align}\label{eq:estimate_gamma2}
\begin{split}
\gamma^{(2)} 
&  \stackrel{\eqref{eq:starke_estimate_2}}{\gtrsim}
\frac{\alpha}{\alpha_{\matrixP}}
\bigg( \, 1 
+ \frac{1}{\alpha_{\matrixP}} 
+ \frac{ \norm{\weight_k - \alpha_{\matrixP}}{L^{\infty}(\Omega)} }{ \alpha_{\matrixP} } \, \bigg)^{-2} \,
\widetilde{\kappa}(\mm_h^n,\mmu_h,h^{-2})^{-6} \\
& \stackrel{\eqref{eq:factor:theoretical_1}}{=} 
\frac{\alpha}{\alpha_{\matrixP}}
\bigg( \, 1 
+ \frac{1}{\alpha_{\matrixP}} 
+ \frac{ \norm{\weight_k - \alpha_{\matrixP}}{L^{\infty}(\Omega)} }{ \alpha_{\matrixP} } \, \bigg)^{-2}
\mathbb{F}^{-3}.
\end{split}
\end{align}
With~\eqref{eq:starke_estimate}, we conclude the proof of~{\rm (i)}. The proof of~{\rm (ii)} then follows the same lines but exploits Lemma~\ref{lemma:spectral}~{\rm (ii)} instead of Lemma~\ref{lemma:spectral}~{\rm (i)}. In the latter arguments, this replaces $\widetilde{\kappa}(\mm_h^n,\mmu_h,h^{-2})$ by $\kappa(\mm_h^n,\mmu_h)$. Altogether, this concludes the proof.
\qed

%-------------------------------------------------------------------
\subsection{Proof of Theorem~\ref{theorem:practical_convergence}}\label{proof:practical_convergence}
%-------------------------------------------------------------------
In a\-na\-lo\-gy to~\eqref{eq:starke_estimate}, the fields-of-value analysis for the preconditioned GMRES algorithm (see, e.g.,~\cite[Theorem~3.2]{Starke97}) yields that
\begin{subequations}\label{eq:starke_estimate_practical}
\begin{align}
|\!|\!|\widetilde{\vectorR}^{(\ell)}|\!|\!|_{\mm_h^n}^{\prime} \, 
\leq \, 
\Big( \, 1 - \widetilde{\gamma}^{(1)} \,  \widetilde{\gamma}^{(2)} \, \Big)^{\ell/2} \,
|\!|\!|\widetilde{\vectorR}^{(0)}|\!|\!|_{\mm_h^n}^{\prime},
\end{align} 
where
\begin{align}
\widetilde{\gamma}^{(1)} & := 
\inf_{\vectorX \in \R^{2N} \setminus \{\0 \}} 
\frac{\vectorX \cdot \matrixA_{\matrixQ}[\mm_h^n] \vectorX}{\vectorX \cdot 
\big( \, \widetilde{\matrixP}_{\matrixQ}[\mm_h^n] \, \big)^{-1} \vectorX} >0,
%= 
%\inf_{\vectorX \in \R^{2N} \setminus \{\0 \}} 
%\frac{|\!|\!|\vectorX|\!|\!|_{\mm_h^n}^2}{|\!|\!|\vectorX|\!|\!|_{\mmu_h}^2} >0,
\label{eq:starke_estimate_practical_1} \\
\widetilde{\gamma}^{(2)} & :=
\inf_{\vectorX \in \R^{2N} \setminus \{\0 \}} 
\frac{\vectorX \cdot \big( \, \matrixA_{\matrixQ}[\mm_h^n] \, \big)^{-1} \vectorX}{\vectorX \cdot \widetilde{\matrixP}_{\matrixQ}[\mm_h^n] \vectorX} >0.
\label{eq:starke_estimate_practical_2}
\end{align}
\end{subequations}
Recall from~\eqref{eq:forma_constrained_preconditioner}, the definition of the theoretical preconditioner $\matrixP[\mm_h^n]$ and from~\eqref{eq:starke_estimate} the corresponding definition of $\gamma^{(1)}$ and $\gamma^{(2)}$. We obtain that
\begin{subequations}\label{eq:deltas}
\begin{align}\label{eq:delta_1}
\widetilde{\gamma}^{(1)} 
& \stackrel{\eqref{eq:starke_estimate_practical_1}}{\geq} 
\inf_{\vectorX \in \R^{2N} \setminus \{\0 \}} 
\frac{\vectorX \cdot \matrixA_{\matrixQ}[\mm_h^n] \vectorX}{\vectorX \cdot \big( \, \matrixP_{\matrixQ}[\mm_h^n] \, \big)^{-1} \vectorX} 
\, 
\inf_{\vectorX \in \R^{2N} \setminus \{\0 \}} 
\frac{\vectorX \cdot \big( \, \matrixP_{\matrixQ}[\mm_h^n] \, \big)^{-1} \vectorX}{\vectorX \cdot \big( \, \widetilde{\matrixP}_{\matrixQ}[\mm_h^n] \, \big)^{-1} \vectorX} \notag \\
& \stackrel{\eqref{eq:starke_estimate_1}}{=} 
\gamma^{(1)} 
\inf_{\vectorX \in \R^{2N} \setminus \{\0 \}} 
\frac{\vectorX \cdot \big( \, \matrixP_{\matrixQ}[\mm_h^n] \, \big)^{-1} \vectorX}{\vectorX \cdot \big( \, \widetilde{\matrixP}_{\matrixQ}[\mm_h^n] \, \big)^{-1} \vectorX} 
=: 
\gamma^{(1)} \, \delta^{(1)},
\end{align}
as well as
\begin{align}\label{eq:delta_2}
\widetilde{\gamma}^{(2)} 
& \stackrel{\eqref{eq:starke_estimate_practical_1}}{\geq} 
\inf_{\vectorX \in \R^{2N} \setminus \{\0 \}} 
\frac{\vectorX \cdot \big( \, \matrixA_{\matrixQ}[\mm_h^n] \, \big)^{-1} \vectorX}{\vectorX \cdot \, \matrixP_{\matrixQ}[\mm_h^n] \, \vectorX} 
\, 
\inf_{\vectorX \in \R^{2N} \setminus \{\0 \}} 
\frac{\vectorX \cdot  \matrixP_{\matrixQ}[\mm_h^n]  \vectorX}{\vectorX \cdot \widetilde{\matrixP}_{\matrixQ}[\mm_h^n] \vectorX} \notag \\
& \stackrel{\eqref{eq:starke_estimate_1}}{=} 
\gamma^{(2)} 
\inf_{\vectorX \in \R^{2N} \setminus \{\0 \}} 
\frac{\vectorX \cdot \matrixP_{\matrixQ}[\mm_h^n] \, \vectorX}{\vectorX \cdot \widetilde{\matrixP}_{\matrixQ}[\mm_h^n] \vectorX} 
=: 
\gamma^{(2)} \, \delta^{(2)}. 
\end{align}
\end{subequations}
Here, we implicitly have $\mm_h^n = \mmu_h$ in the definition~\eqref{eq:starke_estimate} of $\gamma^{(1)}$ and $\gamma^{(2)}$. In particular, Lemma~\ref{lemma:spectral} holds with $\widetilde{\kappa}(\mm_h^n,\mm_h^n,h^{-2}) = \kappa(\mm_h^n,\mm_h^n) = 1$. Following the lines of the proof of Theorem~\ref{theorem:theoretical_convergence}, this yields that
\begin{align}\label{eq:gamma_estimates_practical}
\gamma^{(1)} \, \gtrsim \, \frac{\alpha}{\alpha_{\matrixP}}, \quad \textrm{and} \quad 
\gamma^{(2)} 
\, \stackrel{\eqref{eq:estimate_gamma2}}{\gtrsim} \,
\frac{\alpha}{\alpha_{\matrixP}}
\bigg( \, 1 
+ \frac{1}{\alpha_{\matrixP}} 
+ \frac{ \norm{\weight_k - \alpha_{\matrixP}}{L^{\infty}(\Omega)} }{ \alpha_{\matrixP} } \, \bigg)^{-2}.
\end{align}
Hence, in the following four steps, it remains to estimate $\delta^{(1)}$ and $\delta^{(2)}$ from below.

{\bf Step 1.} We will use the fictitious space lemma (see~\cite{Nep91,Go95}) to derive
\begin{align}\label{eq:fictitious_estimate}
& c_1 \, \vectorX \cdot \matrixP_{\matrixQ} [\mm_h^n] \vectorX 
\, \leq \,
\vectorX \cdot \widetilde{\matrixP}_{\matrixQ} [\mm_h^n] \vectorX 
\, \leq \,
c_2 \, \vectorX \cdot \matrixP_{\matrixQ} [\mm_h^n] \vectorX 
\quad \textrm{for all } \vectorX\in \R^{2N}.
\end{align}
Here, the constants $c_1,c_2 > 0$ stem from the following two assumptions~\ref{item:assumptionA}--\ref{item:assumptionB} of the fictitious space lemma: 
\begin{subequations}\label{eq:propreducedpractical}
\begin{enumerate}[label=\rm{{\bf (FS\arabic*)}}]
\item\label{item:assumptionA} For all $\vectorX \in \R^{2N}$, there exists $\vectorY \in \R^{3N}$ with $\matrixQ[\mm_h^n]^T \vectorY = \vectorX$ and
\begin{align}\label{eq:propreducedpractical3}
c_1 \, \vectorY \cdot (\alpha_{\matrixP} \matrixM + \beta(k) k \matrixL) \vectorY \leq 
\vectorX \cdot \big( \, \matrixP_{\matrixQ}[\mm_h^n] \, \big)^{-1} \vectorX .
\end{align}
\item\label{item:assumptionB} For all $\vectorY \in \R^{3N}$, it holds that
\begin{align}\label{eq:propreducedpractical4}
\matrixQ[\mm_h^n]^T \vectorY \cdot  \big( \, \matrixP_{\matrixQ}[\mm_h^n] \, \big)^{-1} \matrixQ[\mm_h^n]^T \vectorY 
\, \leq \,
c_2 \, \vectorY \cdot (\alpha_{\matrixP} \matrixM + \beta(k) k \matrixL) \vectorY.
\end{align}
\end{enumerate}
\end{subequations}
With the assumptions~\ref{item:assumptionA}--\ref{item:assumptionB}, the fictitious space lemma then implies that
\begin{align*}
c_1 \, \widetilde{\vectorX} \cdot \big( \, \matrixP_{\matrixQ}[\mm_h^n] \, \big)^{-1} \widetilde{\vectorX} 
& \stackrel{\phantom{\eqref{eq:pracital_constrained_preconditioner}}}{\leq} 
\big( \, \matrixP_{\matrixQ}[\mm_h^n] \, \big)^{-1} 
\widetilde{\vectorX} \cdot \matrixQ[\mm_h^n]^T (\alpha_{\matrixP} \matrixM + \beta(k) k \matrixL)^{-1} \matrixQ[\mm_h^n] 
\big( \, \matrixP_{\matrixQ}[\mm_h^n] \, \big)^{-1} \widetilde{\vectorX} \\
& \stackrel{\eqref{eq:pracital_constrained_preconditioner}}{=}
\big( \, \matrixP_{\matrixQ}[\mm_h^n] \, \big)^{-1} \widetilde{\vectorX} \cdot \widetilde{\matrixP}_{\matrixQ}[\mm_h^n] 
\big( \, \matrixP_{\matrixQ}[\mm_h^n] \, \big)^{-1} \widetilde{\vectorX} \\
& \stackrel{\phantom{\eqref{eq:pracital_constrained_preconditioner}}}{\leq}
c_2 \, \widetilde{\vectorX} \cdot \big( \, \matrixP_{\matrixQ}[\mm_h^n] \, \big)^{-1} \widetilde{\vectorX} 
\quad \textrm{for all } \widetilde{\vectorX} \in \R^{2N}.
\end{align*}
Then, with $\widetilde{\vectorX} := \matrixP_{\matrixQ}[\mm_h^n] \vectorX \in \R^{2N}$ in the latter estimate, this verifies~\eqref{eq:fictitious_estimate}.

{\bf Step 2.} We verify assumption~\ref{item:assumptionA} of the fictitious space lemma. To that end, let $\vectorX \in \R^{2N}$ and set $\vectorY:=\matrixQ[\mm_h^n]\vectorX \in \R^{3N}$. Then, $\matrixQ[\mm_h^n]^T\vectorY=\matrixQ[\mm_h^n]^T\matrixQ[\mm_h^n]\vectorX  = \vectorX $
and 
\begin{align*}
& \vectorY \cdot (\alpha_{\matrixP} \matrixM + \beta(k) k \matrixL) \vectorY
\stackrel{\phantom{\eqref{eq:forma_constrained_preconditioner}}}{=} 
\matrixQ[\mm_h^n] \vectorX \cdot (\alpha_{\matrixP} \matrixM + \beta(k) k \matrixL) \matrixQ[\mm_h^n] \vectorX \\
& \stackrel{\phantom{\eqref{eq:forma_constrained_preconditioner}}}{=} 
\vectorX \cdot \matrixQ[\mm_h^n]^T (\alpha_{\matrixP} \matrixM + \beta(k) k \matrixL) \matrixQ[\mm_h^n] \vectorX 
\stackrel{\eqref{eq:forma_constrained_preconditioner}}{=} 
\vectorX \cdot \big( \, \matrixP_{\matrixQ}[\mm_h^n] \, \big)^{-1} \vectorX,
\end{align*}
i.e., assumption~\ref{item:assumptionA} holds with $c_1=1$. 

{\bf Step 3.} We verify assumption~\ref{item:assumptionB} of the fictitious space lemma. To that end, let $\vectorY \in \R^{3N}$. Define $\vv_h := \sum_{i=1}^{3N} \vectorY_i \pphi_i \in \courantspace_h$. With Lemma~\ref{lemma:representations}~{\rm (ii)}, we obtain that
\begin{eqnarray}\label{eq:C2_finalize}
& & \hspace*{-1.5cm} 
\matrixQ[\mm_h^n]^T \vectorY \cdot  \big( \, \matrixP[\mm_h^n] \, \big)^{-1} \matrixQ[\mm_h^n]^T \vectorY \notag \\
& \stackrel{\eqref{eq:forma_constrained_preconditioner}}{=} &
\vectorY \cdot 
\matrixQ[\mm_h^n] \matrixQ[\mm_h^n]^T 
(\alpha_{\matrixP} \matrixM + \beta(k) k \matrixL) 
\matrixQ[\mm_h^n] \matrixQ[\mm_h^n]^T  
\vectorY \notag \\
& \stackrel{\eqref{eq:deforthtangentspace2}}{=} &
\alpha_{\matrixP} \norm{ \widetilde{\Ppi}_{h}[\mm_h^n] \vectorY}{\LL^2(\Omega)}^2
+ 
\beta(k) k \, \norm{\nabla \widetilde{\Ppi}_{h}[\mm_h^n] \vectorY }{\LL^2(\Omega)}^2 \notag \\
& = &
\alpha_{\matrixP} \norm{ \Ppi_{h}[\mm_h^n] \vv_h}{\LL^2(\Omega)}^2
+ 
\beta(k) k \, \norm{\nabla \Ppi_{h}[\mm_h^n] \vv_h }{\LL^2(\Omega)}^2.
\end{eqnarray}
For the verification of~\ref{item:assumptionB} in~{\rm (i)}, Lemma~\ref{lemma:l2norms}~{\rm (iii)} and an inverse estimate yield that
\begin{subequations}\label{eq:thisistheend}
\begin{align}
& \alpha_{\matrixP} \norm{ \Ppi_{h}[\mm_h^n] \vv_h}{\LL^2(\Omega)}^2
+ 
\beta(k) k \, \norm{\nabla \Ppi_{h}[\mm_h^n] \vv_h }{\LL^2(\Omega)}^2 \notag \\
& \lesssim 
\Big( \, \alpha_{\matrixP}  + \frac{\beta(k)k}{h^2} \, \Big) \norm{\vv_h}{\LL^2(\Omega)}^2
\lesssim 
 \, \Big( \, 1  + \frac{\beta(k)k}{\alpha_{\matrixP} h^2} \, \Big) \bigg[ \alpha_{\matrixP} \norm{ \vv_h}{\LL^2(\Omega)}^2
+ 
\beta(k) k \, \norm{\nabla \vv_h }{\LL^2(\Omega)}^2 \bigg] \notag \\
& = \, \Big( \, 1  + \frac{\beta(k)k}{\alpha_{\matrixP} h^2} \, \Big)  
\, \vectorY \cdot \big( \, \alpha_{\matrixP} \matrixM + \beta(k) k \matrixL \, \big) \vectorY.
\end{align}
For the verification of~\ref{item:assumptionB} in~{\rm (ii)}, we use the stronger assumption $1 + (\matrixT_n\mm_h^n)_3 \geq \gamma > 0$. Then, the definition of $\vv_h$, Lemma~\ref{lemma:l2norms}~{\rm (iii)} and Lemma~\ref{lemma:superfancylemma}~{\rm (iii)} yield that
\begin{align}
& \alpha_{\matrixP} \norm{ \Ppi_{h}[\mm_h^n] \vv_h}{\LL^2(\Omega)}^2
+ 
\beta(k) k \, \norm{\nabla \Ppi_{h}[\mm_h^n] \vv_h }{\LL^2(\Omega)}^2 \notag \\
& \qquad \lesssim 
\Big( \, \alpha_{\matrixP} + \frac{\beta(k)k}{\gamma^4} \norm{\nabla \mm_h^n}{\LL^\infty(\Omega)}^2 \, \Big) 
\,  \norm{\vv_h}{\LL^2(\Omega)}^2 + \beta(k) k \norm{\nabla \vv_h}{\LL^2(\Omega)}^2  \notag \\
& \qquad \leq 
\Big( \, 1 + \frac{\beta(k)k}{\alpha_{\matrixP} \gamma^4} \norm{\nabla \mm_h^n}{\LL^\infty(\Omega)}^2 \, \Big) 
\, \bigg[ \, \alpha_{\matrixP} \norm{ \vv_h}{\LL^2(\Omega)}^2
+ 
\beta(k) k \, \norm{\nabla \vv_h }{\LL^2(\Omega)}^2 \, \bigg] \notag \\
& \qquad = \Big( \, 1 + \frac{\beta(k)k}{\alpha_{\matrixP} \gamma^4} \norm{\nabla \mm_h^n}{\LL^\infty(\Omega)}^2 \, \Big)
\, \vectorY \cdot \big( \, \alpha_{\matrixP} \matrixM + \beta(k) k \matrixL \, \big) \vectorY.
\end{align}
\end{subequations}
We combine~\eqref{eq:C2_finalize}--\eqref{eq:thisistheend} and obtain that~\ref{item:assumptionB} holds with
\begin{align}\label{eq:whatc1}
c_2 \, \lesssim \,
\begin{cases}
\displaystyle{ 1 + \frac{\beta(k)k}{\alpha_{\matrixP} \gamma^4} \norm{\nabla \mm_h^n}{\LL^\infty(\Omega)}^2 }
& \textrm{if } 1 + (\matrixT_n\mm_h^n)_3 \geq \gamma > 0, \\
\displaystyle{ 1  + \frac{\beta(k)k}{\alpha_{\matrixP} h^2} }
& \textrm{else.}
\end{cases}
\end{align}

{\bf Step 4.} With {\bf Step 1}--{\bf Step 3}, the matrices $\matrixB := \widetilde{\matrixP}_{\matrixQ}[\mm_h^n]$ and $\matrixB_0 := \matrixP_{\matrixQ}[\mm_h^n]$ satisfy the assumptions of Lemma~\ref{lemma:abstract_matrix_setting} with $c_1 = 1$ and $c_2$ from~\eqref{eq:whatc1}. Hence, we get that
\begin{align}\label{eq:fictitious_estimate_2} \!\!\!
\frac{1}{c_2^2} \, \vectorX \cdot \big( \, \matrixP_{\matrixQ}[\mm_h^n] \, \big)^{-1} \vectorX 
\, \leq \,
\vectorX \cdot \big( \, \widetilde{\matrixP}_{\matrixQ}[\mm_h^n] \, \big)^{-1} \vectorX 
\, \leq \, \vectorX \cdot \big( \, \matrixP_{\matrixQ}[\mm_h^n] \, \big)^{-1} \vectorX 
\quad \textrm{for all } \vectorX\in \R^{2N}.
\end{align}
From this and~\eqref{eq:fictitious_estimate}, we obtain that
\begin{align*}
\delta^{(1)} \!\stackrel{\eqref{eq:delta_1}}{=}\!
\inf_{\vectorX \in \R^{2N} \setminus \{\0 \}} 
\frac{\vectorX \cdot \big( \, \matrixP_{\matrixQ}[\mm_h^n] \, \big)^{-1} \vectorX}{\vectorX \cdot \big( \, \widetilde{\matrixP}_{\matrixQ}[\mm_h^n] \, \big)^{-1} \vectorX} 
\, \stackrel{\eqref{eq:fictitious_estimate_2}}{\geq} \, 1
%\end{align*}
%as well as
%\begin{align*}
\, \text{ and } \,
\delta^{(2)} \!\stackrel{\eqref{eq:delta_2}}{=}\! \inf_{\vectorX \in \R^{2N} \setminus \{\0 \}} 
\frac{\vectorX \cdot \matrixP_{\matrixQ}[\mm_h^n] \, \vectorX}{\vectorX \cdot \widetilde{\matrixP}_{\matrixQ}[\mm_h^n] \vectorX} 
\, \stackrel{\eqref{eq:fictitious_estimate}}{\geq} \, \frac{1}{c_2}.
\end{align*}
Together with the estimates for $\gamma^{(1)}$ and $\gamma^{(2)}$ from~\eqref{eq:gamma_estimates_practical}, this concludes the proof.
\qed

%%%%%%%%%%%%%%%%%%%%%%%%%%%%%%%%%%%%%%%%%%%%%%%%%%%%%%%%%%%%%%%%%%%%
%%%%%%%%%%%%%%%%%%%%%%%%%%%%%%%%%%%%%%%%%%%%%%%%%%%%%%%%%%%%%%%%%%%%
%%%%%%%%%%%%%%%%%%%%%%%%%%%%%%%%%%%%%%%%%%%%%%%%%%%%%%%%%%%%%%%%%%%%

\numberwithin{statement}{section}
\appendix

%%%%%%%%%%%%%%%%%%%%%%%%%%%%%%%%%%%%%%%%%%%%%%%%%%%%%%%%%%%%%%%%%%%%
\section{Auxiliary results}
%%%%%%%%%%%%%%%%%%%%%%%%%%%%%%%%%%%%%%%%%%%%%%%%%%%%%%%%%%%%%%%%%%%%

\begin{lemma}\label{lemma:auxiliary1}
For $\mmu_h = (\mu_1,\mu_2,\mu_3)^T \in \magnetizationset_h$, there hold the following assertions {\rm (i)}--{\rm (iii)}:

{\rm (i)} For all $i,j \in \{ 1,2 \}$, it holds that
\begin{align*}
\bigg\Vert \frac{\mu_i \mu_j}{1+\mu_3} \bigg\Vert_{\LL^{\infty}(\Omega)} \leq 2.
\end{align*} 

{\rm (ii)} Let $1 + \mu_3(z) \geq \gamma > 0$ for all nodes $z \in \NN_h$. For all $i,j \in \{ 1,2 \}$, it holds that
\begin{align*}
\frac{\mu_i \mu_j}{1+\mu_3} \in W^{1,\infty}(\Omega) 
\quad \textrm{with} \quad
\bigg\Vert \, \partial_k \bigg[ \, \frac{\mu_i \mu_j}{1+\mu_3} \, \bigg] \bigg\Vert_{L^{\infty}(\Omega)} \leq 
3 \,
{\gamma}^{-1} 
\, 
\norm{ \partial_k \mmu_h }{\LL^{\infty}(\Omega)}
\end{align*}
for all $k \in \{1,2,3\}$.

{\rm (iii)} Let $1 + \mu_3(z) \geq \gamma > 0$ for all nodes $z \in \NN_h$. For all $i,j \in \{ 1,2 \}$, it holds that
\begin{align*}
\frac{\mu_i \mu_j}{1+\mu_3} \in W^{2,\infty}(K) 
\quad \textrm{with} \quad
\bigg\Vert \, \partial_{\ell} \partial_k \bigg[ \, \frac{\mu_i \mu_j}{1+\mu_3} \, \bigg] \bigg\Vert_{L^{\infty}(K)} \leq 
12 \,
{\gamma}^{-2} 
\,
\norm{ \partial_k \mmu_h }{\LL^{\infty}(\Omega)} \, \norm{ \partial_\ell \mmu_h }{\LL^{\infty}(\Omega)},
\end{align*}
for all elements $K \in \TT_h$ and for all $\ell,k \in \{1,2,3\}$.
\end{lemma}

\begin{proof}
To prove {\rm (i)}, note that piecewise affine functions attain their maximal length at the nodes. Since $\mmu_h \in \magnetizationset_h$, together with Young's inequality this yield that
\begin{align*}
\bigg| \frac{\mu_i \mu_j}{1+ \mu_3} \bigg| \leq 
\frac{\mu_1^2 + \mu_2^2}{1 + \mu_3} \leq 
\frac{1 - \mu_3^2}{1 + \mu_3} = 
\frac{(1 - \mu_3)(1+\mu_3)}{1 + \mu_3} = 
(1- \mu_3) \leq 2.
\end{align*}
This proves {\rm (i)}. For the proof of~{\rm (ii)} and~{\rm (iii)}, the product rule yields that
\begin{align}\label{eq:auxiliary_step}
\partial_k \bigg[ \, \frac{\mu_i \mu_j}{1+\mu_3} \, \bigg] = 
\frac{(\partial_k \mu_i) \mu_j}{1+\mu_3} 
+ \frac{\mu_i (\partial_k \mu_j)}{1+\mu_3} 
- \frac{ \mu_i \mu_j (\partial_k \mu_3)}{(1+\mu_3)^2}.
\end{align}
Moreover, we exploit that the second derivative of affine functions is zero. Elementwise, the product rule then yields that
\begin{align}\label{eq:auxiliary_step2}
\begin{split}
\partial_{\ell} \partial_k \bigg[ \, \frac{\mu_i \mu_j}{1+\mu_3} \, \bigg] 
& \stackrel{\eqref{eq:auxiliary_step}}{=}
\frac{(\partial_k \mu_i) (\partial_{\ell} \mu_j)}{1+\mu_3}
- \frac{(\partial_k \mu_i) \mu_j (\partial_{\ell} \mu_3)}{(1+\mu_3)^2} 
+ \frac{(\partial_{\ell} \mu_i) (\partial_{k} \mu_j)}{1+\mu_3}  
- \frac{\mu_i (\partial_k \mu_j) (\partial_{\ell} \mu_3)}{(1+\mu_3)^2} \\
& \qquad 
-  \frac{ (\partial_{\ell} \mu_i) \mu_j  (\partial_k \mu_3)}{(1+\mu_3)^2} 
-  \frac{ \mu_i (\partial_{\ell} \mu_j)  (\partial_k \mu_3)}{(1+\mu_3)^2} 
+  \frac{2 \mu_i \mu_j   (\partial_k \mu_3)(\partial_{\ell} \mu_3)}{(1+\mu_3)^3}.
\end{split}
\end{align}
Since $|\mmu_h| \leq 1$ and since $\mmu_h$ is piecewise affine, the assumption $1 + \mu_3(z) \geq \gamma > 0$ for all nodes implies that $1 + \mu_3 \geq \gamma > 0$ in $\Omega$. Together with $|\mmu_h| \leq 1$ and~\eqref{eq:auxiliary_step}, this proves~{\rm (ii)}. For the proof of {\rm (iii)}, additionally note that $\gamma \leq 1 + \mu_3 \leq 1 + |\mu_3| \leq 2$ yields that $1 \leq 2/\gamma$. Together with $|\mmu_h| \leq 1$ and~\eqref{eq:auxiliary_step2}, this proves~{\rm (iii)}. 
\end{proof}

\begin{lemma}\label{lemma:auxiliary2}
For $\mmu_h = (\mu_1,\mu_2,\mu_3)^T, \nnu_h = (\nu_1,\nu_2,\nu_3)^T \in \magnetizationset_h$, let $1 + \mu_3(z) \geq \gamma > 0$ and $1 + \nu_3(z) \geq \gamma > 0$ for all nodes $z \in \NN_h$. Then, there exists a constant $C>0$ such that the following assertions~{\rm (i)}--{\rm (iii)} hold true:

{\rm (i)} For all $i,j \in \{ 1,2 \}$, it holds that
\begin{align*}
\bigg\Vert \frac{\mu_i \mu_j}{1+\mu_3} - \frac{\nu_i \nu_j}{1+\nu_3} \bigg\Vert_{L^{\infty}(\Omega)} \leq C \, \gamma^{-1} \norm{ \mmu_h - \nnu_h }{\LL^{\infty}(\Omega)}.
\end{align*}

{\rm (ii)} For all $i,j \in \{1,2\}$ and all $k \in \{1,2,3\}$, it holds that
\begin{align*}
\bigg\Vert \partial_k \bigg( \frac{\mu_i \mu_j}{1+\mu_3} - \frac{\nu_i \nu_j }{1+\nu_3}\bigg) \bigg\Vert_{L^{\infty}(\Omega)} 
& \leq C
\gamma^{-2} \,
\big( \, \norm{\nabla \mmu_h}{\LL^\infty(\Omega)} + \norm{\nabla \nnu_h}{\LL^\infty(\Omega)} \, \big) 
\norm{\mmu_h - \nnu_h}{\LL^{\infty}(\Omega)} \\
& \quad + C \gamma^{-1} \,
\norm{\nabla \mmu_h - \nabla \nnu_h}{\LL^{\infty}(\Omega)} .
\end{align*} 

{\rm (iii)} For all $i,j \in \{1,2\}$, all elements $K \in \TT_h$ and all $\ell,k \in \{1,2,3\}$, it holds that
\begin{align*}
\bigg\Vert \partial_{\ell} \partial_k \bigg( \frac{\mu_i \mu_j}{1+\mu_3} - \frac{\nu_i \nu_j }{1+\nu_3}\bigg) \bigg\Vert_{L^{\infty}(K)} 
& \leq C
\gamma^{-3} \,
\big( \, \norm{\nabla \mmu_h}{\LL^\infty(\Omega)} + \norm{\nabla \nnu_h}{\LL^\infty(\Omega)} \, \big)^2 \,
\norm{\mmu_h - \nnu_h}{\LL^{\infty}(\Omega)} \\
& \quad + C \gamma^{-2} \,
\big( \, \norm{\nabla \mmu_h}{\LL^\infty(\Omega)} + \norm{\nabla \nnu_h}{\LL^\infty(\Omega)} \, \big) \,
\norm{\nabla \mmu_h - \nabla \nnu_h}{\LL^{\infty}(\Omega)}.
\end{align*}
\end{lemma}

\begin{proof}
Throughout the proof, we write 
\begin{align}\label{eq:notations}
P_{ij} := \frac{\mu_i \mu_j}{1+\mu_3} 
\quad \textrm{and} \quad d_k :=  \mu_k - \nu_k \quad
\textrm{for all } i,j \in \{1,2\} \textrm{ and all } k \in \{1,2,3\}.
\end{align}
Recall that the properties of $P_{ij}$ are discussed in Lemma~\ref{lemma:auxiliary1}. Since $\mmu_h, \nnu_h$ are piecewise affine, we get that $|\mmu_h|, |\nnu_h| \leq 1$ on $\Omega$. Moreover, since $1 + \mu_3(z) \geq \gamma > 0$ and $1 + \nu_3(z) \geq \gamma > 0$ for all nodes $z \in \NN_h$, it follows that $1 + \mu_3 \geq \gamma > 0$ and $1 + \nu_3 \geq \gamma > 0$ on $\Omega$. For the proof of~{\rm (i)}, elementary computations show that
\begin{align}\label{eq:master}
\frac{\mu_i \mu_j}{1+\mu_3} - \frac{\nu_i \nu_j}{1+\nu_3} = 
- P_{ij} \, \frac{d_3}{1+\nu_3} + \frac{\mu_j d_i}{1+ \nu_3} + \frac{\nu_i d_j}{1+ \nu_3}.
\end{align}
Together with Lemma~\ref{lemma:auxiliary1}~{\rm (i)}, this proves~{\rm (i)}. For the proof of~{\rm (ii)}, let $k \in \{1,2,3\}$. We differentiate the terms in~\eqref{eq:master} separately and obtain that
\begin{subequations}\label{eq:seperateterms}
\begin{align}
\partial_k \Big( \, P_{ij} \, \frac{d_3}{1+\nu_3} \, \Big) 
& = 
( \partial_k P_{ij} ) \, \frac{d_3}{1+\nu_3} 
+ P_{ij} \, \frac{(\partial_k d_3)}{1+\nu_3} 
- P_{ij} \, \frac{d_3 (\partial_k \nu_3)}{(1+\nu_3)^2} =: T_1 + T_2 + T_3, \label{eq:seperateterms1}
\\
\partial_k \Big( \, \frac{\mu_j d_i}{1+ \nu_3} \, \Big) &= 
\frac{(\partial_k \mu_j) d_i}{1+ \nu_3} + 
\frac{ \mu_j (\partial_k d_i)}{1+ \nu_3} - 
\frac{\mu_j d_i (\partial_k \nu_3)}{(1+ \nu_3)^2} =: T_4 + T_5 + T_6,
\label{eq:seperateterms2} \\
\partial_k \Big( \, \frac{\nu_i d_j}{1+ \nu_3} \, \Big) &= 
\frac{(\partial_k \nu_i) d_j}{1+ \nu_3} + \frac{ \nu_i (\partial_k d_j)}{1+ \nu_3} - 
\frac{\nu_i d_j (\partial_k \nu_3)}{(1+ \nu_3)^2} =: T_7 + T_8 + T_9 .
\label{eq:seperateterms3}
\end{align}
\end{subequations}
With $1 + \mu_3 \geq \gamma > 0$ and $1 + \nu_3 \geq \gamma > 0$, Lemma~\ref{lemma:auxiliary1}~{\rm (i)--(ii)} yields that
\begin{align*}
\sum_{i \in \{ 1,3,4,6,7,9 \}} |T_{i}|
%|T_1| + |T_3| + |T_6| + |T_9| 
& \lesssim \gamma^{-2} \,
\big( \, \norm{\nabla \mmu_h}{\LL^\infty(\Omega)} + \norm{\nabla \nnu_h}{\LL^\infty(\Omega)} \, \big) 
\norm{\mmu_h - \nnu_h}{\LL^{\infty}(\Omega)}, 
\quad \textrm{and} \\
\sum_{i \in \{ 2,5,8 \}} |T_{i}|
 & \lesssim 
\gamma^{-1} \,
\norm{\nabla \mmu_h - \nabla \nnu_h}{\LL^{\infty}(\Omega)}.
\end{align*}
Together with~\eqref{eq:master}, this proves~{\rm (ii)}. For the proof of~{\rm (iii)}, let $\ell, k \in \{1,2,3\}$. We differentiate the terms in~\eqref{eq:seperateterms} separately and exploit that the second derivative of piecewise affine functions is zero. We start from~\eqref{eq:seperateterms1}. Elementwise, the product rule yields that
\begin{subequations}\label{eq:seperateterms4}
\begin{align}
\partial_{\ell} \bigg( \, (\partial_k  \, P_{ij}) \, \frac{d_3}{1+\nu_3} \, \bigg)
& = (\partial_{\ell} \partial_k P_{ij}) \frac{d_3}{1+\nu_3}
+ (\partial_{k} P_{ij}) \, \frac{\partial_\ell d_3}{1+\nu_3} - 
(\partial_{k} P_{ij}) \, \frac{d_3 (\partial_\ell \nu_3)}{(1+\nu_3)^2}, \notag \\
& =: \widetilde{T}_1 + \widetilde{T}_2 + \widetilde{T}_3.
\\
\partial_{\ell} \bigg( \, P_{ij} \, \frac{(\partial_k d_3)}{1+\nu_3} \, \bigg) & = 
(\partial_{\ell} P_{ij}) \, \frac{(\partial_k d_3)}{1+\nu_3} - 
P_{ij} \frac{(\partial_k d_3)(\partial_{\ell} \nu_3)}{(1+\nu_3)^2} =: \widetilde{T}_4 + \widetilde{T}_5, \\
\partial_{\ell} \bigg( \, P_{ij} \, \frac{d_3 (\partial_k \nu_3)}{(1+\nu_3)^2} \, \bigg) & =
(\partial_{\ell} P_{ij}) \frac{d_3 (\partial_k \nu_3)}{(1+\nu_3)^2} + 
P_{ij} \, \frac{(\partial_{\ell} d_3) (\partial_k \nu_3)}{(1+\nu_3)^2} - 
2 P_{ij} \, \frac{d_3 (\partial_k \nu_3)(\partial_{\ell} \nu_3)}{(1+\nu_3)^3} \notag \\
& =: \widetilde{T}_6 + \widetilde{T}_7 + \widetilde{T}_8.
\end{align}
\end{subequations}
Next, we get for the terms from~\eqref{eq:seperateterms2} elementwise that
\begin{subequations}\label{eq:seperateterms5}
\begin{align}
\partial_{\ell} \bigg( \, \frac{(\partial_k \mu_j) d_i}{1+ \nu_3} \, \bigg) 
& =
\frac{(\partial_k \mu_j) (\partial_{\ell} d_i)}{1+ \nu_3}
- \frac{(\partial_k \mu_j) d_i (\partial_{\ell} \nu_3)}{(1+ \nu_3)^2} =: \widetilde{T}_9 + \widetilde{T}_{10}, \\
\partial_{\ell} \bigg( \, \frac{ \mu_j (\partial_k d_i)}{1+ \nu_3} \, \bigg) & = 
\frac{ (\partial_{\ell} \mu_j) (\partial_k d_i)}{1+ \nu_3} - 
\frac{ \mu_j (\partial_k d_i) (\partial_\ell \nu_3)}{(1+ \nu_3)^2} =: \widetilde{T}_{11} + \widetilde{T}_{12}, \\
\partial_{\ell} \bigg( \, \frac{\mu_j d_i (\partial_k \nu_3)}{(1+ \nu_3)^2} \, \bigg) & = 
\frac{(\partial_{\ell} \mu_j) d_i (\partial_k \nu_3)}{(1+ \nu_3)^2} + 
\frac{ \mu_j (\partial_{\ell} d_i) (\partial_k \nu_3)}{(1+ \nu_3)^2} - 
2 \frac{ \mu_j d_i (\partial_k \nu_3) (\partial_{\ell} \nu_3)}{(1+ \nu_3)^3} \notag \\
& =: \widetilde{T}_{13} + \widetilde{T}_{14} + \widetilde{T}_{15}.
\end{align}
\end{subequations}
Lemma~\ref{lemma:auxiliary1}, $1 \leq 2 / \gamma$ and $1 + \mu_3 \geq \gamma > 0$ as well as $1 + \nu_3 \geq \gamma > 0$ yield that
\begin{align*}
\sum_{i \in \{ 1,3,6,8,10,12,13,15 \}} |\widetilde{T}_{i}|
& \lesssim \gamma^{-3} \,
\big( \, \norm{\nabla \mmu_h}{\LL^\infty(\Omega)} + \norm{\nabla \nnu_h}{\LL^\infty(\Omega)} \, \big)^2 \,
\norm{\mmu_h - \nnu_h}{\LL^{\infty}(\Omega)}, 
\quad \textrm{and}
\\
\sum_{i \in \{ 2,4,5,7,9,11,14 \}} |\widetilde{T}_{i}| 
& \lesssim \gamma^{-2} \,
\big( \, \norm{\nabla \mmu_h}{\LL^\infty(\Omega)} + \norm{\nabla \nnu_h}{\LL^\infty(\Omega)} \, \big) \,
\norm{\nabla \mmu_h - \nabla \nnu_h}{\LL^{\infty}(\Omega)}.
\end{align*}
%Note that $\mu_i$ instead of $\nu_i$ and $d_i$ instead of $d_j$ in~\eqref{eq:seperateterms2} results in~\eqref{eq:seperateterms3}. 
Note that the terms in~\eqref{eq:seperateterms3} are obtained if in~\eqref{eq:seperateterms2} we replace  
$\mu_i$ with $\nu_i$ and $d_i$ with $d_j$. Hence, we can apply the same arguments as in~\eqref{eq:seperateterms5}. This proves~{\rm (iii)}.
\end{proof}

\begin{lemma}[{{\cite[Lemma 3.1]{ABV14}}}]\label{lemma:abstract_matrix_setting}
Let $\matrixB \in \R^{2N \times 2N}$ be a positive definite matrix and $\matrixB_0 \in \R^{2N \times 2N}$ be a symmetric positive definite matrix, which satisfy for $c_1,c_2 > 0$ that
\begin{align*}
\vectorX \cdot \matrixB \vectorX & \, \geq \, 
c_1 \, \vectorX \cdot \matrixB_0 \vectorX \quad \textrm{for all } \vectorX \in \R^{2N}
\quad \textrm{and } \\
\vectorX \cdot \matrixB \vectorY & \, \leq \, 
c_2 \, \big( \, \vectorX \cdot \matrixB_0 \vectorX \, \big)^{1/2} 
\, \big( \, \vectorY \cdot \matrixB_0 \vectorY \, \big)^{1/2} \quad \textrm{for all } \vectorX, \vectorY \in \R^{2N}.
\end{align*}
Then, it holds that
\begin{align*}
\vectorX \cdot \matrixB^{-1} \vectorX & \, \geq \, 
\frac{c_1}{c_2^2} \, \vectorX \cdot \matrixB_0^{-1} \vectorX \quad \textrm{for all } \vectorX \in \R^{2N}
\quad \textrm{and } \\
\vectorX \cdot \matrixB^{-1} \vectorY & \, \leq \, 
{c_1}^{-1} \, \big( \, \vectorX \cdot \matrixB_0^{-1} \vectorX \, \big)^{1/2} 
\, \big( \, \vectorY \cdot \matrixB_0^{-1} \vectorY \, \big)^{1/2} \quad \textrm{for all } \vectorX, \vectorY \in \R^{2N}.
\tag*{\qed}
\end{align*}
\end{lemma}

%%%%%%%%%%%%%%%%%%%%%%%%%%%%%%%%%%%%%%%%%%%%%%%%%%%%%%%%%%%%%%%%%%%%
%%%%%%%%%%%%%%%%%%%%%%%%%%%%%%%%%%%%%%%%%%%%%%%%%%%%%%%%%%%%%%%%%%%%
%%%%%%%%%%%%%%%%%%%%%%%%%%%%%%%%%%%%%%%%%%%%%%%%%%%%%%%%%%%%%%%%%%%%

\bibliographystyle{alpha}
\bibliography{ref}
%%%%%%%%%%%%%%%%%%%%
%\vspace*{-2mm}
\end{document}